\documentclass[10pt,twoside, a4paper, english, reqno]{amsart}
\usepackage[dvips]{epsfig}
\usepackage{amscd}
\usepackage{stmaryrd}
\usepackage{amssymb}
\usepackage{amsthm}
\usepackage{amsmath}
\usepackage{graphicx,enumitem,dsfont}
\usepackage{latexsym}

\usepackage{upref}
\usepackage[colorlinks]{hyperref}
\usepackage{color}
\usepackage{subcaption}
\usepackage[usenames,dvipsnames]{xcolor}
\theoremstyle{plain}
\newtheorem{thm}{Theorem}[section]
\theoremstyle{plain}
\newtheorem{lem}[thm]{Lemma}
\newtheorem{prop}[thm]{Proposition}
\newtheorem{cor}[thm]{Corollary}

\theoremstyle{definition}
\newtheorem{defi}{Definition}[section]
\newtheorem{rem}{Remark}[section]

\newtheorem*{maintheorem*}{Main Theorem}

\newenvironment{Assumptions}
{
\setcounter{enumi}{0}

\begin{enumerate}}
{\end{enumerate} }

\newenvironment{Assumptions2}
{
\setcounter{enumi}{0}

\begin{enumerate}}
{\end{enumerate} }

\newcommand{\R}{\ensuremath{\mathbb{R}}}

\newcommand{\goto}{\ensuremath{\rightarrow}}
\newcommand{\grad}{\ensuremath{\nabla}}
\newcommand{\eps}{\ensuremath{\varepsilon}}

\numberwithin{equation}{section} \allowdisplaybreaks

\title[Large Deviations Principle]
{Stochastic doubly nonlinear PDE: Large Deviation Principles and existence of Invariant measure}

\date{}

\author[A. K. Majee]{Ananta K. Majee}
\address[Ananta K. Majee]{\newline
Department of Mathematics,
Indian Institute of Technology Delhi,
Hauz Khas, New Delhi, 110016, India. }
\email[]{majee@maths.iitd.ac.in}


\thanks{}

\numberwithin{equation}{section} \allowdisplaybreaks
\begin{document}
\begin{abstract}
    In this paper, we establish large deviation principle for the strong solution of a doubly nonlinear PDE driven by small multiplicative Brownian noise. Motononicity arguments and  the weak convergence approach  have been exploited in the proof. Moreover, by using certain a-priori estimates and sequentially weakly Feller property of the associated Markov semigroup, we show existence of invariant probability measure for the strong solution of the underlying problem.

\end{abstract}
\maketitle
\section {Introduction}

Doubly nonlinear PDE appears in modeling the flows of incompressible turbulent fluids through porous media and the turbulent gas flowing in a pipe of uniform cross sectional area \cite{Diaz1994}, the phase transitions
in thermodynamics and also in  the Steafan-type problems, see \cite{Rossi2013,Visintin1997,Visintin2008} and references therein. It also arises in ferromagnetism and plasticity theory; cf.~ \cite{Visintin2008-elctro}.  Due to its wide range of applications in physical contexts as well as its technical novelties, well-posedness theory for doubly nonlinear PDE with nonlinear stochastic forcing is more subtle.  We are interested in the theory of large deviation principle, which concerns the asymptotic behaviour of remote tails of some sequence of probability distribution, for the solution of the following doubly nonlinear PDE driven by multiplicative Brownian noise: 
 \begin{equation}\label{eq:doubly-nonlinear}
 \begin{aligned}
  {\rm d} {\tt B}(u) -\mbox{div}_x {\tt A}(\grad u)\,{\rm d}t &=\sigma(u)\,dW(t)\quad \text{in}\,\,\Omega \times D_T, \\
  u&=0 \quad \text{on}\,\, \Omega \times \partial D_T\,, \\
  u(0,\cdot)&=u_0(\cdot)  \quad \text{in}\,\,\Omega \times  D\,,
  \end{aligned}
 \end{equation}
 where $D_T=(0,T)\times D,~\partial D_T=(0,T)\times \partial D$ for fixed $T>0$ and $W(t)$  is a one-dimensional Brownian motion defined on a given filtered probability space $(\Omega, \mathbb{P}, \mathcal{F}, \{ \mathcal{F}_t\}_{t \geq 0})$  satisfying the usual hypothesis. In \eqref{eq:doubly-nonlinear}, $\sigma: \R \mapsto \R$
 is a given noise coefficient signifying the multiplicative nature of the noise. The operator ${\tt B}$ is
 given by ${\tt B}(u)(x):={\tt b}(u(x))$ for some differentiable function ${\tt b}: \R \goto \R$ with ${\tt b}(0)=0$, and the nonlinear drift operator ${\tt A}$ is the Nemyckii type operator, defined as ${\tt A}(L)(x):={\tt a}(x, L(x))$ for some  Carath\'{e}odory function ${\tt a}:D\times \R^d \goto \R^d$ with $L: D\goto \R^d$ measurable function.
\vspace{0.1cm}

 Nonlinearity of drift and  diffusion terms in \eqref{eq:doubly-nonlinear} prevents us to define a semi-group solution. Instead, one may use the classical monotonicity methods \cite{Krylov2007,Liu2015,pard} to study wellposedness theory of  \eqref{eq:doubly-nonlinear}; see also \cite{Lions1969} for PDEs.
In the articles \cite{Majee2020,Vallet2019}, the authors have considered stochastic perturbation of a evolution $p$-Laplacian equation~($p>2$) with nonlinear source, which
can be viewed as of the form \eqref{eq:doubly-nonlinear} with ${\tt B}={\tt Id}$ and ${\tt A}=-\Delta_p$, the $p$- Laplacian operator~(monotone), but with extra first order term $-{\rm div}_x \vec{f}$ for
some Lipschitz continuous function $\vec{f}:\R \mapsto \R^d$.   By using  pseudo monotonicity methods and adapting the theorems of Skorokhod and Prokhorov, the authors, in  \cite{Vallet2019}, have shown the existence of a martingale solution.  Moreover,  Gy\"{o}ngy-Krylov characterization \cite{Krylov1996} of convergence in probability and path-wise uniqueness is used to show the  wellposedness of strong solution.  Based on pseudo monotonicity methods along with the Jakubowski-Skorokhod theorem \cite{Jakubowski1998} in a non-metric space, the author in \cite{Majee2020} prove existence of a martingale solution. The author also established existence of a weak optimal solution to the associated initial value control problem with L\'{e}vy noise. In a recent article  \cite{Wittbold2019}, the authors have established well-posedness theory for the strong solution of  \eqref{eq:doubly-nonlinear} by employing  semi-implicit time discretization scheme, monotonicity arguments, the theorems of Skorokhod and Prokhorov together with  Gy\"{o}ngy-Krylov characterization  of convergence in probability and path-wise uniqueness of martingale solution (which was obtained by the standard $L^1$-method).  Very recently,  in \cite{Majee2023} the author  considered an initial value control problem for a doubly nonlinear
PDE perturbed by multiplicative L\'{e}vy noise and proved the existence of a weak martingale
solution and its path-wise uniqueness. The proof of a weak martingale solution combines a semi-
implicit time discretization scheme and monotonicity method. The author derived a-priori estimates of the scheme and showed the tightness of the probability measures generated by the approximate solution of the scheme via Aldou's condition and then applied the Jakubowski-Skorokhod
theorem in a non-metric space to obtain the existence of a weak martingale solution. Moreover, using the variational approach, the author proved the existence of a weak optimal solution. The ingredients for the proof used a minimizing weak admissible solution and the Skorokhod’s theorem.

\vspace{0.2cm}

Theory of large deviation principle~(LDP in short)  plays an important role in the field of probability and statistics; see  e.g., \cite{varadha1,varadhan, freidlin, stroock, dembo, Ellis-1985, Stroock-1989,kurtz} and reference therein. During the last decade there has been growing interested in the study of LDP for stochastic partial differential equations~(SPDEs) and new results are emerging faster than ever before. A number of authors have contributed since then, and we
mention only few  see e.g., \cite{liu, dong1, dong2, matous, rock, ren} and references therein, based on either weak convergence method \cite{dupuis, budhi} or 
exponential equivalence method.  In \cite{rock}, the authors established LDP for the solution of the stochastic tamed Navier-Stokes equations via 
exponentially equivalent technique. In \cite{ren}, Ren and Zhang have studied Freidlin-Wentzell's large deviation for stochastic evolution equation in the evolution triple. LDP  for stochastic evolution equation involving strongly monotone drift  was studied by Liu in \cite{liu}. The study of LDP was carried out for 3D stochastic primitive equation by Dong, Zhai and Zhang in \cite{dong1}, stochastic scalar conservation laws by Dong, Wu and Zhang in \cite{dong2}, Obstacle problems of quasi-linear SPDEs by Matoussi, Sabbagh and Zhang in \cite{matous} via weak convergence method. Very recently, the authors in \cite{Kavin-Majee-2022}
established large deviation principle for the strong solution of  evolutionary nonlinear perturbation ${\rm div}_xf(u)$ of  $p$-Laplace equation driven by small multiplicative Brownian noise. By employing semi-discrete time discretization together with a-priori estimation on some appropriate fractional Sobolev space, the authors able to establish LDP via weak convergence method. Also in the case of bounded diffusion coefficient, based on Girsanov transformation along with $L^1$-contraction approach, the authors derived  the quadratic transportation cost inequality---an important area of research as it has close connections with well-known functional inequalities, e.g., poincar\'{e} inequalities, logarithmic Sobolev inequalities--- for the strong solution of the underlying problem.  
 \vspace{0.15cm}

In this paper, we wish to study the LDP and  to show existence of invariant probability measure for the strong solution of \eqref{eq:doubly-nonlinear}. Due to  the doubly nonlinear nature of the underlying problem, approach of  \cite{liu}, to show LDP, is not applicable. Instead, we follow the  similar line of approach invoked in  \cite{Kavin-Majee-2022} to settle LDP for the solution of  \eqref{eq:doubly-nonlinear} on the Polish space $\mathcal{Z} = C([0,T];L^{2}(D))$.  The aim of this paper is twofold:
\begin{itemize}
 \item[i)] Firstly, we  prove that the solution $u^\eps$ of small multiplicative noise driven SPDE \eqref{eq:ldp} satisfies the LDP on the solution space $\mathcal{Z} = C([0,T];L^{2}(D))$ via the weak convergence method. 
 \begin{itemize}
 \item[{\rm a)}] We first prove the well-posedness result of the corresponding skeleton equation~(cf.~\eqref{eq:skeleton}).  We construct  an approximate solutions $\{ \tilde{\tt B}_\tau\}$~(cf.~Definition \ref{defi:approximate-skeleton}) via semi-implicit time discretization, and derive its a-priori bounds which plays a key role to show the strong convergence of $\{\tilde{\tt B}_\tau\}$ and hence of 
  $\{{\tt B}(u_\tau)\}$. Moreover, a Minty type monotonicity arguments yields the existence of weak solution of the Skeleton equation. Furthermore, we use smooth approximation of absolute value function and chain rule together with the monotonicity argument and pass to the limit as approximation parameter tends to zero to show the path-wise uniqueness of weak solutions via standard $L^1$-method. 
 \item[{\rm b)}] Due to Budhiraja and Dupuis in \cite[Theorem 4.4]{budhi}, to prove LDP for $u^\eps$, it suffices to validate the conditions \ref{C1} and \ref{C2}. For the validation of condition \ref{C2}, thanks to weakly compactness of the set $S_M$~(cf.~\eqref{defi:sets-for-ldp}), it is enough to show strong convergence of $u_n$ to $u_h$ where $u_n$ and $u_h$ are the solution of \eqref{eq:skeleton} corresponding to $h_n$ and $h$ respectively with $h_n \goto h$ weakly in $L^2(0,T; \R)$. 
 By using uniform bounds for ${\tt B}(u_n)$ (see \eqref{eq:ub}, \eqref{esti:frac-Sov-sequence} ), boundedness property of $h_n$ along with  Arzel$\acute{a}$–Ascoli theorem, we show strong convergence of $B(u_n)$ and hence strong convergence  of $u_n$ to some function $\bar{u}_*$ in some appropriate space ---which then yields thanks to monotonicity arguments and uniqueness of solution of the Skeleton equation that $u_n\goto \bar{u}_*=u_h$ in $\mathcal{Z}$. 
 \item[{\rm c)}] Thanks to Girsanov's theorem and uniqueness result for the equation \eqref{eq:ldp}, to validate condition \ref{C1}, we basically show convergence of $v^\eps$ to $u_h$ in distribution on $\mathcal{Z}$ where $v^\eps$ is the unique solution of \eqref{eq:epsilon-main}.  We use similar arguments as mentioned in 
 ${\rm b)}$ but due to its stochastic setup, more regularity estimate of $v^\eps$ is needed ~(cf.~Lemma \ref{lem:est2.2}). We first derive necessary uniform estimates
 (Lemmas \ref{lem:est2.1}-\ref{lem:cpt2}) which is used to show tightness of the sequence of laws of $\{{\tt B}(v^\eps)\}$ and then use Prokhorov compactness theorem and the modified version of Skorokhod representation theorem to get a.s. convergence of ${\tt B}(\overline{v}^\eps)$ to ${\tt B}(\overline{v}^*)$ with respect to a new probability space $(\overline{\Omega},\overline{\mathcal{F}},\overline{\mathbb{P}})$. Moreover, using chain rule on ${\tt B}(\overline{v}^*)$ and Ito-formula on $\|{\tt B}(\overline{v}^\eps)$ and a Minty type monotonicity argument together with higher regularity of $v^\eps$ (and hence $\overline{v}^\eps$)---which also required for passing the limit in the controlled drift term in particular to get \eqref{conv:weak-c1-drift-1}---we show that $\overline{v}^*$ is indeed a solution of Skeleton equation \eqref{eq:skeleton}. Thanks to uniqueness of Skeleton equation and equality of laws of $\overline{v}^\eps$ and $v^\eps$, we show convergence of the expected value of $\| {\tt B}(\overline{v}^\eps)-{\tt B}(\overline{v}^*)\|_{\mathcal{Z}}$ as $\eps$ approaches to zero--which then implies the assertion \ref{C1} via Markov inequality. 
 \end{itemize}

\item[ii)] Secondly, we prove the existence of invariant measure for the strong solution of \eqref{eq:doubly-nonlinear} by invoking the technique of Maslowski and Seidler \cite{Seidler1999} working with weak topology. We show two conditions of Theorem \ref{thm:M-S}, ${\rm i^\prime)}$ boundedness in probability and ${\rm ii^\prime)}$ sequentially weakly Feller property, are satisfied by the associated Markov semigroup $(P_t)_{t\ge 0}$, defined by \eqref{defi:semi-group}, to establish the existence of invariant measure. 
\end{itemize}
\vspace{0.2cm}

The rest of the paper is organized as follows. We discuss the assumption, basic definition, framework and main results in Section \ref{sec:preliminaries}. In Section \ref{sec:sk}, we show well-posedness result of Skeleton equation.  Section \ref{sec:LDP} is devoted to prove the LDP for the strong solution of \eqref{eq:doubly-nonlinear} . In the final section \ref{sec:invariant-measure}, we prove existence of invariant measure for associated semigroup for the strong solution of \eqref{eq:doubly-nonlinear}. 

\section{Preliminaries And Technical Framework} \label{sec:preliminaries}
 By $C$, $K$, etc., we mean various constants which may change from line to line. Here and in the sequel, for  any $r\in \mathbb{N}$, we denote by $(L^r, \|\cdot\|_{L^r})$ the standard spaces
 of $r$th order integrable functions on $D$.  For $p\in \mathbb{N}$, we write $(W_0^{1,p}, \|\cdot\|_{W_0^{1,p}})$ for the standard Sobolev spaces on $D$, and $W^{-1,p^\prime}$ for the 
 dual space of  $W_0^{1,p}$, where $p^\prime$ denotes the convex conjugate of $p$. Moreover, we use the notation $\big\langle \cdot,\cdot\big\rangle$ for 
 the pairing between $W_0^{1,p}$ and $W^{-1,p^\prime}$, and $ x\wedge y$ by $ \min\{x,y\}$ for any $x,y\in \R$. We denote by $(\mathbb{Y},w)$ the topological space $\mathbb{Y}$ equipped with the weak topology.
 \vspace{.1cm}
 
 We now define notion of weak and strong solution of \eqref{eq:doubly-nonlinear}.
 
 \begin{defi}[Weak solution]\label{defi:weak-solun}
  Let the deterministic initial datum  $u_0 \in L^2$ be given.
  We say that  a $6$-tuple $\bar{\pi}=\big( \bar{\Omega}, \bar{\mathcal{F}}, \bar{\mathbb{P}}, \{\bar{\mathcal{F}}_t\},  
 \bar{W}, \bar{u}\big)$ is a weak solution of \eqref{eq:doubly-nonlinear}, provided
 \begin{itemize}
  \item [(i)] $(\bar{\Omega}, \bar{\mathcal{F}},\bar{\mathbb{P}}, \{\bar{\mathcal{F}}_t\}_{t\ge 0} )$ is a complete stochastic basis, and 
  $\bar{W}$ is a one-dimensional Brownian motion defined on this stochastic basis.
 
  \item[(ii)] $\bar{u}: \bar{\Omega}\times [0,T]\rightarrow L^2$ is an $\{\bar{\mathcal{F}}_t\}$-predictable stochastic process such that 
  \begin{itemize}
   \item [(a)] $\bar{u}\in L^p\big(0,T; W_0^{1,p}\big)\cap L^\infty\big(0,T; L^2\big)$, and $\bar{\mathbb{P}}$-a.s., $\bar{u}(0)=u_0$ in $L^2$, 
   \item[(b)]  $\bar{\mathbb{P}}$-a.s., ${\tt B}(\bar{u}) \in L^\infty(0,T;L^2)\cap  \mathcal{C}([0,T]; W^{-1,p^\prime}) \subset  \mathcal{C}([0,T]; L^2_w)$,
   \item[(d)] $\bar{\mathbb{P}}$-a.s., and for all $t\in [0,T]$, 
   \begin{align*}
   {\tt B}(\bar{u}(t))={\tt B}( u_0 ) + \int_{0}^t \mbox{div}_x {\tt A}(\nabla \bar{u})\,{\rm d}s + 
   \int_0^t\sigma(\bar{u}(s))\,d\bar{W}(s)  \quad \text{in}~~ W^{-1,p^\prime}\,. 
  \end{align*}
  \end{itemize}
 \end{itemize}
 \end{defi}
In view of  the property ${\rm (b)}$, $\bar{\mathbb{P}}$-a.s., $B(\bar{u}(t))\in L^2$ and hence $\bar{u}(t)= B^{-1} B(\bar{u}(t))\in L^2$ for all $t\in [0,T]$.
\begin{defi}[Strong solution]
Let  $W(t)$  be a one-dimensional Brownian motion defined on a given filtered probability space $(\Omega, \mathbb{P}, \mathcal{F}, \{ \mathcal{F}_t\}_{t \geq 0})$  satisfying the usual hypothesis. We say that a predictable process $u: \Omega \times [0,T]\goto L^2$ is a strong solution of  \eqref{eq:doubly-nonlinear} if and only if 
conditions ${\rm (a)}$-${\rm (c)}$ of Definition \ref{defi:weak-solun} hold true with respect to the given stochastic basis. 
\end{defi}

The aim of this paper is to establish LDP and existence of invariant measure for the strong solution of \eqref{eq:doubly-nonlinear}, and we will do so under the following assumptions: 
\begin{Assumptions}
\item \label{A2}  ${\tt a}: D \times \R^d \goto \R^d$ is a Carath\'{e}odory function satisfying the following:
\begin{itemize}
\item [(i)] ${\tt a}$ is monotone in the second argument, i.e., $$
\big({\tt a}(x,\xi_1)- {\tt a}(x, \xi_2) \big)\cdot (\xi_1- \xi_2) \ge 0 \quad \text{
for a.e. $x\in D$ and for all $\xi_1, \xi_2 \in \R^d$}.$$
\item [(ii)] ${\tt a}$ is coercive. Namely, there exist $C_1>0$ and a function $K_1 \in L^1$ such that 
$$ {\tt a} (x, \xi) \cdot \xi \ge C_1 |\xi|^p - K_1(x) \quad \text{  for $\xi \in \R^d$ and  a.e. $x\in D$ with $p>2$}. $$
\item [(iii)] ${\tt a}$ is bounded. In particular, there exist $C_2>0$ and a function $K_2 \in L^{p^\prime}$ such that 
\begin{align*}
|{\tt a}(x, \xi)| \le C_2 |\xi|^{p-1} + K_2(x) \quad \text{ for a.e. $x\in D$ and for all $\xi \in \R^d$}\,.
\end{align*}
\end{itemize}
 
\item \label{A3} ${\tt b}: \R \goto \R$ is a differentiable function such that ${\tt b}(0)=0$. Moreover, there exist constants $C_3, C_4, C_{{\tt b}^\prime} >0$ such that 
\begin{align*}
 & |{\tt b}^\prime(\zeta_1)-{\tt b}^\prime(\zeta_2)| \le C_{{\tt b}^\prime}|\zeta_1 - \zeta_2| \quad \forall\, \zeta_1, \zeta_2 \in \R\,, \\
& C_3 \le {\tt b}^\prime(\zeta) \le C_4  \quad \forall\, \zeta \in \R\,. 
\end{align*}
\item \label{A4} $\sigma(0)=0$ and $\sigma$ is Lipschitz i.e., there exists a constant $c_\sigma>0$ such that
 \begin{align*}
  | \sigma(\zeta_1)-\sigma(\zeta_2)|  & \leq  c_\sigma |\zeta_1-\zeta_2| \quad \forall~ \zeta_1, \zeta_2 \in \R\,. 
   \end{align*}
\end{Assumptions}
In view of the assumptions \ref{A2}-\ref{A3}, the operators ${\tt A}$ and ${\tt B}$  satisfy the following conditions; see \cite[Remark $2.1.1 \&$ Subsection $2.3$ ]{Wittbold2019} and \cite[Remark $2.1$]{Majee2023}.

\begin{itemize}
 \item [(I)] The operator ${\tt A}: (L^p)^d\goto (L^{p^\prime})^d$, defined as 
${\tt A}(v)(x)= {\tt a}(x, v(x))$ for any $v\in (L^p)^d$, is continuous, monotone, coercive and bounded. In particular, for all $v, v_1, v_2 \in (L^p)^d$
\begin{align}
& \big( {\tt A}v_1 - {\tt A} v_2, v_1 - v_2 \big)_{ \big((L^{p^\prime})^d, (L^p)^d\big)} \ge 0\,, \label{A:monotonocity} \\
& \big( {\tt A}v , v \big)_{\big((L^{p^\prime})^d, (L^p)^d\big)} \ge C_1 \|v\|_{(L^p)^d}^p - \|K_1\|_{L^1} \,, \label{A:coercivity} \\
& \|{\tt A} v\|_{(L^{p^\prime})^d} \le C_2 \|v\|_{(L^p)^d}^{p-1} + \|K_2\|_{L^{p^\prime}}\,.  \label{A:boundedness}
\end{align}
\item [(II)] ${\tt b}$ is strictly monotone and coercive. The inverse function ${\tt b}^{-1}$ is differentiable and satisfies
$$ \frac{1}{C_4} \le ({\tt b}^{-1})^\prime \le \frac{1}{C_3}.$$
The operator ${\tt B}: L^2 \goto L^2$, defines as ${\tt B} (w)(x)= {\tt b}(w(x))$ for any $w\in L^2$, is strongly monotone and Lipschitz continuous. Specifically, 
\begin{align}
 \|{\tt B}w_1 - {\tt B} w_2\|_{L^2}  & \le C_4 \|w_1- w_2\|_{L^2}\,. \label{esti:B-Lipschitz}
 \end{align}
 For any $w\in W_0^{1,p}$, ${\tt B}(w)\in W_0^{1,p}$ and satisfies 
\begin{align}
\|\nabla {\tt B}(w)\|_{L^p}=\|{\tt b}^\prime(w) \nabla w\|_{L^p} \le C_4 \|\nabla w\|_{L^p}\,. \label{esti:operator-B-grad}
\end{align} 
\item[(III)]Thanks to Zarantonello's theorem \cite[Theorem $25. B$]{zeidler1990}, the inverse operator ${\tt B}^{-1}: L^2 \goto L^2$ exists and  Lipschitz continuous. In particular, for any $u, w\in L^2$ and a.e. $x\in D$,  one has
\begin{align}
 & \|w\|_{L^2}  \le \frac{1}{C_3} \|{\tt B} w\|_{L^2}\,,\label{B:reverse-L2} \\
  & \big|u(x)-w(x)| \le \frac{1}{C_3}\big| {\tt B}(u)(x)-{\tt B}(w)(x)|\,.\label{B:reverse-poinwise}
\end{align}
\end{itemize}

Under the assumptions \ref{A2}-\ref{A4} and for deterministic $u_0\in L^2$, the equation \eqref{eq:doubly-nonlinear} has a unique strong solution $u\in L^2(\Omega; C([0,T];W^{-1,p'})\cap L^{p}([0,T];W_{0}^{1,p}))$ satisfying the following uniform bounds; cf.~ \cite[Theorem $3.5$]{Wittbold2019}:
\begin{equation}\label{esti:weak-solun}
\begin{aligned}
 \mathbb{E}\Big[\sup_{0\le t\le T} \|{\tt B}(u(t))\|_{L^2}^2 + \int_0^T \| {\tt B}(u(t))\|_{W_0^{1,p}}^p\,{\rm d}t \Big] \le C\,,  \\
  \mathbb{E}\Big[ \int_{D_T} \Big( |{\tt A}(\nabla u(t))|^{p^\prime} + |\nabla u(t)|^p\Big)\,{\rm d}x\,{\rm d}t\Big] \le C\,.
\end{aligned}
\end{equation}
 We can write down the strong solution of \eqref{eq:doubly-nonlinear} in the form
\begin{align*}
{\tt B}(u(t))= {\tt B}(u_0) + \int_0^t g(s)\,ds + M(t)\,.
\end{align*}
In view of the estimate \eqref{esti:weak-solun}, we see that
\begin{itemize}
\item[i)] $u\in L^p(\Omega\times(0,T);W_0^{1,p})$,
\item[ii)] $g\in L^{p^\prime}(\Omega\times (0,T); W^{-1, p^\prime})$,
\item[iii)] $M\in L^2(\Omega\times [0,T];L^2)$.
\end{itemize}
Hence by \cite[Lemma $4.1$]{Pardoux1974-75}, we conclude that $u\in L^2(\Omega; C([0,T];L^2))$. Thus,  the path-wise unique solution $u$ of  \eqref{eq:doubly-nonlinear}  has improved regularity property namely $u\in L^2(\Omega;  C([0,T];L^{2})\cap L^{p}([0,T];W_{0}^{1,p}))$. Moreover, for $\mathbb{P}$-a.s $\omega \in \Omega$, the trajectory $u(\cdot, \omega)$ is almost everywhere equal to a continuous $L^2$-valued function defined on $[0,\infty)$. 
 

\subsection{Large deviation principle} 
Let $\mathcal{Z}$ be a Polish space and $\{X^\eps\}_{\eps > 0}$  be a sequence of $\mathcal{Z}$-valued random variables defined on a given probability space $(\Omega, \mathcal{F}, \mathbb{P})$. 

 \begin{defi}[Rate Function]
      A function $\mathcal{I} : \mathcal{Z}$ $\rightarrow$ $[0,\infty]$ is said to be a rate function if $\mathcal{I}$ is a lower semi-continuous function. It is said to be a good rate function if the level set $$\{x\in\mathcal{Z} : ~\mathcal{I}(x) \leq M \}$$ is compact for each $M<\infty$.
\end{defi}
\begin{defi} [Large deviation principle]
The sequence  $\{X^\eps\}_{\eps> 0}$  is said to satisfy the LDP with rate function $\mathcal{I}$ if for each $ A \in \mathcal{B}(\mathcal{Z})$, the Borel $\sigma$-algebra of $\mathcal{Z}$, the following inequality holds:
\begin{equation*}
    \begin{aligned}
        \underset{x\in A^0}{-\inf} \,\mathcal{I}(x) \leq \underset{\eps \rightarrow 0}{\liminf} \eps^{2}\log\mathbb{P}(X^\eps \in A) & \leq    \underset{\eps\rightarrow 0}{\limsup} \eps^{2} \log \mathbb{P}(X^\eps \in A)
        \leq  \underset{x\in \bar{A}} {-\inf}\,\mathcal{I}(x)\,,
    \end{aligned}
\end{equation*}
where $A^0$ and $\bar{A}$ are interior and closure of $A$ in $\mathcal{Z}$ respectively. 
\end{defi}
Consider the following sets:
\begin{equation}\label{defi:sets-for-ldp}
\begin{aligned} 
      &\mathcal{A} := \Big\{ \phi :  \text{ $\phi$  is a $\{\mathcal{F}_t\}$-predictable process with } 
       \int_0^T|\phi(s)|_{\mathbb{R}}^2ds <  \infty, ~~ \mathbb{P}\text{-a.s}.\Big\}, \\
      &S_M := \Big\{h \in L^2([0,T]; \mathbb{R}) : ~\int_0^T|h(s)|_{\mathbb{R}}^2\,ds\leq M\Big\}, \\
    &\mathcal{A}_M :=\Big\{\phi \in \mathcal{A} :~ \phi(\omega) \in S_M,  \mathbb{P} \text{-a.s.}\Big\}. 
\end{aligned}
\end{equation}
For each  $\eps > 0$,  suppose there exist measurable maps $\mathcal{G^{\eps}}$ : C([0,T], $\mathbb{R}$) $\rightarrow$ $\mathcal{Z}$ and $\mathcal{G}^{0}$ : C([0,T], $\mathbb{R}$) $\rightarrow$ $\mathcal{Z}$ such that $X^\eps=\mathcal{G}^{\eps}(W)$ for a given one dimensional Brownian motion on the stochastic basis
$(\Omega, \mathcal{F}, \mathbb{P}, \{\mathcal{F}_t\}_{t\ge 0})$.  Then, due to Budhiraja and Dupuis \cite[Theorem 4.4]{budhi},  we state a sufficient conditions for $X^\eps$ to satisfy LDP on $\mathcal{Z}$.
\begin{thm}\label{thm:ldp-general-sufficient}
$X^\eps$ satisfies the LDP on $\mathcal{Z}$ with the good rate function $\mathcal{I}$ given by
 \begin{align} \label{eq:Imain}
  \mathcal{I}(f)= \begin{cases}
\displaystyle  \underset{\big\{h \in L^2([0,T]; \mathbb{R}) :~ f =  \mathcal{G}^{0} \big(\int_0^{\cdot}h(s)ds\big)\big\}} \inf  \bigg \{ \frac{1}{2} \int_0^T|h(s)|_{\mathbb{R}}^2\,ds \bigg \}, \quad \forall f \in \mathcal{Z}\,,
 \\
\displaystyle   \infty, \quad  \text{if} \quad \bigg \{h \in L^2([0,T]; \mathbb{R}) : f = \mathcal{G}^0 \bigg( \int_0^{.}h(s)ds \bigg) \bigg \} = \emptyset\,,
\end{cases}
 \end{align}
 provided  $\mathcal{G}^\eps$ and $\mathcal{G}^0$ satisfies the following sufficient conditions:
\begin{Assumptions2}
\item \label{C1}  Let \{$h^\eps:~\eps> 0\} \subset \mathcal{A}_M$, for some $M < \infty$, be a family of process converges to $h$ in distribution as $S_{M}$-valued random elements.  Then  
\begin{align*}
    \mathcal{G}^{\eps} \bigg( W(\cdot) + \frac{1}{{\sqrt{\eps}}} \int_0^{\cdot} h^{\eps}(s)\,ds\bigg) \rightarrow  \mathcal{G}^0 \bigg( \int_0^{\cdot}h(s) \, ds \bigg)
\end{align*}
in distribution as $\eps \rightarrow 0$.
\item \label{C2} For every M $<$ $\infty$, the set
\begin{align*} 
    K_{M} = \left \{\mathcal{G}^0\bigg(\int_0^{.}h(s)\,ds \bigg) : h \in S_{M} \  \right \}
\end{align*}
is a compact subset of $\mathcal{Z}$.
\end{Assumptions2}
\end{thm}
 For fixed $\eps>0$,  consider the SPDE driven by small multiplicative Brownian noise:
 \begin{equation}
 \label{eq:ldp}
      \begin{aligned}
            d {\tt B}(u^{\eps})  - {\rm div}_x {\tt A}(\nabla u^{\eps}) \, dt &= \sqrt{\eps}\, \sigma(u^{\eps})\, dW(t) \quad  &in& \ \ \Omega \times D_T, \\
            u^{\eps}(0,.) &= u_0(.) \in L^2 &in& \ \ \Omega \times D\,.
        \end{aligned}
\end{equation}

Thanks to  \cite[Theorem $3.5$]{Wittbold2019} and \cite[Lemma $4.1$]{Pardoux1974-75}, under the assumptions \ref{A2}-\ref{A4}, equation \eqref{eq:ldp} has a path-wise unique strong solution $u^{\eps}\in L^2(\Omega;  C([0,T];L^{2})\cap L^{p}([0,T];W_{0}^{1,p}))$. Note that $\mathcal{Z}: = C([0,T];L^{2})$ is a Polish space with respect to its usual metric. Thanks to the infinite dimensional version of Yamada-Watanabe theorem in \cite{Zhang-2008}, 
 there exists a Borel-measurable function $\mathcal{G^{\eps}} : C([0,T];\mathbb{R}) \rightarrow \mathcal{Z}$ such that $u^{\eps}$ := $\mathcal{G}^{\eps}(W)$ a.s.
 \vspace{0.1cm}
 
For $h \in L^2([0,T],\mathbb{R})$,  let  $u_h \in C([0,T];L^{2})\cap L^{p}([0,T];W_{0}^{1,p})$ be the unique weak solution~(see Section
\ref{sec:sk} for its well-posedness results) of the deterministic Skeleton equation 
\begin{equation}
\label{eq:skeleton}
   \begin{aligned}
            d {\tt B}(u_{h})  - {\rm div}_x {\tt A} (\nabla u_{h}) \, dt &= \sigma(u_{h})\,h(t)\,dt, \\
            u_{h}(0,.) &= u_0(.). 
        \end{aligned}
\end{equation}
 Define a measurable function $\mathcal{G}^{0} : C([0,T],\mathbb{R}) \rightarrow  \mathcal{Z}$ by
\begin{equation}
    \mathcal{G}^{0}(\phi) := \begin{cases}
\displaystyle   u_{h} \quad  \text{if $\phi=\int_{0}^{\cdot} h(s) \,ds$  for some $ h \in L^2([0,T],\mathbb{R})$},\\
     0 \quad \text{otherwise}.
    \end{cases}
\end{equation}
Now we state one of the main result of this article.
 \begin{thm} \label{thm:main-ldp}
  Let $p > 2$ and $u_0$ be the given deterministic function. Under the assumptions \ref{A2} - \ref{A4}, $u^{\eps}$ satisfies the LDP on $\mathcal{Z} = C([0,T];L^{2})$ with the good rate function $\mathcal{I}$ given by \eqref{eq:Imain}, where $u^{\eps}$ is the unique solution of \eqref{eq:ldp}.
\end{thm}
\begin{rem}
We mention here that due to the lack of strong monotonicity property of the operator ${\tt A}$ and the presence of doubly nonlinearity in the underlying problem, it prevents us to show LDP property of $u^\eps$ on  $C([0,T];L^{2})\cap L^{p}([0,T];W_{0}^{1,p})$---which was reported in \cite{Kavin-Majee-2022} .
\end{rem}
\subsection{Invariant measure:} Let $u(t,v)$ denotes the path-wise unique strong solution of \eqref{eq:doubly-nonlinear} with fixed initial data $u_0=v\in L^2$. Define a family of  semigroup $\{P_t\}_{t\ge 0}$ as
 \begin{align}
(P_t \phi)(v):=\mathbb{E}\big[\phi( {\tt B}(u(t,v)))\big], \quad \phi \in\mathcal{B}(L^2)\,, \label{defi:semi-group}
\end{align}
 where $\mathcal{B}(L^2)$ denotes the space of all bounded Borel measurable functions on $L^2$.
We state the following main theorem regarding the existence of invariant measure.
\begin{thm}\label{thm:existence-invariant-measure}
Let the assumptions \ref{A2}-\ref{A4} are satisfied. In addition assume that there exists $\delta>0$ such that 
\begin{align}
2C_1 C_3 \|\nabla u\|_{L^p}^p-\|\sigma(u)\|_{L^2}^2 \ge \delta \|u\|_{L^2}^p\,. \label{cond:extra-sigma}
\end{align}
Then there exists an invariant measure $\mu \in \mathcal{P}(L^2)$, the set of all Borel probability measures on $L^2$, of the semigroup
 $\{P_t\}$ for the solution of  \eqref{eq:doubly-nonlinear}. In particular, for any $t\ge 0$, there holds
\begin{align*}
\int_{L^2} P_t\phi \,d\mu=\int_{L^2} \phi\,d\mu, \quad \forall ~\phi \in SC_b(L_w^2)\,,
\end{align*}
where  $SC_b(L_w^2)$ denotes the space of bounded, sequentially weakly continuous function on $L^2$; see Definition \ref{defi:se-weak-bounded}. 
\end{thm}

\section{Skeleton Equation: Well-posedness result} \label{sec:sk}

In this section, we will prove the existence and uniqueness of solution of the skeleton equation \eqref{eq:skeleton}.  To do so, we first construct an
 approximate solution via semi-implicit time discretization scheme, and then derive necessary uniform bounds.

\subsection{Time discretization and a-priori estimate}
 First, we introduce the semi-implicit time discretization. For $N \in \mathbb{N}$, let $0 = t_0 < t_1 < ... < t_N = T$ be a uniform partition of $[0,T]$ with $\tau := \frac{T}{N} = t_{k+1}-t_k$ for all $k=0,1,...,N-1$. For any $u_0 \in L^2$, denote $\bar{u}_{0} = u_{0, \tau}$, where $ u_{0, \tau} \in W_0^{1,p}$ is the unique solution to the problem $$u_{0, \tau} - \tau \Delta_{p}u_{0, \tau} = u_0.$$ 
 Then  \cite[Lemma 30]{Vallet2019} gives
 \begin{equation}\label{eq:1.1}
\begin{aligned} 
 &  u_{0, \tau} \rightarrow u_0 \ in \ L^2 \text{ as}~~  \tau \rightarrow 0, \\
   &  \frac{1}{2} \left \|u_{0, \tau} \right \|_{L^2}^{2} + \tau \left \| \nabla u_{0, \tau}  \right \|_{L^p}^{p} \leq  \frac{1}{2} \left \|u_{0} \right \|_{L^2}^{2}.
\end{aligned}
\end{equation}
Let us define a projection  $\Pi_{\tau}$ in time as follows:  for any $h \in  L^2([0,T]; \mathbb{R})$, we define the projection $\Pi_{\tau} h \in \mathcal{P}_{\tau}$ by 
          $$\int_{0}^{T} \left ( \Pi_{\tau} h-h, \psi_{\tau}  \right ) \, dt = 0 \quad \forall \, \psi_{\tau} \in \mathcal{P}_{\tau}\,,$$
          where $\mathcal{P}_{\tau} := \left \{ \psi_{\tau}: (0,T) \rightarrow \mathbf{R} : \psi_{\tau} \mid_{(t_j,t_{j+1}]} \text{is a constant in} \, \mathbb{R}  \right \}.$
          Then \cite[Lemma 4.4]{Xin} yields
          \begin{align} \label{est:SI1}
              \begin{cases}
                  \Pi_{\tau} h \rightarrow h \quad in \, L^2([0,T]; \mathbb{R}),\\
                  \left \|\Pi_{\tau} h \right \|_{L^2([0,T]; \mathbb{R})} \leq \left \| h \right \|_{L^2([0,T]; \mathbb{R})}, 
               \displaystyle 
              \end{cases}
          \end{align}
         We  set 
          \begin{align}
          h_{k+1} = \Pi_{\tau} h(t_{k+1}) \quad k=0,1,...,N-1. \label{defi: discrete-h}
          \end{align}
       Then, by \eqref{est:SI1}, we have 
\begin{align}
\tau \sum_{k=0}^{N-1} |h_{k+1}|^2 \le \int_0^T |\Pi_\tau h|^2\,dt \le \int_0^T h^2(t)\,dt \le C\,. \label{esti:proj-time}
\end{align}
Consider a semi-implicit Euler Maruyama scheme of \eqref{eq:skeleton}:  for $k=0,1,...,N-1$,
          \begin{equation}\label{SI1}
              {\tt B}(\bar{u}_{k+1}) - {\tt B}(\bar{u}_{k}) - \tau {\rm div}_x {\tt A}(\nabla {\bar{u}}_{k+1}) = \tau \sigma(\bar{u}_{k})h_{k+1}\,,
          \end{equation}
          where $h_{k+1}$ is given in \eqref{defi: discrete-h}.  The following proposition ensures the existence of solution of  the semi-implicit scheme \eqref{SI1} 
\begin{prop} \label{extdis}
Let $\tau>0$ be small and $\bar{u}_0$ is defined as above. Then there exists a unique weak solution to the problem \eqref{SI1} i.e., for any $k = 0,1,2,...,N-1$, there exists a unique $\bar{u}_{k+1} \in W_{0}^{1,p}(D) $ such that
\begin{align}
  \int_D \left( \big\{ {\tt B}(\bar{u}_{k+1}) -{\tt B}( \bar{u}_{k})\big\}v  +  \tau {\tt A}( \nabla \bar{u}_{k+1})\cdot \nabla v \right)\,dx = \tau \int_D \sigma(\bar{u}_{k})h_{k+1} v\,dx\,. \label{variational_formula_discrete}
  \end{align}
  \end{prop}
 \begin{proof}
 For a fixed $\tau > 0$, we define an operator $\mathcal{T} : W_{0}^{1,p}(D) \rightarrow W^{-1,p'}(D)$ by
 $$ \left \langle \mathcal{T}(u), v  \right \rangle_{W^{-1,p'}(D),W_{0}^{1,p}(D)} := \int_{D} \left ({\tt B}(u) v + \tau {\tt A}(\nabla u)\cdot \nabla{v}  \right ) \, dx, \quad \forall \, u,v \in  W_{0}^{1,p}(D). $$
 Then, $\mathcal{T}$ is strictly monotone, coercive and continuous operator. Hence an application of Minty-Browder theorem \cite{Ruzicka2004} yields that $\mathcal{T}$
  is bijective. Moreover, using a similar argument as in the proof of 
  \cite[Lemma $4.1.1$]{Wittbold2019}, we infer that $\mathcal{T}^{-1}:W^{-1,p^\prime} \goto W_0^{1,p}$ is demi-continuous. Since $W_0^{1,p}$ is separable, by Dunford-Pettis theorem,
  $\mathcal{T}^{-1}:W^{-1,p^\prime} \goto W_0^{1,p}$ is continuous. 
\vspace{.1cm}

Thanks to \eqref{esti:B-Lipschitz} and the assumption \ref{A3} together with \eqref{esti:proj-time}, we observe that
\begin{align*}
\|{\tt B}(\bar{u}_{k}) + \tau \sigma(\bar{u}_{k})h_{k+1} \|_{L^2}^2\le C\big(1 + \tau^2\|h_{k+1}|^2\big)\|\bar{u}_k\|_{L^2}^2\le C\left( 1+ \int_0^T h^2(t)\,dt \right) \|\bar{u}_k\|_{L^2}^2\,.
\end{align*}
 Hence, there exists unique $\bar{u}_{k+1} \in W_{0}^{1,p}(D)$ such that
 \begin{align*} 
  & \bar{u}_{k+1} = \mathcal{T}^{-1}\big({\tt B}(\bar{u}_{k}) + \tau \sigma(\bar{u}_{k})h_{k+1}\big) \quad \forall \,k = 0,1,...,N-1. 
 \end{align*}
 Thus, by induction we finish the proof of Proposition \ref{extdis}.
 \end{proof}
 
 We wish to derive a-priori estimates for ${\tt B}(\bar{u}_{k+1})$.  Regarding this, we have the following lemma.
 \begin{lem}
Let the assumptions \ref{A2}-\ref{A4} hold true, and $\bar{u}_{k+1}$ be a solution of \eqref{SI1} for $0\le k\le N-1$.
Then there exists a constant $C>0$, independent of the time step size $\tau$, such that 
\begin{align}
 \sup_{1\le n\le N} \| {\tt B}(\bar{u}_{n})\|_{L^2}^2 +  \sum_{k=0}^{N-1} \| {\tt B}(\bar{u}_{k+1})-{\tt B}(\bar{u}_k)\|_{L^2}^2 +
  \tau \sum_{k=0}^{N-1} \|\grad \bar{u}_{k+1}\|_{L^p}^p  \le  C\,. \label{a-prioriestimate:1}
\end{align}
\end{lem}
 \begin{proof}
 To prove \eqref{a-prioriestimate:1}, we follow the ideas from \cite{Majee2020,Wittbold2019}. 
Taking  $v={\tt B}(\bar{u}_{k+1})$ as a test function in \eqref{variational_formula_discrete}, and using Young's inequality, and the identity
\begin{align}
 (a-b)a= \frac{1}{2}\big[a^2 + (a-b)^2 -b^2\big] \quad \forall\,a,b \in \R, \label{eq:identity-0}
\end{align}
we get
\begin{align}
 &\frac{1}{2} \Big\{\| {\tt B}(\bar{u}_{k+1})\|_{L^2}^2+ \| {\tt B}(\bar{u}_{k+1})-{\tt B}(\bar{u}_k)||_{L^2}^2 - \|\ {\tt B}(\bar{u}_k)\|_{L^2}^2 \Big\}  
  + \tau\,   \int_D {\tt A}(\nabla \bar{u}_{k+1}) \cdot \nabla {\tt B}(\bar{u}_{k+1})\,{\rm d}x  \notag \\
 &  \le \tau \|\sigma(\bar{u}_k)\|_{L^2} |h_{k+1}|\|{\tt B}(\bar{u}_{k+1})\|_{L^2}
 \le \frac{1}{4}\|{\tt B}(\bar{u}_{k+1})\|_{L^2}^2 + C\tau^2 |h_{k+1}|^2 \|\bar{u}_k\|_{L^2}^2 \,. \label{esti:1}
  \end{align}
  Observe that, in view of \eqref{A:coercivity} and the boundedness property of ${\tt b}^\prime$,
\begin{align}
&  \int_D {\tt A}(\nabla \bar{u}_{k+1}) \cdot \nabla {\tt B}(\bar{u}_{k+1})\,{\rm d}x =  \int_D {\tt A}(\nabla \bar{u}_{k+1}) \cdot  \nabla \bar{u}_{k+1} {\tt b}^\prime(\bar{u}_{k+1})\,{\rm d}x \notag \\
& =   \int_D  \big\{ {\tt A}(\nabla \bar{u}_{k+1}) \cdot   \nabla \bar{u}_{k+1} + K_1(x) \big\} {\tt b}^\prime(\bar{u}_{k+1})\,{\rm d}x -  \int_D K_1(x) {\tt b}^\prime(\bar{u}_{k+1})\,{\rm d}x  \notag \\
& \ge C_3 C_1 \|\nabla \bar{u}_{k+1}\|_{L^p}^p - C_4  \|K_1\|_{L^1}\,. \label{esti:2}
\end{align}
Combining \eqref{esti:1} and \eqref{esti:2}, and using the assumptions \ref{A4}, we have 
\begin{align*}
 &\frac{1}{4}\left( \| {\tt B}(\bar{u}_{k+1})\|_{L^2}^2- \|{\tt B}(\bar{u}_k)\|_{L^2}^2  +\| {\tt B}(\bar{u}_{k+1})-{\tt B}(\bar{u}_k)||_{L^2}^2 \right)  + \tau C_3 C_1 \|\nabla \bar{u}_{k+1}\|_{L^p}^p  \\
 & \le \tau C_4 \|K_1\|_{L^1} + C \tau^2 \|\bar{u_k}\|_{L^2}^2 |h_{k+1}|^2 \le   \tau C_4 \|K_1\|_{L^1} + \frac{C}{C_3^2} \tau^2 \|{\tt B}(\bar{u_k})\|_{L^2}^2 |h_{k+1}|^2  \notag \\
 & \le C\left( 1 +  \tau \int_0^T h^2(t)\,dt  \|{\tt B}(\bar{u_k})\|_{L^2}^2 \right)
\end{align*}
where in the second resp. last inequality we have used \eqref{B:reverse-L2}  resp. \eqref{esti:proj-time}.  For fixed $n$ with $1 \le n \le N$,  we sum over $k=0,1,\ldots, n-1$ to have
 \begin{align*}
&  \| {\tt B}(\bar{u}_{n})\|_{L^2}^2- \mathbb{E}\big[\| {\tt B}(\bar{u}_0)\|_{L^2}^2 +   \sum_{k=0}^{n-1}\| {\tt B}(\bar{u}_{k+1})-{\tt B}(\bar{u}_k)||_{L^2}^2 
  +  \tau \sum_{k=0}^{n-1} \mathbb{E}\big[ \|\nabla \bar{u}_{k+1}\|_{L^p}^p\big]  \notag \\
 & \le  CT \|K_1\|_{L^1}  + C\tau  \sum_{k=0}^{n-1}\mathbb{E}\big[\| {\tt B}(\hat{u}_k)\|_{L^2}^2\big]\,.
\end{align*}
We use discrete Gronwall's lemma to arrive at the assertion \eqref{a-prioriestimate:1}.
 \end{proof}
 
 Next, we define some functions on the whole time interval $[0,T]$ in terms of discrete solutions $\{\bar{u}_k\}$, and then derived essential a-priori bounds. To proceed further, we introduce the following functions:
  \begin{defi}\label{defi:approximate-skeleton}
          For $N \in \mathbb{N}$, $\tau > 0$ define the right-continuous step functions
          \[ u_{\tau}(t) = \sum_{k=0}^{N-1} \bar{u}_{k+1} \chi_{[t_k,t_{k+1})} (t), \quad t \in [0,T), \quad u_\tau(T)=u_N\,,\]
          \[ h_{\tau}(t) = \sum_{k=0}^{N-1} {h}_{k+1} \chi_{[t_k,t_{k+1})} (t), \quad t \in [0,T),\quad h_{\tau}(T)=h_N\,, \]
          the left-continuous step function
          \[\hat{u}_{\tau}(t) = \sum_{k=0}^{N-1} \bar{u}_{k} \chi_{(t_k,t_{k+1}]} (t), \quad t \in (0,T], \quad \hat{u}_{\tau}(0)=\bar{u}_0\,,\]
          and the piecewise affine functions
          \begin{align*}
          \tilde{u}_{\tau}(t) &= \sum_{k=0}^{N-1} \left (\frac{\bar{u}_{k+1}-\bar{u}_{k}}{\tau} (t-t_k) + \bar{u}_{k}  \right ) \chi_{[t_k,t_{k+1})} (t), \quad t \in [0,T),\quad  \tilde{u}_{\tau}(T) =\bar{u}_N, \notag \\
           \tilde{{\tt B}}_{\tau}(t)& = \sum_{k=0}^{N-1} \left (\frac{{\tt B}(\bar{u}_{k+1})-{\tt B}(\bar{u}_{k})}{\tau} (t-t_k) + {\tt B}(\bar{u}_{k})  \right ) \chi_{[t_k,t_{k+1})} (t), \quad t \in [0,T),\quad  \tilde{{\tt B}}_{\tau}(T) ={\tt B}(\bar{u}_N)\,. 
           \end{align*}
         \end{defi}
 Next lemma is about a-priori estimate for $\tilde{{\tt B}}_\tau$. 
 \begin{lem}\label{lem:a-priori:2}
Under the assumptions \ref{A2}-\ref{A4}, there holds
\begin{itemize}
  \item [(i)] $   \displaystyle \underset{t\in [0,T]}\sup\, \| \tilde{{\tt B}}_{\tau}(t)\|_{L^2}^2 \le C, \quad
  \int_0^T\|\nabla u_{\tau}(t)\|_{L^p}^p\,{\rm d}t  \le C,  \quad  \int_{D_T} | \nabla {\tt B}( u_\tau(t))|^{p}\,{\rm d}x\,{\rm d}t \le C$, 
  \item [(ii)] $ \displaystyle   \int_{D_T} |{\tt A}(\nabla u_\tau(t))|^{p^\prime}\,{\rm d}x\,{\rm d}t \le C, \quad \big\|\tilde{{\tt B}}_{\tau}\big\|_{L^p(0,T; W_0^{1,p})}^p\le C
  $,
\end{itemize}
for some positive constant $C$, independent of the mesh size $\tau>0$.
\end{lem}
 \begin{proof}
 A straightforward calculation shows that 
\begin{align*}
\begin{cases}
  \displaystyle  \underset{t\in [0,T]}\sup\,\| \tilde{{\tt B}}_{\tau}(t)\|_{L^2}^2 = \underset{0\le k\le N-1}\max\,\|{\tt B}(\bar{u}_{k+1})\|_{L^2}^2\,, \\
  \displaystyle \int_0^T\|\nabla u_{\tau}(t)\|_{L^p}^p\,{\rm d}t  \le \tau \sum_{k=0}^{N-1}  \|\nabla \bar{u}_{k+1}\|_{L^p}^p \,.
 \end{cases}
\end{align*}
Moreover, thanks to \eqref{esti:operator-B-grad}, one have 
$$  \int_{D_T} | \nabla {\tt B}( u_\tau(t))|^{p}\,{\rm d}x\,{\rm d}t \le C_4^p  \int_0^T\|\nabla u_{\tau}(t)\|_{L^p}^p\,{\rm d}t  .$$
Hence the assertion ${\rm (i)}$ follows from \eqref{a-prioriestimate:1}.
\vspace{0.1cm}

 Using \eqref{A:boundedness}, the assumption \ref{A3}, and the definition of $u_{\tau}$, we observe that
\begin{align*}
 &  \int_{D_T} |{\tt A}(\nabla u_\tau(t))|^{p^\prime}\,{\rm d}x\,{\rm d}t 
\le \tau \sum_{k=0}^{N-1}  \|{\tt A}(\nabla \hat{u}_{k+1})\|_{L^{p^\prime}}^{p^\prime}
  \le C_2 \tau  \sum_{k=0}^{N-1}  \|\nabla \hat{u}_{k+1}\|_{L^p}^p  + T \|K_2\|_{L^{p^\prime}}^{p^\prime}\,.
  \end{align*}
Moreover, in view of \eqref{eq:1.1} and \eqref{esti:operator-B-grad}, one can easily see that
\begin{align*}
 &  \big\|\tilde{{\tt B}}_{\tau}\big\|_{L^p(0,T; W_0^{1,p})}^p
 \le C \tau \sum_{k=0}^N   \|{\tt B}(\bar{u}_k)\|_{W_0^{1,p}}^p 
 \le C  \int_0^T \|\grad  {\tt B}(u_{\tau}(t))\|_{L^p}^p\,{\rm d}t + \tau\,\|\grad  {\tt B}(\hat{u}_0)\|_{L^p}^p\notag \\
 & \le C\,\Big(\|u_0\|_{L^2}^2  +  \int_{D_T} | \nabla {\tt B}( u_\tau(t))|^{p}\,{\rm d}x\,{\rm d}t \Big).
\end{align*}
Hence ${\rm (ii)}$ follows from  ${\rm (i)}$. This completes the proof. 
 \end{proof}
 
    \subsection{Convergence analysis of $\{\tilde{{\tt B}}_\tau\}$} This subsection is devoted to analyze the convergence of the family $\{\tilde{\tt B}_\tau\}$ in some appropriate space. Before that, we need some preparation. Let $(\mathbb{X}, \left \| \cdot \right \|_{\mathbb{X}})$ be a separable metric space. For any $q > 1$ and $\alpha \in (0,1)$, let $W^{\alpha,q}([0,T];\mathbb{X})$ be the Sobolev space  of all $u \in L^{q}([0,T];\mathbb{X})$ such that
\begin{align} \label{norm def1}
   \int_{0}^{T} \int_{0}^{T} \frac{\left \| u(t) - u(s) \right \|_{\mathbb{X}}^{q}}{\left | t-s \right |^{1+\alpha q}} \, dt \, ds < \infty 
\end{align}
with the norm 
\[ \left \| u \right \|_{W^{\alpha,q}([0,T];\mathbb{X})}^{q} = \int_{0}^{T}  \left \| u \right \|_{\mathbb{X}}^{q} \, dt + \int_{0}^{T} \int_{0}^{T} \frac{\left \| u(t) - u(s) \right \|_{\mathbb{X}}^{q}}{\left | t-s \right |^{1+\alpha q}} \, dt \, ds.  \]

\begin{lem}
\label{lem:cpt0}
The following estimation holds: for $\alpha \in (0, \frac{1}{p})$
\[ \sup_{\tau > 0} \left\{ \left\| \tilde{\tt B}_{\tau}\right\|_{W^{\alpha,p}([0,T];{W^{-1,p'}})}  \right\}  < \infty. \]
\end{lem}
\begin{proof}
By Sobolev embedding $W_{0}^{1,p} \hookrightarrow W^{-1,p'}$, and Lemma \ref{lem:a-priori:2},  we get
\begin{align}
 \int_{0}^{T} \left \|\tilde{\tt B}_{\tau}  \right \|_{W^{-1,p'}}^{p} \, dt  \leq C \int_{0}^{T} \left \|\tilde{\tt B}_{\tau}  \ \right \|_{W_{0}^{1,p}}^{p} \, dt  \leq C. \label{esti:1-frac-sov}
 \end{align}
 To proceed further, we rewrite \eqref{SI1}  as
\begin{align} \label{eq:cpt000}
    \tilde{\tt B}_{\tau}(t) &= {\tt B}(\bar{u}_{0}) + \int_{0}^{t} {\rm div}_x {\tt A}(\nabla {u}_{\tau}) \, ds
    + \int_{0}^{t} \sigma(\hat{u}_{\tau}){h}_{\tau}(s) \, ds 
     \equiv \mathcal{K}_{0}^\tau(t) + \mathcal{K}_{1}^{\tau}(t) + \mathcal{K}_{2}^{\tau}(t) \,.
\end{align}
In view  of \eqref{esti:1-frac-sov}, it remains to show  that $\mathcal{K}_{i}^{\tau}(t)$ satisfies \eqref{norm def1} for $0 \leq i \leq 2$ and $\alpha \in (0,\frac{1}{p})$.
Since $\mathcal{K}_{0}^{\tau}(t)$ is independent of time, clearly it satisfies \eqref{norm def1} for any  $\alpha \in (0,1)$.
W.L.O.G., we assume that $s < t$. By employing \eqref{A:boundedness} and H\"{o}lder's inequality, one has

\begin{align}
    \left \|\mathcal{K}_{1}^{\tau}(t) - \mathcal{K}_{1}^{\tau}(s)   \right \|_{W^{-1,p'}(D)}^{p}  & \leq \bigg( \int_{s}^{t} \left \| {\rm div}_x {\tt A}(\nabla {u}_{\tau}(r))  \right \|_{W^{-1,p'}(D)}  \, dr  \bigg)^{p} 
      \leq \bigg(\int_{s}^{t} \left \| {\tt A}(\nabla {u}_{\tau}(r))\right \|_{L^{p^\prime}} \, dr \bigg)^{p} \notag \\
      &  \leq  \bigg( C_2 \int_s^t \|\nabla u_{\tau}(r)\|_{L^p}^{p-1}\,dr + (t-s)\|K_2\|_{L^{p^\prime}} \bigg)^{p} \notag \\
      & \le C \left\{ (t-s)^\frac{1}{p}\Big( \int_s^t  \|\nabla u_{\tau}(r)\|_{L^p}^{p}\,dr\Big)^\frac{p-1}{p} + (t-s)\|K_2\|_{L^{p^\prime}} \right\}^{p} \notag \\
      & \le C (t-s) \left\{ \Big( \int_0^T  \|\nabla u_{\tau}(r)\|_{L^p}^{p}\,dr\Big)^{p-1} + T^{p-1} \|K_2\|_{L^{p^\prime}}^p \right\}\,. \notag
   \end{align}
  Thus, by  Lemma \ref{lem:a-priori:2}, there exists a constant $C>0$, independent of $\tau$, such that for any $\alpha \in (0, \frac{1}{p})$, 
\begin{align} \label{eq:cpt01}
   \int_{0}^{T}\int_{0}^{T} \frac{\left \|\mathcal{K}_{1}^{\tau}(t) - \mathcal{K}_{1}^{\tau}(s)   \right \|_{W^{-1,p'}(D)}^{p}}{\left | t-s \right |^{1+\alpha p}} \, dt \, ds \leq C.  
\end{align}
Similarly,  by employing the assumption \ref{A4} and  \eqref{B:reverse-L2} , we have 
\begin{align}
    \left \|\mathcal{K}_{2}^{\tau}(t) - \mathcal{K}_{2}^{\tau}(s)   \right \|_{L^{2}}^{p} &\leq \left \| \int_{s}^{t}  \sigma({u}_{\tau}){h}_{\tau}(r) \, dr  \right \|_{L^{2}}^{p} \leq C \bigg( (t-s) \int_{s}^{t}  \left \| {u}_{\tau}(r) \right \|_{L^2}^{2} \left | {h}_{\tau}(r) \right |^{2} \, dr \bigg)^{\frac{p}{2}}  \notag \\
    &\leq C  \bigg( (t-s) \sup_{0\leq t\leq T}\left \| {u}_{\tau}(t) \right \|_{L^2}^{2} \int_{s}^{t} \left | {h}_{\tau}(r) \right |^{2} \, dr  \bigg)^{\frac{p}{2}} \notag \\
     &\leq C \bigg( (t-s) \sup_{0\leq t\leq T}\left \| {\tt B}(u_{\tau}(t)) \right \|_{L^2}^{2} \int_{s}^{t} \left | {h}_{\tau}(r) \right |^{2} \, dr  \bigg)^{\frac{p}{2}} \notag \\
      &\leq C \bigg( (t-s) \sup_{0\leq t\leq T}\left \| \tilde{\tt B}_{\tau}(t) \right \|_{L^2}^{2} \int_{s}^{t} \left | {h}_{\tau}(r) \right |^{2} \, dr  \bigg)^{\frac{p}{2}} \notag \\
  &  \leq C (t-s)^{\frac{p}{2}} \bigg( \int_{0}^{T} \left | {h}_{\tau}(r) \right |^{2} \, dr  \bigg)^{\frac{p}{2}},  \notag
\end{align}
where in the last inequality, we have used ${\rm (i)}$ of Lemma \ref{lem:a-priori:2}. 
Therefore, by \eqref{esti:proj-time},  we have
\begin{align} \label{eq:cpt03}
   \displaystyle \int_{0}^{T}\int_{0}^{T} \frac{\left \|\mathcal{K}_{2}^{\tau}(t) - \mathcal{K}_{2}^{\tau}(s)   \right \|_{W^{-1,p'}}^{p}}{\left | t-s \right |^{1+\alpha p}} \, dt \, ds \leq C, \quad \forall \, \alpha \in (0,\frac{1}{p}).
\end{align}
Putting the inequalities \eqref{eq:cpt01}-\eqref{eq:cpt03} in \eqref{eq:cpt000}, we arrive at the assertion that
$$ \underset {\tau > 0}\sup \left\| \tilde{\tt B}_{\tau}\right\|_{W^{\alpha,p}([0,T];{W^{-1,p'}})} \leq C, \quad \forall \, \alpha \in (0,\frac{1}{p}), $$
where $C>0$ is a constant, independent of $\tau$.
\end{proof}
As mentioned earlier, we would like to have strong convergence of $\{\tilde{\tt B}_\tau\}$ in some appropriate space. For that, we will use the following well-known lemma.
  \begin{lem} \cite[Theorem 2.1]{fland}
\label{lem:cpt}
Let $\mathbb{X} \subset \mathbb{Y} \subset \mathbb{X^*}$ be Banach spaces, $\mathbb{X}$ and $\mathbb{X^*}$ reflexive, with \\
compact embedding of $\mathbb{X}$ in $\mathbb{Y}$. For any $q \in (1,\infty)$ and $\alpha \in (0,1)$, the  embedding of \\ $L^{q}([0,T];\mathbb{X}) \cap W^{\alpha,q}([0,T];\mathbb{X^*}) $  equipped with natural norm in $L^{q}([0,T];\mathbb{Y})$ is compact.
\end{lem}

Let $ \mathcal{O} = L^{2}(D_T) \cap C([0,T];W^{-1,p'}).$
\begin{thm}
\label{eq:cgs1}
There exists a subsequence of $\{\tilde{\tt B}_\tau\}$, still denoted by $\{\tilde{\tt B}_{\tau}(\cdot)\}$ and \\
${\tt B}_* \in L^{2}(D_T) \, \cap \, L^{p}([0,T];W_{0}^{1,p}) \, \cap \, L^{\infty}([0,T];L^{2})$ such that
\begin{itemize}
    \item[(a)] $\tilde{\tt B}_{\tau} \rightarrow {\tt B}_*$ in $\mathcal{O},$
\item[(b)]  $\tilde{\tt B}_{\tau} \rightharpoonup {\tt B}_* $ in $L^p([0,T];W_0^{1,p}),$
\item[(c)] $\tilde{\tt B}_{\tau} \overset{*}{\rightharpoonup}   {\tt B}_* $ in $L^{\infty}([0,T];L^{2}).$
\item[(d)] ${\tt B}(u_\tau) \goto {\tt B}_*$ in $L^2(D_T)$.
\end{itemize}
\end{thm}
\begin{proof}
\text{Proof of (a).}  
It is easy to view from  Lemma \ref{lem:cpt}  that $ L^{p}([0,T];W_{0}^{1,p}) \, \cap \, W^{\alpha,p}([0,T];{W^{-1,p'}})$ is compactly embedded in $L^{2}([0,T];L^{2})$. In view of Lemmas \ref{lem:a-priori:2} and  \ref{lem:cpt0} together with Lemma \ref{lem:cpt}, we see that $\left \{ \tilde{\tt B}_{\tau} \right \}$ is pre-compact in $L^{2}(D_T)$. Moreover, the proof of Lemma \ref{lem:cpt0} yields that there exists $\beta>0$ such that 
\begin{align*}
\|\tilde{\tt B}_\tau(t)-\tilde{\tt B}_\tau(s)\|_{W^{-1,p'}(D)} \le C |t-s|^\beta \,.
\end{align*}
 Furthermore, by the compact embedding of $L^{2} \hookrightarrow W^{-1,p'}$ together with the first part of ${\rm (i)}$, Lemma  \ref{lem:a-priori:2}, one can easily see that the family $\{\tilde{\tt B}_\tau\}$ is uniformly bounded in $W^{-1,p'}$. Hence  Arzel$\acute{a}$–Ascoli theorem yields that the family  $\left \{ \tilde{\tt B}_{\tau} \right \}$ is pre-compact in $C([0,T];W^{-1,p'})$. Hence the assertion follows. 
\vspace{0.2cm}

\noindent{ Proof of ${\rm (b)}$.} We use the estimate ${\rm (ii)}$ of Lemma  \ref{lem:a-priori:2}, and $p > 2$ to obtain that, there exists a subsequence of $\{\tilde{\tt B}_\tau\}$, still denoted by $\{\tilde{\tt B}_{\tau}\}$, and $Z \in L^p([0,T];W_0^{1,p})$  such that $\tilde{\tt B}_{\tau}\rightharpoonup Z $ in $L^{2}([0,T];W_0^{1,2})$
i.e.$$\int_{0}^{T} \int_{D} \tilde{\tt B}_{\tau}(t,x) \psi(t,x) \, dx \, dt  \rightarrow \int_{0}^{T} \int_{D} Z(t,x) \psi(t,x) \, dx \, dt  \quad \forall \, \psi \in L^{2}([0,T];W^{-1,2}(D)). $$
Note that $L^{2} \hookrightarrow W^{-1,2}$, so in particular $\tilde{\tt B}_{\tau} \rightharpoonup Z $ in $L^{2}([0,T];L^{2})$. By part $(a)$, we have seen that $\tilde{\tt B}_{\tau}(t) \rightharpoonup {\tt B}_* $ in $L^{2}(D_T)$, and hence  the assertion follows thanks to the uniqueness of weak limit.
\vspace{0.2cm}

\noindent{Proof of ${\rm (c)}$.} One can follow the similar arguments (under cosmetic changes) as in the proof of \cite[Lemma 17(6)]{gv1}, see also \cite[Lemma 3.8(iii)]{Majee2020} to conclude the result ${\rm (c)}$. 

\vspace{0.2cm}

\noindent{Proof of ${\rm (d)}$.} Observe that 
\begin{align*}
   \int_{0}^T \|{\tt B}_{\tau}^*(t)- {\tt B}(u_{\tau}(t))\|_{L^2}^2\,{\rm d}t &
 =  \sum_{k=0}^{N-1}  \|{\tt B}(\bar{u}_{k+1})-{\tt B}(\bar{u}_k)\|_{L^2}^2  \int_{t_k}^{t_{k+1}} \Big( 1- \frac{t-t_k}{\tau}\Big)^2\,{\rm d}t\Big]  \\
 &= \frac{\tau}{3}   \sum_{k=0}^{N-1} \| {\tt B}(\bar{u}_{k+1})-{\tt B}(\bar{u}_k)\|_{L^2}^2 \le C \tau\,, 
\end{align*}
where in the last inequality, we have used \eqref{a-prioriestimate:1}.  Hence the result follows from  the triangle inequality and part ${\rm (a)}$. 
\end{proof}
Observe that, since ${\tt B}^{-1}: L^2(D_T)\goto L^2(D_T)$ is Lipschitz continuous function, the function $u_*$ given by 
$$u_*(t,x):={\tt B}^{-1}{\tt B}_*(t,x)$$
is well-defined and $u_* \in L^2(D_T)$. Moreover, 
\begin{align}
\sup_{0\le t\le T} \|u_*(t)\|_{L^2}^2\le C\,. \label{esti:sup-l2-limit-fun-skeleton}
\end{align}
 We show that $u_*$ is indeed a solution of \eqref{eq:skeleton}. One may arrive at the following convergence results.
\begin{lem}\label{lem:con-skeleton-final-1}
The following hold:
\begin{itemize}
\item[i)] $ \tilde{u}_\tau \goto u_*,~ \hat{u}_\tau \goto u_*$ and $u_\tau \goto u_*$ in $L^2(D_T)$,
\item[ii)] $ \nabla u_\tau \rightharpoonup \nabla u_*$, and $ \nabla {\tt B}( u_\tau) \rightharpoonup \nabla {\tt B}_*$  in $L^p(D_T)^d$.
\end{itemize}
\end{lem}
\begin{proof}
\noindent{Proof of ${\rm i)}$:} In view of Lipschitz continuity property of ${\tt B}$, \eqref{B:reverse-poinwise} and definitions of $\tilde{u}_\tau, \hat{u}_\tau$ and $u_\tau$, one can easily get
\begin{align}\label{inq:diff-B-app-solun}
\begin{cases}
\displaystyle \int_0^T \|{\tt B}(\tilde{u}_\tau(t))- {\tt B}(u_\tau(t))\|_{L^2}^2\,dt \le C \tau\,, \\
\displaystyle \int_0^T \|{\tt B}(\hat{u}_\tau(t))- {\tt B}(u_\tau(t))\|_{L^2}^2\,dt \le C \tau\,.
\end{cases}
\end{align}
We use Lipschitz continuity of ${\tt B}^{-1}$, definition of $u_*$ together with triangle inequality, \eqref{inq:diff-B-app-solun} and ${\rm (d)}$ of Theorem 
\ref{eq:cgs1} to have
\begin{align}
\int_0^T \|\tilde{u}_\tau(t)-u_*(t)\|_{L^2}^2\,dt &= \int_0^T \| {\tt B}^{-1} {\tt B}(\tilde{u}_\tau(t))- {\tt B}^{-1} {\tt B}_*(t)\|_{L^2}^2\,dt \notag \\
& \le C  \int_0^T \|  {\tt B}(\tilde{u}_\tau(t))-  {\tt B}_*(t)\|_{L^2}^2\,dt \notag \\
& \le C \Big( \tau + \int_0^T  \|  {\tt B}({u}_\tau(t))-  {\tt B}_*(t)\|_{L^2}^2\,dt \Big)\goto 0 \quad \text{as $\tau \goto 0$}. \notag
\end{align}
In a similar way, we can show that $\hat{u}_\tau \goto u_*$ and $u_\tau \goto u_*$ in $L^2(D_T)$. 

\vspace{0.1cm}

\noindent{Proof of ${\rm ii)}$:} Observe that  the sequence $\{u_\tau\}$ is bounded in
$L^p(0,T; W_0^{1,p})$ (cf.~${\rm (i)}$, Lemma \ref{lem:a-priori:2}) and hence there exists
a function $g \in L^p(0,T; W_0^{1,p})$ such that $u_\tau  \rightharpoonup  g$ in  $ L^p(0,T; W_0^{1,p})$. Since
 $u_\tau \goto u_*$ in 
 $L^2(D_T)$ and  $L^p(0,T; W_0^{1,p})$ is continuously embedded in  $L^2(D_T)$, one may easily identify the function $g$ as
$g= u_*$. Therefore, $ \nabla u_\tau \rightharpoonup \grad u_*$ in $L^p(D_T)^d$.  A similar argument reveals that $ \nabla {\tt B}( u_\tau) \rightharpoonup \nabla {\tt B}_*$  in $L^p(D_T)^d$. This completes the proof. 
\end{proof}
\begin{lem}
\label{eq:cgseq1}
For any $\phi \in W_0^{1,p}$ and $\xi \in \mathcal{D}(0,T)$, the followings hold:
\begin{itemize}
    \item[(i)]  $\displaystyle \underset{\tau \rightarrow 0}\lim \int_{0}^{T} \int_D \big( \sigma(\hat{u}_{\tau}(t)){h}_{\tau}(t) - \sigma(u_*(t))h(t)\big)\,\phi(x)\xi(t)\,dx \, dt = 0.$
    \item[(ii)] There exists $ \vec{F} \in L^{p'}(D_T)^{d}$ such that
  $$\displaystyle \underset{\tau \rightarrow 0}\lim \int_{0}^{T}\int_D \big({\tt A}(\nabla u_\tau(t))-\vec{F}\big)\cdot \nabla \phi(x) \xi(t)\,dx dt = 0.$$
  \item[iii)] For all $t\in [0,T]$, $u_\tau(t) \rightharpoonup  u_*(t)$ in $L^2$. Moreover, $u_*$ solves the equation
\begin{align}
 {\tt B}(u_*(t))= {\tt B}(u_0) +  \int_{0}^{t}  {\rm div}_x \vec{F}(s) ds + \int_{0}^{t} \sigma(u_*(s))h(s) ds ~~~\text{in $L^2$ for all $t\in [0,T]$}. \label{eq:weak-form-limit-function-skeleton}
 \end{align}
 \end{itemize}
\end{lem}
\begin{proof}
\noindent{Proof of (i):} Using  the convergence of $h_{\tau} \rightarrow h$ in $L^2(0,T;\R)$,  ${\rm i)}$ of Lemma \ref{lem:con-skeleton-final-1}  and  the assumption \ref{A4}, it is easy to see that for any $\phi \in W_0^{1,p}$ and $\xi \in \mathcal{D}(0,T)$
\begin{align*}
    & \left |\int_{0}^{T} \int_D \big( \sigma(\hat{u}_{\tau}(t)){h}_{\tau}(t) - \sigma(u_*(t))h(t)\big)\,\phi(x)\xi(t)\,dx \, dt\right | \\
          & \leq \int_{0}^{T} \int_D \left|\big(\sigma(\hat{u}_{\tau}(t)) - \sigma(u_*(t)) \big) {h}_{\tau}(t) , \phi(x) \xi(t) \right| \, dt\,dx + \int_{0}^{T}\int_D \left|  \big(h_{\tau}(t) - h(t) \big)\sigma(u_*(t)) \phi(x)\xi(t)  \right| \,dx\, dt \\
     & \leq C ||\phi||_{W_0^{1,p}}\left(  \int_{0}^{T}   ||\hat{u}_{\tau}(t)-u_*(t)||_{L^2} |h_{\tau}(t)| \,  \, dt + \int_{0}^{T}  |h_{\tau}(t)-h(t)| \,  ||u_*(t)||_{L^2} \, dt \right)\\
     & \leq C \bigg(\int_{0}^{T} ||\hat{u}_{\tau}(t)-u_*(t)||_{L^2}^{2} \, dt \bigg)^{\frac{1}{2}} \bigg(\int_{0}^{T} |h_{\tau}(t)|^2 dt \bigg)^{\frac{1}{2}} \notag \\
     & \hspace{2cm} + C \bigg(\int_{0}^{T} |h_{\tau}(t)-h(t)|^2 dt \bigg)^{\frac{1}{2}} \bigg(\int_{0}^{T} ||u_*(t)||_{L^2}^2 \, dt \bigg)^{\frac{1}{2}} \\
     & \leq C \bigg(\int_{0}^{T} ||\hat{u}_{\tau}(t)-u_*(t)||_{L^2}^{2} \, dt \bigg)^{\frac{1}{2}} + C \bigg(\int_{0}^{T} |h_{\tau}(t)-h(t)|^2 dt \bigg)^{\frac{1}{2}} \rightarrow 0.
\end{align*}

\noindent{Proof of ${\rm (ii)}$}: In view of ${\rm ii)}$, Lemma \ref{lem:a-priori:2}, we note that the sequence $\{ {\tt A}(\nabla u_\tau)\}$ is bounded in $ L^{p'}(D_T)^{d}$, and hence there exists $ \vec{F} \in L^{p'}(D_T)^{d}$ such that such that  $ {\tt A}(\nabla u_\tau) \rightharpoonup \vec{F} $ in $ L^{p'}(D_T)^{d}$ as $\tau \rightarrow 0$. Therefore, for any $\phi \in W_0^{1,p}$ and $\xi \in \mathcal{D}(0,T)$, 
 $$\displaystyle \underset{\tau \rightarrow 0}\lim \int_{0}^{T}\int_D \big({\tt A}(\nabla u_\tau(t))-\vec{F}\big)\cdot \nabla \phi(x) \xi(t)\,dx\, dt = 0.$$
 \noindent{Proof of ${\rm (iii)}$}: For any $\xi \in \mathcal{D}(0,T)$, we have from \eqref{variational_formula_discrete}
 \begin{align*}
& - \int_0^T \int_D \tilde{\tt B}_\tau(t) \phi(x) \xi^\prime(t)\,dx\,dt 
 = - \sum_{k=0}^{N-1} \int_D \int_{t_k}^{t_{k+1}}  \tilde{\tt B}_\tau(t) \phi(x) \xi^\prime(t)\,dx\,dt  \\
& = \sum_{k=0}^{N-1} \int_D \int_{t_k}^{t_{k+1}}  \frac{d}{dt}\tilde{\tt B}_\tau(t) \phi(x) \xi(t)\,dx\,dt  
 = \sum_{k=0}^{N-1} \int_D \int_{t_k}^{t_{k+1}}  \frac{{\tt B}(\bar{u}_{k+1})-{\tt B}(\bar{u}_{k}) }{\tau}
 \phi(x) \xi(t)\,dx\,dt  \\
 &= \int_0^T \int_D \sigma(\hat{u}_\tau) h_\tau(t) \phi(x) \xi(t)\,dx\,dt - 
  \int_0^T \int_D {\tt A}(\nabla {u}_\tau) \cdot \nabla \phi(x) \xi(t)\,dx\,dt \,.
 \end{align*}
 Passing to the limit in the above equality, we have, thanks to  ${\rm i)}$-${\rm ii)}$ and ${\rm a)}$ of Theorem \ref{eq:cgs1}
 \begin{align*}
 -\int_0^T\int_D {\tt B}_*(t,x) \phi(x) \xi^\prime(t)\,dx\,dt= 
 \int_0^T \int_D \sigma(u_*) h(t) \phi(x) \xi(t)\,dx\,dt - 
  \int_0^T \int_D \vec{F} \cdot \nabla \phi(x) \xi(t)\,dx\,dt \,.
  \end{align*}
  This implies that 
  \begin{align}
  \frac{d {\tt B}_*}{dt}  = {\rm div}_x \vec{F} + \sigma(u_*) h \quad \text{in $L^{p^\prime}(0,T; W^{-1, p^\prime})$}\,. \label{eq: weak-limit-final-1-1}
  \end{align}
  We have seen that ${\tt B}_*$ is weakly continuous with values in $L^2$ and therefore, for any $\xi \in C^\infty([0,t])$ for any fixed 
  $t\in [0,T]$
  \begin{align*}
  \int_0^t \left\langle  \frac{d {\tt B}_*}{dt}, \xi(r) \phi \right \rangle_{W^{-1, p^\prime}, W_0^{1,p}}\,dr
  = -\int_0^t  \int_D{\tt B}_*(r) \xi^\prime(r) \phi(x)\,dx\,dr + \int_D \big( {\tt B}_*(t) \xi(t)- {\tt B}_*(0) \xi(0)\big)\phi(x)\,dx\,.
  \end{align*}
 Hence, using  \eqref{eq: weak-limit-final-1-1}, we have
 \begin{align}
&  \int_0^t  \int_D{\tt B}_*(r) \xi^\prime(r) \phi(x)\,dx\,dr-  \int_D \big( {\tt B}_*(t) \xi(t)- {\tt B}_*(0) \xi(0)\big)\phi(x)\,dx \notag \\
& - \int_0^t \int_D \vec{F}\cdot \nabla \phi(x) \xi(r)\,dx\,dr +  \int_0^t \int_D \sigma(u_*) h(r) \phi(x) \xi(r)\,dx\,dr=0 \notag \\
&=  \int_0^t  \int_D \tilde{\tt B}_\tau(r) \xi^\prime(r) \phi(x)\,dx\,dr-  \int_D \big( \tilde{\tt B}_\tau(t) \xi(t)- {\tt B}(\bar{u}_0) \xi(0)\big)\phi(x)\,dx \notag \\
& - \int_0^t \int_D {\tt A}(\nabla u_\tau)\cdot \nabla \phi(x) \xi(r)\,dx\,dr +  \int_0^t \int_D \sigma(\hat{u}_\tau) h_\tau(r) \phi(x) \xi(r)\,dx\,dr\,. \label{eq: weak-limit-original-equality}
 \end{align}
 Sending the limit as $\tau \goto 0$ in \eqref{eq: weak-limit-original-equality} and using again ${\rm a)}$ of Theorem  \ref{eq:cgs1} and ${\rm i)}$-${\rm ii)}$, we have 
 \begin{align}
 \lim_{\tau \goto 0 }  \int_D \big( \tilde{\tt B}_\tau(t) \xi(t)- {\tt B}(\bar{u}_0) \xi(0)\big)\phi(x)\,dx 
 =\int_D \big( {\tt B}_*(t) \xi(t)- {\tt B}_*(0) \xi(0)\big)\phi(x)\,dx \,. \label{limit-weak-original-skeleton}
 \end{align}
 Using \eqref{eq:1.1} in \eqref{limit-weak-original-skeleton}, we conclude that ${\tt B}_*(0)= {\tt B}(u_0)$ in $L^2$ and for all $t\in [0,T]$,
 $\tilde{\tt B}_\tau(t)  \rightharpoonup {\tt B}_*(t) $ in $L^2$.  Moreover for all $t\in [0,T]$, we have 
 \begin{align*}
 {\tt B}_*(t)= {\tt B}(u_0) +  \int_{0}^{t}  {\rm div}_x \vec{F}(s) ds + \int_{0}^{t} \sigma(u_*(s))h(s) ds\,. 
  \end{align*}
  Thus,  equation \eqref{eq:weak-form-limit-function-skeleton} holds true if we show that ${\tt B}_*(t,x)= {\tt B}(u_*(t,x))$.  Note that, since ${\tt B}$ is Lipschitz continuous and $u_\tau \goto u_*$ in $L^2(D_T)$,
we have ${\tt B}(u_\tau) \goto {\tt B}(u_*)$ in $L^2(D_T)\big)$. We have already seen that  ${\tt B}(u_\tau) \rightharpoonup  {\tt B}_*$
in $ L^p(0,T; W_0^{1,p})$, and ${\tt B}_*\in  \mathcal{C}([0,T]; L^2_w) $. Hence 
$${\tt B}(u_*(t))={\tt B}_*(t) \quad \forall\, t\in [0,T]\,.$$
 This completes the proof.
\end{proof}

\subsection{Weak solution of Eq. \eqref{eq:skeleton}--existence  proof} \label{subsec:existence-skeleton}
 We show that the limit function $u_*$ is indeed a weak solution of  \eqref{eq:skeleton}. In view of the weak form 
\eqref{eq:weak-form-limit-function-skeleton}, it it is enough to prove that $\vec{F}={\tt A}(\nabla u_*)$.  To do so, we follow \cite{Majee2020,Wittbold2019} of its deterministic counterpart.
Take a test function ${\tt B}(\bar{u}_{k+1})$ in \eqref{SI1} and  use the identity \eqref{eq:identity-0}. The result is 
\begin{align*}
 & \frac{1}{2}\left[||{\tt B}(\bar{u}_{k+1})||_{L^2}^{2} - ||{\tt B}(\bar{u}_{k})||_{L^2}^{2} + || {\tt B}(\bar{u}_{k+1}) - {\tt B}(\bar{u}_{k})||_{L^2}^{2}\right] +  \tau \int_{D} {\tt A}(\nabla \bar{u}_{k+1}) \cdot \nabla {\tt B}(\bar{u}_{k+1}) \, dx \\
   & \hspace{2cm} = \tau \int_{D} \sigma(\bar{u}_{k})h_{k+1} \cdot {\tt B}(\bar{u}_{k+1}) \, dx.
\end{align*}
Summing over $k = 0,1,...,N-1$ and using the definition that $\tilde{\tt B}_{\tau}(T) = {\tt B}(\bar{u}_N)$, we obtain
\begin{align}
     \frac{1}{2}||\tilde{\tt B}_{\tau}(T)||_{L^2}^{2} +  \int_{0}^{T} \int_{D} {\tt b}^\prime(u_{\tau}(t)){\tt A}(\nabla u_{\tau}(t))\cdot \nabla u_{\tau}(t)\,{\rm d}x\,{\rm d}t
      \notag \\
     -  \int_{0}^{T} \int_{D}  \sigma(\hat{u}_{\tau})(t)h_{\tau}(t) {\tt B} (u_{\tau}(t)) \, dx \, dt \leq \frac{1}{2}|| {\tt B}(\bar{u}_{0})||_{L^2}^{2}. \label{eq:discrte-ito-type}
\end{align}
On the other hand, from Lemma \ref{eq:cgseq1}, we get
\begin{align}
  \left \|{\tt B}(u_*(T)) \right \|_{L^2}^{2}  + 2 \int_{0}^{T} \int_{D}   \vec{F} \cdot \nabla {\tt B}(u_*(t)) \, dx \, dt  = \left \| {\tt B}(u_{0}) \right \|_{L^2}^{2} + 2 \int_{0}^{T} \left<\sigma(u_*(t))h(t), {\tt B}(u_*(t)) \right> \, dt .  \label{eq:ito-type}
\end{align}
Subtracting \eqref{eq:ito-type} from \eqref{eq:discrte-ito-type}, one has
\begin{align}
   &  \frac{1}{2}  \left\{||\tilde{\tt B}_{\tau}(T)||_{L^2}^{2}  -   \left \|{\tt B}(u_*(T)) \right \|_{L^2}^{2}\right\}+ \left\{ \int_{D_T} \Big[ {\tt b}^\prime(u_{\tau}(t)){\tt A}(\nabla u_{\tau}(t))\cdot \nabla u_{\tau}(t)-  \vec{F} \cdot \nabla {\tt B}(u_*(t))\Big] \, dx \, dt \right\} \notag \\
     & \leq  \frac{1}{2}\left\{ || {\tt B}(\bar{u}_{0})||_{L^2}^{2}- \left \| {\tt B}(u_{0}) \right \|_{L^2}^{2} \right\} + \left \{ \int_{D_T} \Big[ \sigma(\hat{u}_{\tau})(t)h_{\tau}(t) {\tt B} (u_{\tau}(t)) - \sigma(u_*(t))h(t) {\tt B}(u_*(t))\Big]\,dx\, dt \right\}. \label{esti:vecF-identify-step-1}
\end{align}
Using the fact that for all $t\in [0,T]$,  $\tilde{\tt B}_\tau(t)  \rightharpoonup {\tt B}_*(t)={\tt B}(u_*) $ in $L^2$ and \eqref{eq:1.1}, one easily check that 
\begin{align}
\underset{\tau}\liminf \, ||\tilde{\tt B}_{\tau}(T)||_{L^2}^{2}  -  \left \|{\tt B}(u_*(T)) \right \|_{L^2}^{2} \geq 0, \quad 
  || {\tt B}(\bar{u}_{0})||_{L^2}^{2} \rightarrow  \left \| {\tt B}(u_{0}) \right \|_{L^2}^{2}\,. \label{esti:vecF-identify-step-1-0}
\end{align}
Like in Lemma \ref{lem:a-priori:2}, ${\rm i)}$, one can easily show that $\underset{0\le t\le T} \sup \|{\tt B}(u_\tau(t))\|_{L^2}^2< + \infty$. Using the above uniform estimate, the convergence results in ${\rm i)}$ of Lemma \ref{lem:con-skeleton-final-1}, the uniform estimate \eqref{esti:sup-l2-limit-fun-skeleton} and the fact that
$h_\tau \goto h$ in $L^2(0,T;\R)$, we have
\begin{align}
 \int_{D_T} \sigma(\hat{u}_{\tau})(t)h_{\tau}(t) {\tt B} (u_{\tau}(t))\,dx\,dt  \goto  \int_{D_T}  \sigma(u_*(t))h(t) {\tt B}(u_*(t))\Big]\,dx\, dt \,. \label{esti:vecF-identify-step-1-1}
\end{align}
Using \eqref{esti:vecF-identify-step-1-0}-\eqref{esti:vecF-identify-step-1-1} in \eqref{esti:vecF-identify-step-1}, we obtain
\begin{align}
    \underset{\tau}\limsup  \int_{D_T} {\tt b}^\prime(u_{\tau}(t)){\tt A}(\nabla u_{\tau}(t))\cdot \nabla u_{\tau}(t)\,dx\,dt  \leq  \int_{0}^{T} \int_{D}   \vec{F} \cdot \nabla {\tt B}(u_*(t))\Big] \, dx \, dt \,. \label{esti:vecF-identify-step-2}
\end{align}

Observe that
\begin{align}
 & \int_{D_T} {\tt b}^\prime(u_\tau(t)) {\tt A}(\grad u_{\tau}(t))\cdot \grad u_{\tau}(t)\,{\rm d}x\,{\rm d}t
-   \int_{D_T} \vec{F} \cdot \grad  {\tt B}(u_*) \,{\rm d}x\,{\rm d}t \notag \\
&=    \int_{D_T} {\tt b}^\prime(u_\tau(t)) \big( {\tt A}(\nabla u_\tau(t))-{\tt A}(\nabla u_*(t))\big)\cdot (\nabla u_\tau(t)-\nabla u_*(t)) \,{\rm d}x\,{\rm d}t \notag \\
&  \quad +   \int_{D_T} {\tt b}^\prime(u_\tau(t))\, {\tt A}(\nabla u_*(t))\cdot (\nabla u_\tau(t)-\nabla u_*(t)) \,{\rm d}x\,{\rm d}t \notag \\
&\qquad +   \int_{D_T} {\tt b}^\prime(u_\tau(t))\, \nabla u_*(t) \cdot \big\{ {\tt A}(\nabla u_\tau(t))- \vec{F}\big\} \,{\rm d}x\,{\rm d}t
-  \int_{D_T} {\tt b}^\prime(u_\tau(t))\,  \vec{F} \cdot \big\{ \nabla u_\tau(t)- \nabla u_*(t)\big\} \,{\rm d}x\,{\rm d}t \notag \\
& \qquad \qquad +  \int_{D_T} G \cdot \big\{ \nabla {\tt B}(u_\tau(t))- \nabla {\tt B}(u_*(t))\big\} \,{\rm d}x\,{\rm d}t  := \sum_{i=1}^5 \mathcal{A}_i\,. \notag
\end{align}
We estimate each of the term $\mathcal{A}_i~(1\le i\le 5)$ separately. Since $u_\tau \goto u_*$ in $L^2(D_T)$, the assumption \ref{A3} yields that
 ${\tt b}^\prime(u_\tau) \goto {\tt b}^{\prime}(u_*)$ a.e. in $D_T$.  Moreover, ${\tt b}^\prime(u_\tau)\vec{F} \goto  {\tt b}^{\prime}(u_*) \vec{F}$ a.e. in $ D_T$ and,
 $|{\tt b}^\prime(u_\tau) \vec{F}|^{p^\prime} \in L^1( D_T)$.
  Hence by Lebesgue convergence theorem,
 $$ {\tt b}^\prime(u_\tau)\vec{F} \goto {\tt b}^\prime(u_*) \vec{F} \quad \text{in}~~L^{p^\prime}( D_T)^d\,.$$
Using Lemma \ref{lem:con-skeleton-final-1}, ${\rm ii)}$ and above convergence, we conclude that $\mathcal{A}_4 \goto 0$ as $\tau \goto 0$. 
A similar argument yields that  $\mathcal{A}_2 \goto 0$ as $\tau \goto 0$. 
\vspace{0.1cm}

Next we consider $\mathcal{A}_3$. Observe that  ${\tt b}^\prime(u_\tau) \nabla u_* \goto  {\tt b}^{\prime}(u_*) \nabla u_*$ a.e. in $D_T$, and 
$|{\tt b}^\prime(u_\tau) \nabla u_*|^{p} \in L^1(D_T)$.
 Hence an application of Lebesgue convergence theorem, along with the fact that ${\tt A}(\nabla u_\tau) \rightharpoonup \vec{F} $ in $ L^{p^\prime}( D_T)^d$ implies that  $\mathcal{A}_3 \goto 0$ as $\tau \goto 0$.
 
  Since $\nabla {\tt B}(u_\tau) \rightharpoonup  \grad {\tt B}( u_*)$ in $ L^p( D_T)^d$, and $\vec{F}\in L^{p^\prime}( D_T)^d$, it follows that $ \underset{\tau\goto 0}\lim \mathcal{A}_5=0$.
  
Hence, by using \eqref{esti:vecF-identify-step-2}, we get
\begin{align}
& \limsup_{\tau \goto 0} \int_{D_T} {\tt b}^\prime(u_\tau(t)) \big( {\tt A}(\nabla u_\tau(t))-{\tt A}(\nabla u_*(t))\big)\cdot (\nabla u_\tau(t)-\nabla u_*(t)) \,{\rm d}x\,{\rm d}t \notag \\
& =\limsup_{\tau \goto 0} \int_{D_T} {\tt b}^\prime(u_\tau(t)) {\tt A}(\grad u_{\tau}(t))\cdot \grad u_{\tau}(t)\,{\rm d}x\,{\rm d}t
-  \int_{D_T} \vec{F} \cdot \grad  {\tt B}(u_*) \,{\rm d}x\,{\rm d}t \le 0\,, \label{esti:5-new}
\end{align}
In view of \eqref{A:monotonocity}, the boundedness property of ${\tt b}^\prime$ and \eqref{esti:5-new}, we have
\begin{align*}
0\le  & C_3 \limsup_{\tau \goto 0} \int_{D_T} \big( {\tt A}(\grad u_{\tau})- {\tt A}( \grad u_*)\big)\cdot \big( \nabla u_\tau - \nabla u_*\big) \,{\rm d}x\,{\rm d}t \notag \\
 & \le  \limsup_{\tau \goto 0}  \int_{D_T} {\tt b}^\prime(u_\tau(t)) \big( {\tt A}(\nabla u_\tau(t))-{\tt A}(\nabla u_*(t))\big)\cdot (\nabla u_\tau(t)-\nabla u_*(t)) \,{\rm d}x\,{\rm d}t \le 0\,,
\end{align*}
and hence 
\begin{align}
\lim_{\tau \goto 0} \int_{D_T} \big( {\tt A}(\grad u_{\tau})- {\tt A}( \grad u_*)\big)\cdot \big( \nabla u_\tau - \nabla u_*\big) \,{\rm d}x\,{\rm d}t =0\,. \label{esti:6-new}
\end{align}
Since ${\tt A}(\nabla u_\tau) \rightharpoonup \vec{F}$ in $L^{p^\prime}( D_T)^d$ and  $ \grad u_{\tau}\rightharpoonup \grad u_*$ in $L^p(D_T)^d$, by using \eqref{esti:6-new}, we infer that
\begin{align*}
\lim_{\tau \goto 0}  \int_{D_T} {\tt A}(\grad u_{\tau}) \cdot  \nabla u_\tau \,{\rm d}x\,{\rm d}t=\int_{D_T} \vec{F}\cdot  \nabla u_* \,{\rm d}x\,{\rm d}t\,.
\end{align*}
Since ${\tt A}: L^{p}(D_T)^d \goto L^{p^\prime}( D_T)^d$  fulfills the $(M)$-property, we conclude that
$$ {\tt A}(\nabla u_*)= \vec{F} ~~~\text{in}~~L^{p^\prime}(D_T)^d\,.$$ 
  This completes the existence proof. 
 
 
\subsection{Proof of uniqueness} \label{sec:uniqueness}
In this subsection, we prove uniqueness of the solution of \eqref{eq:skeleton} via $L^1$-contraction principle. To do so, let $h \in L^2([0,T]; \mathbb{R})$ be fixed and $u_{1}$ and $u_{2}$ be two solution of \eqref{eq:skeleton}. Let us consider the following convex approximation of the absolute value function. Let $\zeta : \mathbb{R} \rightarrow \mathbb{R}$ be a $C^{\infty}$ function satisfying 

\[ \zeta (0) = 0, \quad \zeta (-r) = \zeta (r), \quad \zeta^{'}(-r) = - \zeta^{'}(r), \quad \zeta^{''} \geq 0, \]
and
\begin{center}
    $ \zeta^{'}(r) = \left\{\begin{matrix}
  -1 & \text{when} \ r \leq 0, \\ 
  \in [-1,1] & \text{when} \ |r| < 1, \\
  1 & \text{when} \ r \geq 1.
\end{matrix}\right.$
\end{center}
 
 For any $\vartheta > 0$, define $\zeta_{\vartheta} : \mathbb{R} \rightarrow \mathbb{R}$ by $ \displaystyle \zeta_{\vartheta} = \vartheta \zeta (\frac{r}{\vartheta})$. Then, one can check that  $\zeta_\vartheta \in C^\infty(\R)$ and satisfies the following properties: 
 \begin{align*}
  |r| - K_1 \vartheta \leq \zeta_{\vartheta} \leq |r|,  \quad |\zeta_{\vartheta}^{''}(r)| \leq \frac{K_2}{\vartheta} \textbf{1}_{|r| \leq \vartheta},
  \end{align*}
 where $K_1 = \underset{|r|\leq 1}\sup \left | |r| - \zeta(r) \right |$ and $K_2 = \underset{|r|\leq 1}\sup |\zeta^{''}(r)|$.
 \vspace{0.5cm}

Using chain-rule and integration by parts formula, we have
 \begin{align}
&\int_{D} \zeta_{\vartheta}({\tt B}(u_1(t)) -{\tt B}( u_2(t))) dx \notag \\
& =  - \int_{0}^{t}\int_{D} \left( {\tt A}(\nabla u_1) - {\tt A}(\nabla u_2 )\right) \cdot \left( \nabla {\tt B}(u_1(s)) -\nabla {\tt B} (u_2(s))\right)
 \, \zeta_{\vartheta}^{''}({\tt B}(u_1(s)) -{\tt B}( u_2(s))) \, dx \, ds  \notag \\
  & \qquad  + \int_{0}^{t}\int_{D} \zeta_{\vartheta}^{'}({\tt B}(u_1(t)) -{\tt B}( u_2(t))) \left(\sigma(u_1) - \sigma(u_2)\right) \, h(s)  \, dx \, ds\,. \notag
 \end{align}
 Thanks to \eqref{A:monotonocity}, the assumption \ref{A3} and the fact that $\zeta_{\vartheta}^{\prime \prime}\ge 0$, one can easily observe that
\begin{align*}
 & -  \int_{D}\big(  {\tt A}(\grad u_1) - {\tt A}(\grad u_2)\big)
 \cdot  \big(\grad {\tt B}(u_1)- \grad {\tt B}(u_2)\big)\zeta_{\vartheta}^{\prime\prime} \big({\tt B}(u_1)-{\tt B}(u_2)\big)\,{\rm d}x \\
 &= -  \int_D {\tt b}^\prime(u_1)\big(  {\tt A}(\grad u_1) - {\tt A}(\grad u_2)\big)
 \cdot  \big(\grad u_1- \grad u_2\big)\zeta_{\vartheta}^{\prime\prime} \big({\tt B}(u_1)-{\tt B}(u_2)\big)\,{\rm d}x \\
 & \qquad  -  \int_D \big(  {\tt A}(\grad u_1) - {\tt A}(\grad u_2)\big)
 \cdot   \grad u_2 \big( {\tt b}^\prime(u_1)-{\tt b}^\prime(u_2)\big)\zeta_{\vartheta}^{\prime\prime} \big({\tt B}(u_1)-{\tt B}(u_2)\big)\,{\rm d}x  \\
 & \le  -  \int_D \big(  {\tt A}(\grad u_1) - {\tt A}(\grad u_2)\big)
 \cdot   \grad u_2 \big( {\tt b}^\prime(u_1)-{\tt b}^\prime(u_2)\big)\zeta_{\vartheta}^{\prime\prime} \big({\tt B}(u_1)-{\tt B}(u_2)\big)\,{\rm d}x \,.
\end{align*}
Thus, we have 
\begin{align}
 &  \int_{D}\zeta_\vartheta \big( {\tt B}(u_1(t))-{\tt B}(u_2(t))\big)\,{\rm d}x  \notag \\
 & \le  - \int_{0}^t \int_D \big(  {\tt A}(\grad u_1(s)) - {\tt A}(\grad u_2(s))\big)
 \cdot   \grad u_2(s) \big( {\tt b}^\prime(u_1(s))-{\tt b}^\prime(u_2(s))\big) \notag \\
 & \hspace{3cm} \times \zeta_{\vartheta}^{\prime\prime} \big({\tt B}(u_1(s))-{\tt B}(u_2(s))\big)\,{\rm d}x\,{\rm d}s
  \notag \\
 & \qquad + \int_{0}^{t}\int_{D} \left |\sigma(u_1) - \sigma(u_2)\right | \, h(s)  \, dx \, ds 
  \equiv \mathcal{A}_6 + \mathcal{A}_7 \,. \label{inq:sum-a-b-pathwise-uniqueness}
\end{align}
We estimate each of the above terms separately. Consider the term $\mathcal{A}_6$. 
Since $\zeta_{\vartheta}^{\prime \prime}(r)\le \frac{M_2}{\vartheta} \textbf{1}_{\{|r| \le \vartheta\}}$ and ${\tt b}^\prime$ is Lipschitz continuous 
function, by using \eqref{B:reverse-poinwise}, we have 
\begin{align*}
& \big({\tt b}^\prime(u_1)-{\tt b}^\prime(u_2)\big) \nabla u_2\cdot \big( {\tt A}(\nabla u_1)-{\tt A}(\nabla u_2)\big)\zeta_{\vartheta}^{\prime\prime}\big( {\tt B}(u_1)-{\tt B}(u_2)\big) \\
& \le C_{{\tt b}^\prime} |u_1 -u_2|\,|\grad u_2| \big| {\tt A}(\nabla u_1)-{\tt A}(\nabla u_2)\big| \frac{M_2}{\vartheta} \textbf{1}_{\{|{\tt B}(u_1)-{\tt B}(u_2)| \le \vartheta\}}  \\
&  \le  \frac{C_{{\tt b}^\prime}}{C_3} |{\tt B}(u_1) -{\tt B}(u_2)|\,|\grad u_2| \big|{\tt A}(\nabla u_1)-{\tt A}(\nabla u_2)\big| \frac{M_2}{\vartheta} \textbf{1}_{\{|{\tt B}(u_1)-{\tt B}(u_2)| \le \vartheta\}} \goto 0 \quad (\vartheta \goto 0)
\end{align*}
for almost every $(t,x)\in D_T$. Moreover
\begin{align*}
 & \big| {\tt b}^\prime(u_1)-{\tt b}^\prime(u_2)\big| |\nabla u_2| \big| {\tt A}(\nabla u_1)-{\tt A}(\nabla u_2)\big|\beta_{\vartheta}^{\prime\prime}\big( {\tt B}(u_1)-{\tt B}(u_2)\big)  \notag \\
 & \le  \frac{C_{{\tt b}^\prime}}{C_3} M_2 |\nabla u_2| \big| {\tt A}(\nabla u_1)-{\tt A}(\nabla u_2)\big| \in L^1( D_T)\,.
\end{align*}
 Hence, the dominated convergence theorem yields that $\mathcal{A}_6\goto 0$ as $\vartheta \goto 0$.
\vspace{.1cm}

  We use \eqref{B:reverse-poinwise} and  the assumption \ref{A4}  to estimate $\mathcal{A}_7$ as
 \begin{align*}
    |\mathcal{A}_7| \, & \leq C \int_{0}^{t}\int_{D} |u_1 - u_2| \, |h(s)| \, dx \, ds \leq C \int_{0}^{t} \Big( \int_{D} |u_1 - u_2|\, dx \Big) |h(s)|  \, ds \notag \\
    & \le C  \int_{0}^{t} \Big( \int_{D} \big| {\tt B}(u_1(s)) -{\tt B}( u_2(s))\big|\,dx \Big)\,ds\,.
    \end{align*}
\vspace{.1cm}

Combining things together and passing to the limit in \eqref{inq:sum-a-b-pathwise-uniqueness}, we have for any $t\in [0,T]$,
\begin{align}
 \int_{D}\big| {\tt B}(u_1(t,x))-{\tt B}(u_2(t,x))\big|\,{\rm d}x \le C  \int_{0}^{t} \Big( \int_{D} \big| {\tt B}(u_1(s)) -{\tt B}( u_2(s))\big|\,dx \Big)\,ds\,. \notag
\end{align}
Thus, using Gronwall's lemma, we have  ${\tt B}(u_1(t,x))={\tt B}(u_2(t,x))$ for a.e. $(t,x)\in D_T$. Since ${\tt b}: \R \goto \R$ is bijective, in view of the definition of ${\tt B}$, we may conclude that  $u_1(t,x)=u_2(t,x)$ for a.e. $(t,x)\in D_T$.  This complete the uniqueness proof.

\section{Proof of Large deviation principle~(Theorem \ref{thm:main-ldp})} \label{sec:LDP}
In this section, we give the proof of Theorem \ref{thm:main-ldp}. In view of Theorem \ref{thm:ldp-general-sufficient}, it suffices to show that  the conditions \ref{C1} and \ref{C2} is fulfilled by the solution $u^\eps$
of the equation \eqref{eq:ldp}. 

\subsection{Proof of condition \ref{C2}}\label{subsec:proof-cond-c2}
To prove condition \ref{C2}, since $S_M$ is weakly compact, it is enough to prove that if $h_{n} \rightarrow h$ weakly in $L^2([0,T],\mathbb{R})$ then $u_n \rightarrow u_h$ strongly in $C([0,T];L^{2})$, where $u_n$ and $u_h$ are the solution of \eqref{eq:skeleton} corresponding to $h_n$ and $h$ respectively. We prove this in several steps. 
\vspace{0.1cm}

\noindent {\bf   Step-I:} \underline{A-priori estimate for $ {\tt B}(u_n)$}:   In this step, we derive necessary a-priori a-priori estimates for $ {\tt B}(u_n)$.  There exists a constant $C$, independent of $n$ such that 
\begin{equation}
 \label{eq:ub}
    \begin{aligned}
      \sup_{n}\left[ \sup_{s \in [0,T]} \left \|{\tt B}(u_n(s)) \right \|_{L^2(D)}^{2} + \int_{0}^{T} \left \|u_n(s) \right \|_{W_{0}^{1,p}(D)}^{p} ds\right]    \leq C\,, \\
      \sup_{n}\left[ \int_{D_T} \left( |{\tt A}(\nabla u_n)|^{p^\prime} + |\nabla u_n|^p \right)\,dx\,dt\right] \le C\,.
    \end{aligned}
\end{equation}
Indeed,  an application of chain-rule, the integration by parts formula, the assumption \ref{A4} and Cauchy-Schwartz inequality together with \eqref{B:reverse-L2} gives
\begin{align}
    \frac{d}{dt} \left \| {\tt B}(u_n(t)) \right \|_{L^2}^{2} & = -2\left< {\tt A}(\nabla u_n), 
\nabla {\tt B}(u_n) \right> + 2 \left< \sigma(u_n)h_n(t), {\tt B}(u_n(t))\right> \notag \\
    & \leq -2\left< {\tt A}(\nabla u_n), 
\nabla {\tt B}(u_n) \right> 
    + C \left \| u_n(t) \right \|_{L^2}^{2} + \left | h_n(t) \right |^2 \left \| {\tt B}(u_n(t)) \right \|_{L^2}^{2} \notag \\
    & \leq -2\left< {\tt A}(\nabla u_n), 
\nabla {\tt B}(u_n) \right> 
    + C(1+ \left | h_n(t) \right |^2)\left \| {\tt B}(u_n(t)) \right \|_{L^2}^{2} \,. \notag
\end{align}
Using the estimate \eqref{esti:2}, integrating over time from $0$ to $t$, and then applying Gronwall's lemma together with the boundedness property of $h_n$ i.e., 
\begin{align}
 \displaystyle \int_0^T h_n^2(t)\,dt \le M\,, \label{esti:boundeness-hn}
 \end{align}
  we have
$$ \sup_{n} \sup_{0\le t\le T}  \left \| {\tt B}(u_n(t)) \right \|_{L^2}^{2} \le C(T,M).$$
Moreover, thanks to Poincare inequality and above uniform bound, one has 
\begin{align}
 \sup_{n} \bigg[ \sup_{s \in [0,T]} \left \| {\tt B}(u_n(s)) \right \|_{L^2}^{2} + \int_{0}^{T} \left \|u_n(s) \right \|_{W_{0}^{1,p}}^{p} ds \bigg]   \leq C_{M}. \label{esti:apriori-1-B-un}
 \end{align}
Furthermore, one may use \eqref{A:boundedness} together with \eqref{esti:apriori-1-B-un} to arrive at the second assertion of \eqref{eq:ub}.
\vspace{0.2cm}

Following the proof of Lemma \ref{lem:cpt0} along with \eqref{esti:boundeness-hn}, we get the following estimates.
\begin{align}\label{esti:frac-Sov-sequence}
\begin{cases}
 \displaystyle \sup_{n} \left\{ \left\| {\tt B}(u_n)\right\|_{W^{\alpha,p}([0,T];{W^{-1,p'}})}  \right\}  < \infty \quad \text{for}~ \alpha \in (0,\frac{1}{p})\,, \\
\displaystyle \|{\tt B}(u_n(t))-{\tt B}(u_n(s))\|_{W^{-1,p'}} \le C |t-s|^\beta \,.
\end{cases}
\end{align}
for some $\beta>0$. 
\vspace{0.2cm}

\noindent {\bf Step-II:} \underline{ Convergence analysis of $ {\tt B}(u_n)$:  } In view of a-priori estimates \eqref{eq:ub} and \eqref{esti:frac-Sov-sequence} and  Lemma \ref{lem:cpt} along with  Arzel$\acute{a}$–Ascoli theorem, Poincare inequality and \eqref{esti:operator-B-grad}, there exist a subsequence  of $\{ {\tt B}(u_n)\}$, still denoted by $\left \{{\tt B}(u_n) \right \}$, and $\bar{\tt B}_* \in L^{2}(D_T) \, \cap \, L^{p}([0,T];W_{0}^{1,p}) \, \cap \, L^{\infty}([0,T];L^{2})$ such that
\begin{align}\label{conv:c2-1}
\begin{cases}
{\tt B}(u_n) \goto \bar{\tt B}_* \quad \text{in}~~\mathcal{O}, \\
\nabla {\tt B}(u_n) \rightharpoonup \nabla \bar{\tt B}_*\quad \text{in}~~L^p(D_T)^d\,.
\end{cases}
\end{align}
Define a function $\bar{u}_*$ via
$$ \bar{u}_*(t,x):= {\tt B}^{-1} \bar{\tt B}_*(t,x).$$
Then $ \bar{u}_*$ is well-defined and since $ \bar{\tt B}_*\in  L^{\infty}([0,T];L^{2})$, by Lipschitz continuity property of $ {\tt B}^{-1}$, we have
\begin{align}
\sup_{0\le t\le T} \|\bar{u}_*(t)\|_{L^2}^2\le C\,. \label{esti:sup-l2-limit-fun-c2-cond}
\end{align}
Moreover, Lipschitz continuity property of $ {\tt B}^{-1}$ and convergence result in \eqref{conv:c2-1} implies that 
\begin{align}
u_n\goto \bar{u}_* \quad \text{in}~~L^2(D_T)\,. \label{conv:c2-2}
\end{align}
With \eqref{conv:c2-1} and \eqref{conv:c2-2} at hand, we focus on the convergence of the controlled drift term. 
 Observe that, for any $\phi\in W_0^{1,p}$
\begin{align}
     &\left |\int_{0}^{T} \left<\sigma(u_n(t)){h}_n(t) - \sigma(\bar{u}_*(t))h(t) , \phi \right> \, dt  \right | \notag \\
     &\leq \left | \int_{0}^{T} \left< \big(\sigma(u_n(t)) - \sigma(\bar{u}_*(t)) \big) {h}_n(t) , \phi \right> \, dt + \int_{0}^{T} \left< \big(h_{n}(t) - h(t) \big)\sigma(\bar{u}_*(t)) , \phi \right> \, dt  \right | \notag \\
     &:= \mathcal{I}_1(T) + \mathcal{I}_2(T). \notag
\end{align}
An application of Holder's inequality, the assumption \ref{A4},  \eqref{esti:boundeness-hn} and \eqref{conv:c2-2} reveals that
\begin{align*}
    \mathcal{I}_1(T) & \leq \int_{0}^{T} \left \| \big(\sigma(u_n(t)) - \sigma(\bar{u}_*(t) \big) {h}_n(t) \right \|_{L^2} \left \| \phi \right \|_{L^2} dt \\
    & \leq \int_{0}^{T} \left \| \sigma(u_n(t)) - \sigma(\bar{u}_*(t)   \right \|_{L^2} \left |{h}_n(t)  \right | \left \| \phi \right \|_{L^2} dt \\
    & \leq C \int_{0}^{T} \left \| u_n - \bar{u}_* \right \|_{L^2} \left |{h}_n(t)  \right | dt 
      \leq C  \bigg(\int_{0}^{T} \left \| u_n - \bar{u}_* \right \|_{L^2}^{2} \, dt\bigg)^{\frac{1}{2}} \bigg(\int_{0}^{T}  \left |{h}_n(t)  \right |^2  \, dt \bigg)^{\frac{1}{2}} \rightarrow 0.
    \end{align*}
    Notice that the function $\displaystyle \psi(t): = \int_{D} \sigma(\bar{u}_*(t,x) \phi(x) \, dx $  lies in $L^2([0,T]; \mathbb{R})$. Indeed,  thanks to the assumption \ref{A4} and  \eqref{esti:sup-l2-limit-fun-c2-cond}
    \begin{align*}
   \|\psi\|_{L^2[0,T]}^2\leq \int_{0}^{T} \bigg(\int_{D} |\sigma(\bar{u}_*(t,x)|^2 \, dx \bigg)  \bigg(\int_{D} |\phi(x)|^2 \, dx \bigg) \, dt 
     \leq C \left \| \phi \right \|_{L^2}^2  \int_{0}^{T} \left \|\bar{u}_*(t) \right \|_{L^2}^2 dt < \infty.
\end{align*}
Since $h_{n}  \rightharpoonup h$ in $L^2([0,T],\mathbb{R})$, one then easily get
 $\displaystyle    \mathcal{I}_2(T)=\int_0^T (h_n(t)-h(t))\psi(t)\,dt  \rightarrow 0$.  Therefore, for any $\phi \in W_0^{1,p}$, we have the following convergence result.
 \begin{align}
 \lim_{n\goto \infty} \int_{0}^{T} \left<\sigma(u_n(t)){h}_n(t) - \sigma(\bar{u}_*(t))h(t) , \phi \right> \, dt =0\,. \label{con:sequence-3}
 \end{align}
 Notice that \eqref{con:sequence-3} still holds if one replaces $\phi(x)$  by $\phi(x)\xi(t)$ for any $\xi \in \mathcal{D}(0,T)$.  Following the proof of Lemma \ref{eq:cgseq1}, we arrive at the following conclusion:
 \begin{align}\label{eq:weak-form-limit-function-c2}
 \begin{cases}
 \text{For all $t\in [0,T]$, $u_n(t) \rightharpoonup  \bar{u}_*(t)$ in $L^2$}, \\
 \text{ For all $t\in [0,T]$, $\bar{u}_*$ solves the equation}\\
 $$
 \displaystyle  {\tt B}(\bar{u}_*(t))= {\tt B}(u_0) +  \int_{0}^{t}  {\rm div}_x \vec{\bar{F}}_*(s) ds + \int_{0}^{t} \sigma(\bar{u}_*(s))h(s) ds \quad \text{in $L^2$ }\,, $$
 \end{cases}
 \end{align}
 where $\vec{\bar{F}}_*\in L^{p^\prime}(D_T)^d$ such that ${\tt A}(\nabla u_n) \rightharpoonup \vec{\bar{F}}_*$ in $ L^{p^\prime}(D_T)^d$.
 \vspace{0.15cm}
 
 \noindent {\bf Step-III:} \underline{Identification of $\vec{\bar{F}}_*$: } We claim that
 \begin{align*}
 \vec{\bar{F}}_*= {\tt A}(\nabla \bar{u}_*) \quad \text{in}~~~ L^{p^\prime}(D_T)^d\,.
 \end{align*}
To prove it, we proceed as follows.  An application of chain-rule on ${\tt B}(u_n)$ resp. ${\tt B}(\bar{u}_*)$ in \eqref{eq:skeleton} resp. in \eqref{eq:weak-form-limit-function-c2} and integration by parts formula yields
 \begin{align}
   &  \frac{1}{2}  \left\{||{\tt B}(u_n(T))||_{L^2}^{2}  -   \left \|{\tt B}(\bar{u}_*(T)) \right \|_{L^2}^{2}\right\}+  \int_{D_T} \Big[ {\tt b}^\prime(u_n(t)){\tt A}(\nabla u_n(t))\cdot \nabla u_n(t)-  \vec{\bar{F}}_* \cdot \nabla {\tt B}(\bar{u}_*(t))\Big] \, dx \, dt  \notag \\
     & =  \int_{D_T} \Big[ \sigma({u}_n(t))h_{\tau}(t) {\tt B} (u_n(t)) - \sigma(\bar{u}_*(t))h(t) {\tt B}(\bar{u}_*(t))\Big]\,dx\, dt \equiv \mathcal{I}_3\,. \notag
     \end{align}
 We now show that $\mathcal{I}_3\goto 0$ as $n\goto \infty$.  For $ \displaystyle \bar{\psi}(t):= \int_{D} \sigma(\bar{u}_*(t,x)) \, \bar{u}_*(t,x) \,dx $, we see that
\begin{align*}
      \int_0^T |\bar{\psi}(t)|^2\,dt  \leq \int_{0}^{T} \bigg(\int_{D} |\sigma(\bar{u}_*|^2 \, dx  \bigg)  \bigg(\int_{D} |\bar{u}_*|^2 \, dx \bigg) \, dt \leq C \int_{0}^{T} \left \| \bar{u}_*(t) \right \|_{L^2}^{4} dt < \infty\,.
\end{align*}
 Since  $h_{n} \rightharpoonup h$  in $L^2([0,T],\mathbb{R})$, we have $\displaystyle \int_0^T h_n(t) \bar{\psi}(t)\,dt  \goto \int_0^T h(t)  \bar{\psi}(t) \,dt$. In other words,
\begin{align} \label{eq:lem4.4.1}
    \lim_{n\rightarrow \infty} \left | \int_{0}^{T} \left< \sigma(\bar{u}_*(t)) \big(h_{n}(t) - h(t) \big) , {\tt B}(\bar{u}_*(t)) \right> \, dt  \right | = 0.
\end{align}
 In view of the assumption \ref{A4}, \eqref{conv:c2-2}, \eqref{eq:ub}, \eqref{esti:sup-l2-limit-fun-c2-cond} and  \eqref{esti:boundeness-hn}, we get
\begin{align}  \label{eq:lem4.4.2}
 \left |\int_{0}^{T} \left< \sigma(u_n)h_n(t) , {\tt B}(u_n)-{\tt B}(\bar{u}_*)\right> \, dt \right | & \leq C \int_{0}^{T} \left \| \sigma(u_n)  \right \|_{L^2} \left | h_n \right |  \left \|u_n-\bar{u}_*  \right \|_{L^2} \, dt \notag \\
 & \leq \bigg( \int_{0}^{T} C \left \| u_n  \right \|_{L^2}^{2} \left \|u_n-\bar{u}_*  \right \|_{L^2}^{2} \, dt \bigg)^{\frac{1}{2}} \bigg( \int_{0}^{T} \left | h_n \right |^{2} \, dt  \bigg)^{\frac{1}{2}} \notag \\
 & \leq C_M \bigg[ \sup_{s \in [0,T]} \left \| {\tt B}(u_n(s))  \right \|_{L^2}^{2} \int_{0}^{T}  \left \|u_n-\bar{u}_*  \right \|_{L^2}^{2} \, dt \bigg]^{\frac{1}{2}} \rightarrow 0.
  \end{align}
 \begin{align}  \label{eq:lem4.4.3}
 \left |\int_{0}^{T} \left< \big(\sigma(u_n)-\sigma(\bar{u}_*)\big)h_n(t) , {\tt B}(\bar{u}_*)\right> \, dt \right | & \leq  C \int_{0}^{T} \left \| \sigma(u_n)-\sigma(\bar{u}_*)  \right \|_{L^2} \left | h_n \right |  \left \|\bar{u}_*  \right \|_{L^2} \, dt \notag \\
 & \leq \bigg( \int_{0}^{T} C \left \| u_n-\bar{u}_*  \right \|_{L^2}^{2} \left \|\bar{u}_*  \right \|_{L^2}^{2} \, dt \bigg)^{\frac{1}{2}} \bigg( \int_{0}^{T} \left | h_n \right |^{2} \, dt  \bigg)^{\frac{1}{2}} \notag \\
 & \leq C_M \bigg[ \sup_{s \in [0,T]} \left \| \bar{u}_*(s)  \right \|_{L^2}^{2} \int_{0}^{T}  \left \|u_n-\bar{u}_*  \right \|_{L^2}^{2} \, dt \bigg]^{\frac{1}{2}} \rightarrow 0.
\end{align}
Combining \eqref{eq:lem4.4.1}, \eqref{eq:lem4.4.2} and \eqref{eq:lem4.4.3}, we see that $\mathcal{I}_3\goto 0$ as $n\goto \infty$. Since $u_n(T) \rightharpoonup
\bar{u}_*(T)$ in $L^2$, we have
\begin{align}
    \underset{\tau}\limsup  \int_{D_T} {\tt b}^\prime(u_n(t)){\tt A}(\nabla u_n(t))\cdot \nabla u_n(t)\,dx\,dt  \leq  \int_{D_T}  \vec{\bar{F}}_* \cdot \nabla {\tt B}(\bar{u}_*(t))\Big] \, dx \, dt \,. \label{esti:vecF-identify-c2-2}
\end{align}
Rest of the proof follows from Subsection \ref{subsec:existence-skeleton} i.e., from \eqref{esti:5-new}, \eqref{esti:6-new} and ${\rm (M)}$-property of the operator ${\tt A}$.
\vspace{0.15cm}

\noindent{\bf Step-IV:} \underline{ Convergence of $u_n$ to $\bar{u}_*=u_h$ in $\mathcal{Z}$:} In view of \eqref{eq:weak-form-limit-function-c2} and the fact that
 $\vec{\bar{F}}_*= {\tt A}(\nabla \bar{u}_*)$ in $L^{p^\prime}(D_T)^d$~(cf.~ {\bf Step-III}), we see that $\bar{u}_*$ is a solution of \eqref{eq:skeleton}.
 Moreover, thanks to uniqueness of solution of skeleton equation ~(cf.~Section \ref{sec:uniqueness}), $\bar{u}_*=u_h$ is the unique  solution of \eqref{eq:skeleton}.
We now show that $u_n \goto u_h$ in $\mathcal{Z}$. To do so, we  apply chain-rule for the equation satisfied by $u_n-u_h$ and use the integration by parts formula, the assumption \ref{A3} and \eqref{A:monotonocity} to have
\begin{align}
  &  \frac{d}{dt} \left \| {\tt B}(u_n(t)) -{\tt B}(u_h(t)) \right \|_{L^2}^{2} \notag \\
  & = -2\left<{\tt A}(\nabla u_n) -{\tt A}(\nabla u_h), \nabla {\tt B}(u_n)- \nabla {\tt B}(u_h) \right>  
     + 2 \left< \sigma(u_n)h_n(t)-\sigma(u_h)h(t) , {\tt B}(u_n)-{\tt B}(u_h)\right> \notag  \\
    & \leq - 2 \int_D \big(  {\tt A}(\grad u_n) - {\tt A}(\grad u_h)\big)
 \cdot   \grad u_h \big( {\tt b}^\prime(u_n)-{\tt b}^\prime(u_h)\big)+ 2 \left< \sigma(u_n)h_n(t)-\sigma(u_h)h(t) ,{\tt B}(u_n)-{\tt B}(u_h)\right>. \notag 
\end{align}
 Integrating over time, we get
\begin{align}
   & \sup_{s \in [0,T]} \left \| {\tt B}(u_n(s))-{\tt B}(u_h(s)) \right \|_{L^2}^{2}  \notag \\  
   & \leq - 2 \int_{D_T} \big(  {\tt A}(\grad u_n) - {\tt A}(\grad u_h)\big)
 \cdot   \grad u_h {\tt b}^\prime(u_n) \,dx\,dt 
 + 2  \int_{D_T} \big(  {\tt A}(\grad u_n) - {\tt A}(\grad u_h)\big)\cdot \nabla u_h{\tt b}^\prime(u_h)\,dx\,dt \notag \\
 & \quad + 2 \int_{0}^{T} \left< \sigma(u_n)h_n(t) , {\tt B}(u_n)-{\tt B}(u_h)\right>\,dt
     - 2 \int_{0}^{T} \left< \sigma(u_h)h(t) ,{\tt B}(u_n)-{\tt B}(u_h)\right>\,dt\equiv \sum_{i=4}^7 \mathcal{I}_i\,.
     \label{inq:con-det-}
\end{align}
Like  $\mathcal{A}_3$, we easily see that $\mathcal{I}_4\goto 0$ as $n\goto \infty$. Moreover, since ${\tt A}(\nabla u_n)  \rightharpoonup {\tt A}(\nabla u_h) $ in 
$L^{p^\prime}(D_T)^d$ and $\nabla u_h{\tt b}^\prime(u_h) \in L^p(D_T)^d$, one has $\mathcal{I}_5\goto 0$ as $n\goto \infty$. Furthemore,  thanks to Cauchy-Schwartz inequality, \eqref{B:reverse-L2}, \eqref{esti:B-Lipschitz},  \eqref{esti:boundeness-hn}  and \eqref{conv:c2-2}, one has
\begin{align*}
\mathcal{I}_6 & \leq \int_{0}^{T} \left \| \sigma(u_n)  \right \|_{L^2} \left | h_n \right |  \left \| {\tt B}(u_n)-{\tt B}(u_h)  \right \|_{L^2} \, dt \\
 & \leq C \bigg( \int_{0}^{T}  \left \| u_n  \right \|_{L^2}^{2} \left \| {\tt B}(u_n)-{\tt B}(u_h)  \right \|_{L^2}^{2} \, dt \bigg)^{\frac{1}{2}} \bigg( \int_{0}^{T} \left | h_n \right |^{2} \, dt  \bigg)^{\frac{1}{2}}  \\
 & \leq C_M \bigg\{ \sup_{s \in [0,T]} \left \| {\tt B}(u_n(s))  \right \|_{L^2}^{2} \int_{0}^{T}  \left \|u_n-u_h  \right \|_{L^2}^{2} \, dt \bigg\}^{\frac{1}{2}} \rightarrow 0\,. \notag 
 \end{align*}
 Similarly, we have
 \begin{align*}
   \mathcal{I}_7 & \leq C_M \bigg( \int_{0}^{T}  \left \|u_n-u_h  \right \|_{L^2}^{2} \, dt \bigg)^{\frac{1}{2}} \rightarrow 0\,.
\end{align*}
 Combining these results in \eqref{inq:con-det-}, we have
 \begin{align}
 \sup_{s \in [0,T]} \left \| {\tt B}(u_n(s))-{\tt B}(u_h(s)) \right \|_{L^2}^{2} \goto 0  \quad \text{as $n\goto \infty$}. \label{conv:c2-final-1}
 \end{align}
 Hence the assertion follows from \eqref{conv:c2-final-1} once we use \eqref{B:reverse-poinwise}. This completes the proof of condition \ref{C2}.

\subsection{Proof of condition \ref{C1}}\label{subsec:cond-c1}
In this subsection, we prove condition \ref{C1}. To proceed further, we consider the following equation: for $\eps>0$
\begin{align}
    d {\tt B}(v^{\eps})  - {\rm div}_x {\tt A}( \nabla v^{\eps}) \, dt &= \sigma(v^{\eps})h(t)\,dt  + \sqrt{\eps}\, \sigma(v^{\eps})\, dW(t)\,, \label{eq:epsilon-main}
\end{align}
where $ \{h^\eps\} \subset \mathcal{A}_M$ converges to $h$ in distribution as $S_{M}$-valued random variables.  An application of Girsanov theorem and  the uniqueness of equation \eqref{eq:ldp}, we conclude that the equation \eqref{eq:epsilon-main} has unique solution $ \displaystyle v^{\eps} = \mathcal{G}^{\eps}\left( W(\cdot) + \frac{1}{{\sqrt{\eps}}} \int_0^{\cdot} h^{\eps}(s) \, ds \right) $. Therefore,  to  prove condition \ref{C1}, it is sufficient to show that $v^{\eps} \rightarrow u_h $ in distribution  on $\mathcal{Z}$ as $\eps \rightarrow 0$, where $u_h$ is the unique solution of the Skeleton equation \eqref{eq:skeleton}. 
 \vspace{0.2cm}
 
 \noindent\underline{\bf Step A: a-priori estimates.}  We assume that $ 0<\eps \leq 1$. We first  derive necessary uniform estimates of $v^{\eps}$ in the upcoming Lemmas \ref{lem:est2.1} and \ref{lem:est2.2}.
\begin{lem} There exists a constant $C>0$ such that 
\label{lem:est2.1}
 $$ \sup_{\eps \in (0,1]} \mathbb{E} \bigg[ \sup_{s \in [0,T]} \left \| {\tt B}(v^{\eps}(s)) \right \|_{L^2}^{2} + \int_{0}^{T}  \left \| {\tt B}( v^{\eps}(s)) \right \|_{W_{0}^{1,p}}^{p} ds \bigg] \leq C. $$   
\end{lem}

\begin{proof}
By  applying It$\hat{o}$ formula to the functional $x\mapsto \|x\|_{L^2}^2$, integration by parts formula, \eqref{esti:2} and Cauchy-Schwartz inequality, we have 

\begin{align*}
     \left \| {\tt B}(v^{\eps}(t)) \right \|_{L^2}^{2}  = &  \left \| {\tt B}(u_0) \right \|_{L^2}^{2} -2 \int_0^t \left<{\tt A}(\nabla v^{\eps}(s)), {\tt B}(v^{\eps}(s)) \right>\,ds + 2 \int_0^t  \left< \sigma(v^{\eps}(s)) \,
     h^{\eps}(s), {\tt B}(v^{\eps}(s))\right>\,ds \notag \\
    &  \qquad + \eps \int_0^t  \left \| \sigma(v^{\eps}(s))  \right \|_{L^2}^{2}\,ds  + 2 \sqrt{\eps} \int_0^t  \left<{\tt B}(v^{\eps}(s)), \sigma(v^{\eps}(s)) \, dW(s) \right> \notag \\
 & \leq \left \| {\tt B}(u_0) \right \|_{L^2}^{2} -2C_3C_1 \int_0^t\|\nabla v^\eps(s)\|_{L^p}^p\,ds + 2C_4 \|K_1\|_{L^1}\,t \notag \\
   &\quad   + C \int_0^t \|v^\eps(s)\|_{L^2} \|{\tt B}(v^\eps(s))\|_{L^2} |h^\eps(s)|\,ds + C\eps \int_0^t  \|v^\eps(s)\|_{L^2}^2\,ds \notag \\
   & \qquad  + 2 \sqrt{\eps} \int_0^t  \left<{\tt B}(v^{\eps}(s)), \sigma(v^{\eps}(s)) \, dW(s) \right> \notag \\
   & \le \left \| {\tt B}(u_0) \right \|_{L^2}^{2} -2C_3C_1 \int_0^t\|\nabla v^\eps(s)\|_{L^p}^p\,ds + 2C_4 \|K_1\|_{L^1} t \notag \\
   & \quad + C \int_0^t \left( \eps + |h^\eps(s)| \right)  \left \| {\tt B}(v^{\eps}(s)) \right \|_{L^2}^{2}\,ds
    + 2 \sqrt{\eps} \int_0^t  \left<{\tt B}(v^{\eps}(s)), \sigma(v^{\eps}(s)) \, dW(s) \right> \,,
\end{align*}
where in the last inequality we have used \eqref{B:reverse-L2}. 
 Thus, by Poincare inequality, we obtain
\begin{align}
  & \displaystyle  \left \| {\tt B}(v^{\eps}(t)) \right \|_{L^2}^{2} + C \int_{0}^{t} \left \|  v^{\eps}(s) \right \|_{W_{0}^{1,p}}^{p} ds \notag \\
   & \leq \left \| {\tt B}(u_0) \right \|_{L^2}^{2}  + C \|K_1\|_{L^1} t+ C \int_0^t \left( \eps + |h^\eps(s)| \right)  \left \| {\tt B}(v^{\eps}(s)) \right \|_{L^2}^{2}\,ds \notag \\
   & \qquad +  2 \sqrt{ \eps} \int_0^t  \left<{\tt B}(v^{\eps}(s)), \sigma(v^{\eps}(s)) \, dW(s) \right>. \label{eq:est}
\end{align} 
Discarding nonnegative term and taking the expectation, we get 
\begin{align} \label{eq:est2.1.0}
  \mathbb{E} \bigg[\sup_{s \in [0,t]} \left \| {\tt B}(v^{\eps}(s)) \right \|_{L^2}^{2} \bigg]  & \leq 
  \left \| {\tt B}(u_0) \right \|_{L^2(D)}^{2}  + C \|K_1\|_{L^1} t+ C \int_0^t \left( \eps + |h^\eps(s)| \right) 
   \left \| {\tt B}(v^{\eps}(s)) \right \|_{L^2}^{2}\,ds \notag \\
  & \quad \quad +  2 \, \sqrt{\eps} \, \mathbb{E} \Bigg[\sup_{s \in [0,t]} \left| \int_{0}^{s}   \left< {\tt B}(v^{\eps}(r)), \sigma(v^{\eps}(r)) \, dW(r) \right> \right|  \Bigg]. 
\end{align}
By applying the Burkh$\ddot{o}$lder-Davis-Gundy inequality,  the assumption \ref{A4}, \eqref{B:reverse-L2} and Young's inequality, one has
\begin{align} \label{eq:est2.1.1}
    & 2 \, \sqrt{\eps} \, \mathbb{E} \Bigg[\sup_{s \in [0,t]} \left| \int_{0}^{s}   \left<{\tt B}(v^{\eps}(r)), \sigma(v^{\eps}(r)) \, dW(r) \right> \right|  \Bigg]  \notag \\
     & \leq 2 \, \sqrt{\eps} \, \mathbb{E} \Bigg[\left( \int_{0}^{t} \left \| {\tt B}(v^{\eps}(s)) \right \|_{L^2}^{2} \left \| \sigma(v^{\eps}(s)) \right \|_{L^2}^{2} ds \right)^{\frac{1}{2}} \Bigg]\notag \\
     & \leq 2 \, \sqrt{\eps} \, \mathbb{E} \Bigg[ \left(\sup_{s \in [0,t]}\left \| {\tt B}(v^{\eps}(s)) \right \|_{L^2}^{2} C \int_{0}^{t}  \left \| {\tt B}(v^{\eps}(s))\right \|_{L^2}^{2} ds \right)^\frac{1}{2} \Bigg] \notag \\
     & \leq \sqrt{\eps} \, \delta \, \mathbb{E} \bigg[\sup_{s \in [0,t]}\left \| {\tt B}(v^{\eps}(s)) \right \|_{L^2}^{2} \bigg] + \sqrt{\eps} \, C_{\delta^2 } \, \mathbb{E} \Bigg[ \int_{0}^{t}  \left \| {\tt B}(v^{\eps}(s)) \right \|_{L^2}^{2} ds  \Bigg]\,
 \end{align}
 for some $\delta>0$. 
Applying Holder's inequality and Young's inequality, we obtain 
\begin{align} \label{eq:est2.1.2}
      \mathbb{E} \left[\int_{0}^{t} \left|h^{\eps} \right| \left \| {\tt B}(v^{\eps}(s)) \right \|_{L^2}^{2} \,ds \right] & \leq  \mathbb{E} \left[\bigg(\int_{0}^{t} \left|h^{\eps} \right|^2 \,ds \bigg)^{\frac{1}{2}} \bigg(\int_{0}^{t} \left \| {\tt B}(v^{\eps}(s)) \right \|_{L^2}^{4} \,ds \bigg)^{\frac{1}{2}} \right]\notag \\
     &\leq C_M  \mathbb{E} \Bigg[\sup_{s \in [0,t]}\left \| {\tt B}(v^{\eps}(s)) \right \|_{L^2} \left( \int_{0}^{t}  \left \| v^{\eps}(s) \right \|_{L^2}^{2} ds \right)^{\frac{1}{2}}\Bigg] \notag \\
      & \leq C_M \delta \, \mathbb{E} \bigg[\sup_{s \in [0,t]}\left \| {\tt B}(v^{\eps}(s)) \right \|_{L^2}^{2} \bigg] + C_M C_{\delta^2 } \mathbb{E} \Bigg[ \int_{0}^{t}  \left \| {\tt B}(v^{\eps}(s)) \right \|_{L^2}^{2} ds \Bigg].
\end{align}
Therefore using \eqref{eq:est2.1.1} and \eqref{eq:est2.1.2} in \eqref{eq:est2.1.0}, we have for all $\eps \in (0,1]$ 
\begin{align*}
     \mathbb{E} \bigg[\sup_{s \in [0,t]} \left \| {\tt B}(v^{\eps}(s)) \right \|_{L^2}^{2} \bigg] &
     \le    \left \| {\tt B}(u_0) \right \|_{L^2}^{2}  + C t\|K_1\|_{L^1}  + C(M, \delta) \mathbb{E}\left[ \int_{0}^{t}  \left \| {\tt B}(v^{\eps}(s)) \right \|_{L^2}^{2} ds\right] \notag \\ & \qquad + (1+C_M) \, \delta \, \mathbb{E} \bigg[\sup_{s \in [0,t]}\left \| {\tt B}(v^{\eps}(s)) \right \|_{L^2}^{2} \bigg]\,.
    \end{align*}
    Choosing $0<\delta < \frac{1}{1+ C_M}$ and then applying Gronwall's lemma, we have 
      \begin{align} \label{eq:est2.1.3}
        \mathbb{E} \bigg[\sup_{s \in [0,T]} \left \|{\tt B}( v^{\eps}(s)) \right \|_{L^2}^{2} \bigg] & \leq C. 
         \end{align}
      By using \eqref{eq:est2.1.2} and \eqref{eq:est2.1.3}, we have   from \eqref{eq:est}
        \begin{align}
        \mathbb{E} \Big[\int_{0}^{T}  \left \|  v^{\eps}(s) \right \|_{W_{0}^{1,p}}^{p} ds \Big] 
        &\leq \left \| {\tt B}(u_0) \right \|_{L^2}^{2}  + C t\|K_1\|_{L^1}  + C(M,T)\,  \mathbb{E} \bigg[\sup_{s \in [0,T]}\left \| {\tt B}(v^{\eps}(s)) \right \|_{L^2}^{2} \bigg]  \le C\,. \notag
\end{align}
In view of Poincare inequality, \eqref{esti:operator-B-grad} and above estimation, we conclude that
\begin{align}
   \mathbb{E} \Big[\int_{0}^{T}  \left \| {\tt B}( v^{\eps}(s)) \right \|_{W_{0}^{1,p}}^{p} ds \Big] \le C\,.  \label{eq:est2.1.4}
\end{align}
Hence the assertion follows from  \eqref{eq:est2.1.3} and \eqref{eq:est2.1.4}.
\end{proof}
 Due to its stochastic setup, a-priori bounds in Lemma \ref{lem:est2.1} is not sufficient to show convergence of $v^\eps$ to $u_h$. We need more regularity estimate of $v^\eps$---which is given by the following lemma. 
\begin{lem}\label{lem:est2.2}
Let $v^\eps$ be a solution of \eqref{eq:epsilon-main}. Then there exists a constant $C>0$, independent of $\eps>0$ such that
\begin{align}
   \sup_{\eps >0} \mathbb{E} \Bigg[\sup_{s \in [0,T]}  \left \| {\tt B}(v^{\eps}(s)) \right \|_{L^2}^{4} + \Big(\int_{0}^{T}  \left \|  v^{\eps}(s) \right \|_{W_{0}^{1,p}}^{p} ds \Big)^2\Bigg] \leq C. \label{esti:regularity-cond-c2}
   \end{align}
   \end{lem}
\begin{proof}
First we prove that 
\begin{align} \label{eq:est2.2.1}
 \sup_{s \in [0,T]}  \mathbb{E} \Big[ \left \|{\tt B}( v^{\eps}(s)) \right \|_{L^2}^{4} \Big] & \leq C. 
\end{align}
Squaring both sides of \eqref{eq:est} and taking expectation, we have
\begin{align}
\mathbb{E}\left[ \left \| {\tt B}(v^{\eps}(t)) \right \|_{L^2}^{4}\right] & \leq 4\left( \,\left \| {\tt B}(u_0) \right \|_{L^2}^{4}  + C t^2 \|K_1\|_{L^1}^2\right)+ C
   \mathbb{E}\left[ \bigg(\int_{0}^{t} (\left|h^{\eps} \right|+\eps) \left \| {\tt B}(v^{\eps}(s)) \right \|_{L^2}^{2} ds \bigg)^{2}\right]  \notag \\
    & \qquad +  16 \, \eps\, \mathbb{E}\left[ \bigg(\int_{0}^{t}   \left<{\tt B}(v^{\eps}(s)), \sigma(v^{\eps}(s)) \, dW(s) \right> \bigg)^{2}\right]
    \equiv \mathcal{I}_8 + \mathcal{I}_9 + \mathcal{I}_{10}\,. \label{eq:for-cond-c2-1}
    \end{align}
    Applying Holder's  and Jensen's inequalities and using the fact that $h^{\eps} \in \mathcal{A}_M$, we have
    \begin{align}
    \label{eq:est2}
\mathcal{I}_9  \leq C \,\mathbb{E} \Bigg[\bigg (\int_{0}^{t} \left|h^{\eps} \right|^{2} ds +1 \bigg) \bigg (\int_{0}^{t} \left \| {\tt B}(v^{\eps}(s)) \right \|_{L^2}^{4} ds \bigg) \Bigg]  
 \leq C_M \,\mathbb{E} \Bigg[\int_{0}^{t} \left \| {\tt B}(v^{\eps}(s)) \right \|_{L^2}^{4} ds \Bigg]. 
\end{align}
We use Ito-isometry, the assumption \ref{A3} and  \eqref{B:reverse-L2}  to estimate $\mathcal{I}_{10}$ as 
\begin{align} \label{eq:est2.2.2}
 \mathcal{I}_{10} = 16\, \eps\, \mathbb{E}\left[ \int_0^t \left<{\tt B}(v^{\eps}(s)), \sigma(v^{\eps}(s)) \right> ^2\,ds\right]  \le C \mathbb{E} \Bigg[\int_{0}^{t} \left \|{\tt B}( v^{\eps}(s)) \right \|_{L^2}^{4} ds \Bigg].
\end{align}
Combining  \eqref{eq:est2} and \eqref{eq:est2.2.2} in \eqref{eq:for-cond-c2-1}, we obtain
\begin{align*}
   \mathbb{E} \Big[  \left \| {\tt B}(v^{\eps}(t)) \right \|_{L^2}^{4} \Big] \leq C\Bigg\{ \,\left \| u_0 \right \|_{L^2}^{4}  + T^2 \|K_1\|_{L^1}^2+ \mathbb{E} \Big[\int_{0}^{t} \left \| {\tt B}(v^{\eps}(s)) \right \|_{L^2}^{4} ds \Big]\Bigg\}. 
\end{align*}
Hence, an application of Gronwall's lemma yields the result \eqref{eq:est2.2.1}.
\vspace{0.2cm}

Again from \eqref{eq:est} together with \eqref{eq:est2} and \eqref{eq:est2.2.2}, we have 
\begin{align}
    \mathbb{E} \Bigg[ \Big(\int_{0}^{T}  \left \|  v^{\eps}(s)  \ \right \|_{W_{0}^{1,p}}^{p} ds\Big)^2 \Bigg] & 
    \le C \Bigg\{  \left \| u_0 \right \|_{L^2}^{4} + T^2 \|K_1\|_{L^1}^2+ \int_{0}^{T} \mathbb{E}\Big[ \left \| {\tt B}(v^{\eps}(s)) \right \|_{L^2}^{4}\Big] ds \Bigg\} \notag \\
    & \le  C \Bigg\{  \left \| u_0 \right \|_{L^2}^{4} + T^2 \|K_1\|_{L^1}^2 + \sup_{t \in [0,T]} \mathbb{E} \Big[  \left \| {\tt B}(v^{\eps}(t)) \right \|_{L^2}^{4} \Big]\Bigg\}\le C\,, \notag
\end{align}
where in the last inequality we have used the uniform estimate \eqref{eq:est2.2.1}. Thus, thanks to Poincare inequality and \eqref{esti:operator-B-grad}, we get
\begin{align}
   \mathbb{E} \Bigg[ \Big(\int_{0}^{T}  \left \|  {\tt B}(v^{\eps}(s)) \ \right \|_{W_{0}^{1,p}}^{p} ds\Big)^2 \Bigg] \le C    \mathbb{E} \Bigg[ \Big(\int_{0}^{T}  \left \|  v^{\eps}(s)  \ \right \|_{W_{0}^{1,p}}^{p} ds\Big)^2 \Bigg] \le C\,.  \label{esti:regularity-sobolev-cond-c1}
\end{align}

Again squaring the both sides of \eqref{eq:est}, we have
\begin{align}
\mathbb{E}\left[ \sup_{s\in [0,t]} \left \| {\tt B}(v^{\eps}(s)) \right \|_{L^2}^{4}\right] & \leq 4\left( \,\left \| {\tt B}(u_0) \right \|_{L^2}^{4}  + C t^2 \|K_1\|_{L^1}^2\right)+ C
   \mathbb{E}\left[ \bigg(\int_{0}^{t} (\left|h^{\eps} \right|+\eps) \left \| {\tt B}(v^{\eps}(s)) \right \|_{L^2}^{2} ds \bigg)^{2}\right]  \notag \\
    & \qquad +  16 \, \eps\, \mathbb{E}\left[  \sup_{s\in [0,t]}\bigg(\int_{0}^{s}   \left<{\tt B}(v^{\eps}(r)), \sigma(v^{\eps}(r)) \, dW(r) \right> \bigg)^{2}\,\right] \notag \\
    &  \equiv \mathcal{I}_8 + \mathcal{I}_9 + \mathcal{I}_{11}\,. \label{eq:for-cond-c2-1-1}
    \end{align}
From \eqref{eq:est2}, we bound $\mathcal{I}_9$ as
\begin{align}
\mathcal{I}_9 \le C \int_0^t \mathbb{E}\left[ \sup_{r\in [0,s]}\|{\tt B}(v^\eps(r))\|_{L^2}^4\right]\,ds\,. \label{esti:regu-cond-c1-drift}
\end{align}
Thanks to Burkh$\ddot{o}$lder-Davis-Gundy inequality and \eqref{eq:est2.2.2}, we have
\begin{align}
\mathcal{I}_{11}\le C  \mathbb{E}\left[ \int_0^t \left<{\tt B}(v^{\eps}(s)), \sigma(v^{\eps}(s)) \right> ^2\,ds\right] \le C 
 \int_0^t \mathbb{E}\left[ \sup_{r\in [0,s]}\|{\tt B}(v^\eps(r))\|_{L^2}^4\right]\,ds\,. \label{esti:regu-cond-c1-diffusion}
\end{align}
We use \eqref{esti:regu-cond-c1-drift} and \eqref{esti:regu-cond-c1-diffusion} in \eqref{eq:for-cond-c2-1-1} and apply Gronwall's lemma to conclude
\begin{align}
\mathbb{E}\left[ \sup_{s\in [0,T]} \left \| {\tt B}(v^{\eps}(s)) \right \|_{L^2}^{4}\right] \le C\,.  \label{esti:regularity-l2-4-sup-cond-c1}
\end{align}
Hence the estimate \eqref{esti:regularity-cond-c2} follows from \eqref{esti:regularity-sobolev-cond-c1} and \eqref{esti:regularity-l2-4-sup-cond-c1}. This completes the proof. 
\end{proof}
Like its deterministic counterpart, we need to derive uniform bound of ${\tt B}(v^\eps)$ in some appropriate fractional Sobolev space.
\begin{lem} \label{lem:cpt2}
The following estimation holds: for any $\alpha \in (0,\frac{1}{2p})$,
$$\sup_{\eps} \mathbb{E}\Big[ \left\| {\tt B}(v^{\eps})\right\|_{W^{\alpha,2}([0,T];{W^{-1,p'}})}\Big] < \infty. $$
\end{lem}

\begin{proof}
 Since $v^\eps$ is the unique solution of \eqref{eq:epsilon-main}, we have $\mathbb{P}$-a.s., and for all $t\in [0,T]$
\begin{align*}
    {\tt B}(v^{\eps}(t)) &= {\tt B}(u_0) + \int_{0}^{t} {\rm div}_x{\tt A} (\nabla v^{\eps}) \, ds
     + \int_{0}^{t} \sigma(v^{\eps})h^\eps(s) \, ds + \sqrt{\eps} \int_{0}^{t} \sigma(v^{\eps}) \, dW(s)  \notag\\
    &\equiv \mathcal{J}^\eps_0(t)+ \mathcal{J}_{1}^{\eps}(t) + \mathcal{J}_{2}^{\eps}(t) + \mathcal{J}_{3}^{\eps}(t) \,.
\end{align*}
By Sobolev embedding $W_{0}^{1,p} \hookrightarrow W^{-1,p'}$ and Lemma \ref{lem:est2.1} , we get
\[ \mathbb{E} \Big[ \int_{0}^{T} \left \|{\tt B}(v^{\eps})  \right \|_{W^{-1,p'}}^{p} \,dt\Big] \leq \mathbb{E}\Big[ \int_{0}^{T} \left \|{\tt B}(v^{\eps})  \right \|_{W_{0}^{1,p}}^{p} \,dt \Big]\leq C. \]
Therefore, to complete the proof, we need to show that for any $\alpha \in (0,\frac{1}{2p})$,
\begin{align}
\mathbb{E}\Big[  \int_{0}^{T}\int_{0}^{T} \frac{\left \|\mathcal{J}_{i}^{\eps}(t) - \mathcal{J}_{i}^{\eps}(s)   \right \|_{W^{-1,p'}}^{2}}{\left | t-s \right |^{1+ 2 \alpha}} \, dt \, ds\Big] \leq C \quad (0\le i\le 3)\,. \label{inq:frac-Sov-main-}
\end{align}
Since $\mathcal{J}_{0}^\eps(t)$ is independent of time, \eqref{inq:frac-Sov-main-}  is satisfied by $\mathcal{J}_{0}^\eps(t)$ for any  $\alpha \in (0,1)$.
\vspace{0.2cm}

W.L.O.G.\, we assume that $s < t$. By employing \eqref{A:boundedness} and H\"{o}lder's inequality, one has
\begin{align}
  &  \mathbb{E}\Big[ \left \|\mathcal{J}_{1}^{\eps}(t) - \mathcal{J}_{1}^{\eps}(s)   \right \|_{W^{-1,p'}}^{2}\Big]  
  = \mathbb{E} \left[ \left \| \int_{s}^{t} {\rm div}_x{\tt A}( \nabla v^{\eps} )\,dr  \right \|_{W^{-1,p'}}^{2} \right] \notag \\
  &\leq \mathbb{E} \left[ \bigg(\int_{s}^{t} \left \| {\rm div}_x {\tt A}(\nabla v^{\eps})  \right \|_{W^{-1,p'}} \,dr\bigg)^{2}\right] 
     \leq \mathbb{E}\left[ \bigg(\int_{s}^{t} \left \|{\tt A}(\nabla v^{\eps})  \right \|_{L^{p^\prime}} \,dr\bigg)^2\right] \notag \\
   &  \le \mathbb{E}\left[  \bigg( C_2 \int_s^t \|\nabla v^\eps(r)\|_{L^p}^{p-1}\,dr + (t-s)\|K_2\|_{L^{p^\prime}} \bigg)^{2}\right] \notag \\
    &  \le C\mathbb{E}\left[ \left\{ (t-s)^\frac{1}{p}\Big( \int_s^t  \|\nabla v^\eps(r)\|_{L^p}^{p}\,dr\Big)^\frac{p-1}{p} + (t-s)\|K_2\|_{L^{p^\prime}} \right\}^{2}\right] \notag \\
    &  \le C\mathbb{E}\left[  (t-s)^\frac{2}{p}\Big( \int_s^t  \|v^\eps(r)\|_{W_0^{1,p}}^{p}\,dr\Big)^\frac{2(p-1)}{p} + (t-s)^2\|K_2\|_{L^{p^\prime}}^2 \right] \notag \\
     &  \le C\mathbb{E}\left[  (t-s)^\frac{2}{p} \left( 1+ \Big( \int_0^T  \|v^\eps(r)\|_{W_0^{1,p}}^{p}\,dr\Big)^2\right)+ (t-s)^2\|K_2\|_{L^{p^\prime}}^2 \right] \notag \\
    & \leq C (t-s)^{\frac{2}{p}} \left\{1 +\mathbb{E} \left[\bigg( \int_{0}^{T} \left \|v^{\eps}  \right \|_{W_{0}^{1,p}}^{p} \,dr\bigg)^{2}\right] + \|K_2\|_{L^{p^\prime}}^2 \right\}\,.
\end{align}
Thanks to Lemma \ref{lem:est2.2}, we get
\begin{align}
   \mathbb{E} \left[\int_{0}^{T}\int_{0}^{T} \frac{\left \|\mathcal{J}_{1}^{\eps}(t) - \mathcal{J}_{1}^{\eps}(s)   \right \|_{W^{-1,p'}}^{2}}{\left | t-s \right |^{1+ 2 \alpha}} \, dt \, ds\right] \leq C, \quad \, \forall \, \alpha \in (0,\frac{1}{2p}).  \notag
\end{align}
For $\mathcal{J}_{2}^{\eps}$, we estimate as follows. In view of the assumption \ref{A4}, Jensen's inequality, \eqref{B:reverse-L2} and Lemma \ref{lem:est2.1}, 
\begin{align}
    \mathbb{E}\left[ \left \|\mathcal{J}_{2}^{\eps}(t) - \mathcal{J}_{2}^{\eps}(s)   \right \|_{L^{2}}^{2}\right] & \leq  C (t-s) \mathbb{E} \left[\int_{s}^{t}  \left \| v^{\eps} \right \|_{L^2}^{2} \left | h^{\eps}(r) \right |^{2} \,dr \right] \notag \\
    & \leq C (t-s) \, \mathbb{E} \bigg[ \sup_{0\leq t\leq T}\left \|{\tt B}( v^{\eps}(t)) \right \|_{L^2}^{2} \int_{0}^{T}  \left | h^{\eps}(r) \right |^{2} \,dr \bigg] \notag \\
    & \leq C (t-s) \, \mathbb{E} \Big[\sup_{0\leq t\leq T}\left \| {\tt B}(v^{\eps}(t)) \right \|_{L^2}^{2} \Big], \notag
\end{align}
where in the last inequality we have used that $h^\eps \in \mathcal{A}_M$. Thus, since $L^2 \hookrightarrow W^{-1,p'}$, \eqref{inq:frac-Sov-main-} is satisfied by $\mathcal{J}_{2}^{\eps}$ for $\alpha \in (0,\frac{1}{p})$. 
\vspace{0.2cm}

Ito-isometry together with the assumption \ref{A4}, \eqref{B:reverse-L2} and Lemma \ref{lem:est2.1} yields that 
\begin{align}
    \mathbb{E}\left[ \left \|\mathcal{J}_{3}^{\eps}(t) - \mathcal{J}_{3}^{\eps}(s)   \right \|_{L^{2}}^{2}\right] &  \leq C \mathbb{E} \bigg[ \int_{s}^{t}  \left \| v^{\eps}(r) \right \|_{L^2}^{2} \,dr \bigg]  \leq C \mathbb{E} \bigg[ \int_{s}^{t}  \left \|{\tt B}( v^{\eps}(r)) \right \|_{L^2}^{2} \,dr \bigg] \notag \\
    &\le  C (t-s)\, \mathbb{E} \Big[\sup_{0\leq t\leq T}\left \| {\tt B}(v^{\eps}(t)) \right \|_{L^2}^{2} \Big]\le C(t-s)\,.\notag
\end{align}
Hence,  for any $\alpha \in (0,\frac{1}{p})$ the estimation \eqref{inq:frac-Sov-main-} holds true for $\mathcal{J}_{3}^{\eps}$. 
\end{proof}
\noindent\underline{\bf Step B: convergence analysis.} In this step, we discuss about the tightness of the family $\{\mathcal{L}({\tt B}(v^\eps))$ and 
to get almost sure convergence on a new probability space. 
\begin{cor} The sequence 
$\left \{ \mathcal{L}({\tt B}(v^{\eps})) \right \}$ is tight on $L^{2}(D_T)$.
\end{cor}
\begin{proof}
  By Lemma \ref{lem:cpt}, we get $ L^{2}([0,T];W_{0}^{1,p}) \, \cap \, W^{\alpha,2}([0,T];{W^{-1,p'}})$ is compactly embedded in $L^{2}(D_T)$. So for any $N > 0$, the set
  $$L_N = \left \{ v \in L^{2}(D_T) : \left \| v \right \|_{L^{2}([0,T];W_{0}^{1,p})} + \left \|v  \right \|_{W^{\alpha,2}([0,T];{W^{-1,p'}})} \leq N \right \} $$
  is a compact subset of $L^{2}(D_T)$. Observe that, thanks to Markov's inequality and Lemmas \ref{lem:est2.1} and  \ref{lem:cpt2}
  \begin{align}
      & \underset{N \rightarrow \infty}\lim \, \underset{\eps >0}\sup \ \mathbb{P}( {\tt B}(v^{\eps}) \notin L_N) \notag \\
      = & \underset{N \rightarrow \infty}\lim \, \underset{\eps >0}\sup \ \mathbb{P} \left (\left \|{\tt B}( v^{\eps}) \right \|_{L^{2}([0,T];W_{0}^{1,p})} + \left \|{\tt B}(v^{\eps} ) \right \|_{W^{\alpha,2}([0,T];{W^{-1,p'}})} > N \right) \notag \\
       = & \underset{N \rightarrow \infty}\lim \, \frac{1}{N} \ \underset{\eps>0}\sup \ \mathbb{E} \left[\left \| {\tt B}(v^{\eps}) \right \|_{L^{2}([0,T];W_{0}^{1,p})} + \left \| {\tt B}(v^{\eps} ) \right \|_{W^{\alpha,2}([0,T];{W^{-1,p'}})}\right]=0\,. \notag
  \end{align}
  Therefore, $\left \{ \mathcal{L}({\tt B}(v^{\eps})) \right \}$ is tight on $L^{2}(D_T)$.
\end{proof}

Since $S_M$ is a Polish space, we apply Prokhorov compactness theorem to the family of laws $\left \{ \mathcal{L}({\tt B}(v^{\eps})) \right \}$ and the modified version of  Skorokhod representation theorem \cite[Theorem $C.1$]{Hausenblus-2018} to have  existence of a new probability space $(\overline{\Omega},\overline{\mathcal{F}},\overline{\mathbb{P}})$, a subsequence of $\{\eps\}$, still denoted by $\left \{\eps  \right \}$, and random variables $(\overline{\tt B}^{\eps}, \overline{h}^{\eps},\overline{W}^{\eps})$ and $( \overline{\tt B}^*, h, \overline{W}^*)$  taking values in $L^{2}(D_T) \times S_M \times \mathcal{C}([0,T]; \mathbb{R})$ such that
\begin{itemize}
    \item[i)]  $\mathcal{L}(\overline{\tt B}^{\eps}, \overline{h}^{\eps},\overline{W}^{\eps}) = \mathcal{L}({\tt B}({v}^{\eps}),{h}^{\eps}, W)$ for all $\eps > 0$, 
    \item[ii)] $(\overline{\tt B}^{\eps}, \overline{h}^{\eps},\overline{W}^{\eps}) \rightarrow (\overline{\tt B}^*, h, \overline{W}^*)$ in $L^{2}(D_T) \times S_M \times \mathcal{C}([0,T]; \mathbb{R})$\, $\overline{\mathbb{P}}$-a.s. \,as $\eps \rightarrow 0$,
    \item[iii)] $\overline{W}^{\eps}(\overline{\omega}) = \overline{W}^*(\overline{\omega}) $ for all $\overline{\omega} \in \overline{\Omega}$.
\end{itemize}
Moreover, \begin{align}
  \overline{\tt B}^\eps={\tt B}(v^\eps)\circ\phi^\eps\,, \quad \overline{h}^\eps= h^\eps \circ \phi^\eps \,, \quad \mathbb{P}=\overline{\mathbb{P}}\circ (\phi^\eps)^{-1}\,. \label{eq:perfect-function}
 \end{align}
 for some sequence of perfect functions $\phi^\eps:\overline{\Omega}\to\Omega$, see e.g., \cite[Theorem $1.10.4$ \&\ Addendum $1.10.5$]{wellner}.
Furthermore, $\overline{W}^{\eps}$ and $\overline{W}^*$ are one dimensional Brownian motion over the stochastic basis 
$\big(\overline{\Omega},\overline{\mathcal{F}},\overline{\mathbb{P}}, \{\overline{\mathcal{F}}_t\}\big) $, where $ \{\overline{\mathcal{F}}_t\}$ is the natural filtration of 
$(\overline{\tt B}^{\eps}, \overline{h}^{\eps},\overline{W}^{\eps},\overline{\tt B}^*, h, \overline{W}^*)$.

Define, for $\eps>0$
\begin{align*}
\overline{v}^\eps= v^\eps \circ\phi^\eps\,.
\end{align*}
Since $v^\eps$ is the solution of  \eqref{eq:epsilon-main},  by using \eqref{eq:perfect-function} and  the equality of laws as stated in ${\rm i)}$, one can easily conclude that $\overline{v}^{\eps}$ is the solution to equation \eqref{eq:epsilon-main} on the stochastic basis $\big(\overline{\Omega},\overline{\mathcal{F}},\overline{\mathbb{P}}, \{\overline{\mathcal{F}}_t\}, \overline{W}^{\eps}\big) $ corresponding to $ \overline{h}^{\eps}$, and $\overline{\mathbb{P}}$-a.s.
$$ \overline{\tt B}^\eps(t)= {\tt B}(\overline{v}^\eps(t)) \quad \forall~t\in [0,T].$$

 Consequently, the following uniform estimates hold:
 \begin{equation}\label{esti:uniform-new-cond-c1}
 \begin{aligned}
 \sup_{\eps >0} \overline{\mathbb{E}} \Bigg[\sup_{s \in [0,T]}  \left \| {\tt B}(\overline{v}^{\eps}(s)) \right \|_{L^2}^{4} + \left(\int_{0}^{T}  \left \|  \overline{v}^{\eps}(s) \right \|_{W_{0}^{1,p}}^{p} ds \right)^2\Bigg] \leq C\,, \\
  \sup_{\eps >0}  \overline{\mathbb{E}}\left[ \left\| {\tt B}(\overline{v}^{\eps})\right\|_{W^{\alpha,2}([0,T];{W^{-1,p'}})}\right] \le C \quad \text{for}~~\alpha\in (0,\frac{1}{2p})\,, \\
   \sup_{\eps >0}  \overline{\mathbb{E}}\left[ \int_{D_T} \left| {\tt A}(\nabla \overline{v}^\eps)\right|^{p^\prime}\,dx\,dt\right]\le C\,.
   \end{aligned}
 \end{equation}
 \begin{lem}\label{lem:conv-limit-1-cond-c1-new}
 We have the following:
 \begin{itemize}
 \item[a)] ${\tt B}(\overline{v}^\eps) \goto \overline{\tt B}^*$ in $L^q(\overline{\Omega}; L^2(D_T))$ for all $1\le q<p$. In particular, ${\tt B}(\overline{v}^\eps) \goto \overline{\tt B}^*$ in $L^2(\overline{\Omega}; L^2(D_T))$.
 \item[b)] ${\tt B}(\overline{v}^\eps) \stackrel{*}{\rightharpoonup}  \overline{\tt B}^*$ in $L_w^2\big(\overline{\Omega}; L^\infty(0,T; L^2)\big)$.
 \item[c)] ${\tt B}(\overline{v}^\eps) \rightharpoonup   \overline{\tt B}^*$ in $L^p(\overline{\Omega}; L^p(0,T;W_0^{1,p}))$.
 \end{itemize}
 \end{lem}
 \begin{proof} Proof of ${\rm a):}$
 In view of \eqref{esti:uniform-new-cond-c1}, the sequence $\{ {\tt B}(\overline{v}^\eps) \}$ is uniformly bounded in $L^p(\overline{\Omega};L^2(D_T))$ and hence
 equi-integrable in $L^q(\overline{\Omega}; L^2(D_T))$ for all $1\le q<p$. Since $\overline{\mathbb{P}}$-a.s., ${\tt B}(\overline{v}^\eps) \goto  \overline{\tt B}^*$, by Vitali convergence theorem we arrive at the assertion ${\rm a)}$.
\vspace{0.15cm}

\noindent{Proof of ${\rm b)}~\&~ {\rm c)}$:}  Following the same proof as done in the proof of \cite[${\rm (iii)}$ of Lemma $3.10$]{Majee2023};  see also \cite[${\rm (v)}$ of Lemma $4.5.7$ ]{Wittbold2019}, one can easily get ${\rm b)}$.  Proof of ${\rm c)}$ follows from   \cite[${\rm (v)}$ of Lemma $3.11$]{Majee2023}. This completes the proof. 
 \end{proof}
 Define the predictable process
 $$ \overline{v}^*(t,x):= {\tt B}^{-1} \overline{\tt B}^*(t,x).$$
 Then for all $t\in [0,T]$ and $\overline{\mathbb{P}}$-a.s., $\overline{v}^*(t)\in L^2$ is well-defined and 
 \begin{align}
 \overline{\mathbb{E}}\left[\sup_{0\le t\le T}\| \overline{v}^*(t)\|_{L^2}^2\right]\le C\,. \label{esti:uni-l2-limit-fun-cond-c1}
 \end{align}
 We show that $\overline{v}^*$ will serve as a weak solution of the Skeleton equation.  To proceed further, we first show certain convergence results of $\overline{v}^\eps$. 
 \begin{lem}\label{lem:conv-limit-2-cond-c1-new}
 The following convergence holds:
 \begin{itemize}
 \item[{\rm i)}] $\overline{v}^\eps \goto \overline{v}^*$ in $L^2(\overline{\Omega}; L^2(D_T))$.
 \item[{\rm ii)}] $ \nabla \overline{v}^\eps  \rightharpoonup \nabla  \overline{v}^*$ in $L^{p}(\overline{\Omega}\times D_T)^d$.
 \item[{\rm iii)}] There exists $\overline{G}^* \in L^{p^\prime}(\overline{\Omega}\times D_T)^d$ such that for any $\phi\in W_0^{1,p}$
 \begin{align*}
 \lim_{\eps\goto 0} \overline{\mathbb{E}}\left[ \int_0^T\left| \int_0^t \langle {\tt A}(\nabla \overline{v}^\eps(s))-\overline{G}^*(s), \nabla \phi \rangle \,ds\right|\,dt\right]=0\,.
 \end{align*}
 \item[{\rm iv)}] For any $\phi\in W_0^{1,p}$
 \begin{align*}
\underset{\eps \rightarrow 0}\lim \,\overline{\mathbb{E}}\left[ \int_{0}^{T} \Big| \int_0^t \left<\sigma(\overline{v}^{\eps}(s))\overline{h}^{\eps}(s) - \sigma(\overline{v}^*(s))h(s) , \phi \right> \, ds\Big|\,dt\right] = 0\,. 
\end{align*}
 \end{itemize}
 \end{lem}
 \begin{proof} Proof of ${\rm i)}:$
 One can use ${\rm a)}$ of Lemma \ref{lem:conv-limit-1-cond-c1-new} and Lipschitz continuity property of ${\tt B}^{-1}: L^2(\overline{\Omega}; L^2(D_T))\goto L^2(\overline{\Omega}; L^2(D_T)) $ together with the definition of $\overline{v}^*$ to arrive at ${\rm i)}$.
 \vspace{0.2cm}
 
 \noindent{Proof of  ${\rm ii)}$ and ${\rm iii)}$:} The assertions ${\rm ii)}$ and ${\rm iii)}$ follow from the 
 uniform boundedness of $\{ \overline{v}^\eps\}$ in $L^p(\overline{\Omega}; L^p(0,T;W_0^{1,p}))$ and uniform boundedness of the family $\{ {\tt A}(\nabla \overline{v}^\eps)\}$ in $L^{p^\prime}(\overline{\Omega}\times D_T)^d$~(cf.~\eqref{esti:uniform-new-cond-c1}).
 \vspace{0.2cm}
 
 \noindent{Proof of  ${\rm iv)}$:}  To prove it, we proceed as follows. 
Like in the estimation of $\mathcal{I}_1(T)$ as in {\bf Step-II} of Subsection \ref{subsec:proof-cond-c2}, we see that
\begin{align*}
& \overline{\mathbb{E}}\left[ \int_0^T\Big|  \int_{0}^{t} \left< \big(\sigma(\overline{v}^\eps) - \sigma(\overline{v}^*(s)) \big) \overline{h}^\eps(s) , \phi \right> \, ds\Big|\,dt\right] \notag \\
& \le CT \left\{ \overline{\mathbb{E}}\Big[ \int_0^T\|\overline{v}^\eps-\overline{v}^*\|_{L^2}^2\,dt\Big]\right\}^\frac{1}{2}\left\{ \overline{\mathbb{E}}\Big[ \int_0^T|\overline{h}^\eps(t)|^2\, dt\Big]\right\}^\frac{1}{2} \goto 0 \quad (\text{by  ${\rm i)}$})\,.
\end{align*}
Since $\overline{\mathbb{P}}$-a.s, $ \overline{h}^\eps   \rightharpoonup  h$ in $L^2([0,T],\mathbb{R})$, one can easily check that $\overline{\mathbb{P}}$-a.s,

\begin{align*}
  \int_0^T \left(\overline{h}^\eps(s)-h(s)\right)\overline{\psi}(s)\,ds  \rightarrow 0\,, 
\end{align*}
where $\overline{\psi} \in L^2([0,T];\R)$ is defined by 
$$\overline{\psi}(t):= \int_D \sigma(\overline{v}^*(t,x))\phi(x)\,dx.$$ 
By using \eqref{esti:uni-l2-limit-fun-cond-c1} and
 the fact that $\overline{h}^\eps, h \in \mathcal{A}_M$, we conclude that, thanks to Vitali convergence theorem
 \begin{align*}
 \lim_{\eps \goto 0} \overline{\mathbb{E}}\left[ \int_0^T\Big|  \int_{0}^{t} \left< \big(\overline{h}^\eps(s) - h(s) \big) \sigma(\overline{v}^*(s)) , \phi \right> \, ds\Big|\,dt\right] =0\,,
 \end{align*}
 and hence the assertion ${\rm iv)}$ follows. 
 \end{proof}

 \begin{thm}
  $\overline{v}^*$ is indeed a solution to the Skeleton equation \eqref{eq:skeleton}. 
 \end{thm}
 \begin{proof}
By using BDG inequality and  \eqref{esti:uniform-new-cond-c1}, one can easily see that
\begin{align}
& \sqrt{\eps}\, \overline{\mathbb{E}} \left[ \int_0^T\Big| \big\langle \int_0^t \sigma(\overline{v}^\eps(s))\,d\overline{W}^\eps(s), \phi \big\rangle\Big|\,dt\right] \le \sqrt{\eps}  \overline{\mathbb{E}} \left[ \int_0^T\left\| \int_0^t \sigma(\overline{v}^\eps(s))\,d\overline{W}^\eps(s)\right\|_{L^2}\|\phi\|_{L^2}\,dt\right] \notag \\
& \le \sqrt{ \eps}\|\phi\|_{L^2}  \overline{\mathbb{E}} \left[ \int_0^T\left( \int_0^T \|\sigma(\overline{v}^\eps(s))\|_{L^2}^2\,ds\right)^\frac{1}{2}\,dt\right] 
 \le C \sqrt{\eps} \|\phi\|_{L^2}  T\, \overline{\mathbb{E}} \left[\sup_{0\le s\le T}\|\overline{v}^\eps(s)\|_{L^2}^2\right] \goto 0 \,. 
 \label{conv:c1-final-3}
\end{align}
Making use of Lemmas \ref{lem:conv-limit-1-cond-c1-new}-\ref{lem:conv-limit-2-cond-c1-new} together with \eqref{conv:c1-final-3}, and following the similar arguments as invoked in ${\rm iii)}$ of Lemma \ref{eq:cgseq1} but its stochastic setting, we get
\begin{itemize}
\item $\overline{\tt B}^*(0)= {\tt B}(u_0)$ in $L^2$.
\item  For all $t\in [0,T]$, ${\tt B}(\overline{v}^\eps(t)) \rightharpoonup  \overline{\tt B}^*(t)$ in $L^2(\overline{\Omega}\times D)$.
\item $\overline{\mathbb{P}}$-a.s., and for all $t\in [0,T]$,  $ \overline{\tt B}^*(t)= {\tt B}(\overline{v}^*)$.
\item $\overline{v}^*$ satisfies the equation: a.s. in $\overline{\Omega}$ and for all $t\in[0,T]$
\begin{align}
{\tt B}(\overline{v}^*(t))= {\tt B}(u_0) +  \int_{0}^{t}  {\rm div}_x  \overline{G}^*(s) ds + \int_{0}^{t} \sigma(\overline{v}^*(s))h(s) ds \quad \text{in $L^2$ }\,. \label{eq:weak-limit-function-cond-c1}
\end{align}
\end{itemize}

\noindent \underline{\bf  Identification of  $\overline{G}^*$:} To show $\overline{v}^*$ as a solution of \eqref{eq:skeleton}, it remains to identify  $\overline{G}^*$ as 
\begin{align}
\overline{G}^* = {\tt A}(\nabla \overline{v}^*) \quad \text{in}\quad  L^{p'}(\overline{\Omega}\times D_T)^{d}\,. \label{equality-G}
\end{align}
 By applying chain-rule on ${\tt B}(\overline{v}^*)$ in  \eqref{eq:weak-limit-function-cond-c1}, and Ito formula on $\|{\tt B}(\overline{v}^\eps)\|_{L^2}^2$, and then subtracting these resulting equations from each other, we have, after taking expectation
\begin{align}
     \overline{\mathbb{E}} \Big[& \left \|  {\tt B}(\overline{v}^{\eps}(T)) \right \|_{L^2}^{2} - \left \| {\tt B}(\overline{v}^*(T)) \right \|_{L^2}^{2} \Big] + 2\overline{\mathbb{E}} \Bigg[\int_{0}^{T} \int_{D}  {\tt b}^\prime(\overline{v}^\eps)(t)) {\tt A}(\nabla  \overline{v}^{\eps}(t)) \cdot  \nabla  \overline{v}^{\eps}(t) \, dx \, dt \Bigg]\notag  \\
    & \leq 2\overline{\mathbb{E}} \Bigg[\int_{0}^{T} \int_{D}   \overline{G}^* \cdot \nabla {\tt B}(\overline{v}^*(t)) \, dx \, dt \Bigg]+ \eps \ \overline{\mathbb{E}} \Bigg[\int_{0}^{T} \left \|  \sigma( \overline{v}^{\eps}) \right \|_{L^2}^{2} dt \Bigg]  \notag \\
    & + 2\overline{\mathbb{E}} \Bigg[\int_{0}^{T} \left< \sigma( v^{\eps})h^{\eps}(t),  {\tt B}(v^{\eps}(t))\right>\, dt - \int_{0}^{T} \left<\sigma(\overline{v}^*(t))h(t), {\tt B}(\overline{v}^*(t)) \right> \, dt \Bigg]\,. \label{eq:conv-limit-c1}
\end{align}
Using the uniform estimate $\overline{\mathbb{E}}\left[ \sup_{0\le t\le T} \| \overline{v}^*(t)\|_{L^2}^4\right]< + \infty$, weak convergence of $\overline{h}^\eps$ to $h$ and Vitaly convergence theorem, we have, similar to \eqref{eq:lem4.4.1}
\begin{align}
    \lim_{\eps \rightarrow 0} \overline{\mathbb{E}} \left [ \int_{0}^{T} \left< \sigma(\overline{v}^*(t)) \big(\overline{h}^{\eps}(t) - h(t) \big) , {\tt B}(\overline{v}^*(t)) \right> \, dt  \right ] = 0. \label{conv:weak-c1-drift-1}
\end{align}
Similar to \eqref{eq:lem4.4.2} and \eqref{eq:lem4.4.3}, one can use ${\rm i)}$, Lemma \ref{lem:conv-limit-1-cond-c1-new} together with uniform estimate for $\overline{v}^*$ and \eqref{esti:uni-l2-limit-fun-cond-c1} to arrive at
\begin{equation} \label{conv:weak-c1-drift-2}
\begin{aligned}
 \lim_{\eps \goto 0} \overline{\mathbb{E}} \left[ \int_0^T \langle \sigma(\overline{v}^\eps) \overline{h}^\eps(t),  {\tt B}(\overline{v}^\eps(t))-{\tt B}(\overline{v}^*(t))\rangle\,dt\right]=0\,, \\
\lim_{\eps \goto 0} \overline{\mathbb{E}} \left[ \int_0^T \langle  (\sigma(\overline{v}^\eps)-\sigma(\overline{v})) \overline{h}^\eps(t), {\tt B}(\overline{v}^*(t))\rangle\,dt\right]=0\,.
\end{aligned}
\end{equation}
Using \eqref{conv:weak-c1-drift-1} and \eqref{conv:weak-c1-drift-2} and passing to the limit as $\eps \goto 0$ in \eqref{eq:conv-limit-c1}, we have
\begin{align}
\limsup_{\eps>0} \overline{\mathbb{E}} \Bigg[\int_{0}^{T} \int_{D}   {\tt b}^\prime(\overline{v}^\eps)(t)) {\tt A}(\nabla  \overline{v}^{\eps}(t)) \cdot  \nabla  \overline{v}^{\eps}(t) \, dx \, dt
 \Bigg] \le \overline{\mathbb{E}} \Bigg[\int_{0}^{T} \int_{D}   \overline{G}^* \cdot \nabla {\tt B}(\overline{v}^*(t)) \, dx \, dt \Bigg]\,. \label{conv:weak-c1-drift-3}
\end{align}
One may follow the calculations as done in \cite[Pages 321-323, proof of Lemma $4.5.22$]{Wittbold2019} and \cite[{\bf Step-iii)}, Subsection $3.5$]{Majee2023} together with \eqref{conv:weak-c1-drift-3} to get  $\overline{G}^*={\tt A}(\nabla \overline{v}^*) \quad \text{in}\quad  L^{p'}(\overline{\Omega}\times D_T)^{d}$. We combine \eqref{eq:weak-limit-function-cond-c1} and \eqref{equality-G} to conclude that $\overline{v}^*$ is indeed a solution of \eqref{eq:skeleton}.
\end{proof}

\noindent\underline{\bf Step C: convergence of $v^\eps$ to $u_h$ in distribution:} In this final step, we show convergence of $v^\eps$ to $u_h$ in  the sense of distribution on $\mathcal{Z}$. Since $u_h$ is a weak solution of the skeleton equation, by uniqueness result,
we see that $\overline{v}^* = u_h$. Since $ \mathcal{L}\big(\overline{v}^{\eps}\big)= \mathcal{L}\big( v^\eps\big)$, to prove  $v^{\eps} \rightarrow u_h $ in distribution as $\eps \rightarrow 0$, it is enough to prove that $\overline{v}^{\eps} \rightarrow \overline{v}^*=u_h $ in distribution on $\mathcal{Z}$ as $\eps \rightarrow 0$---which is provided in the next theorem. 
 \begin{thm}
$\overline{v}^{\eps} \overset{D}{\rightarrow} \overline{v}^* $  in $\mathcal{Z}$.
\end{thm} 
\begin{proof}
Since convergence in probability implies convergence in distribution, by Markov inequality together with \eqref{esti:B-Lipschitz}, the theorem will be proved if we able to show that
\begin{align}
 \overline{\mathbb{E}} \Big[ \left \| {\tt B}(\overline{v}^{\eps}) - {\tt B}(\overline{v}^*) \right \|_{\mathcal{Z}} \Big] \rightarrow 0 \ \ \text{as} \ \eps \rightarrow 0. \label{conv:final-c1}
 \end{align}
 Proof of \eqref{conv:final-c1} is similar to {\bf Step-IV} of Subsection \ref{subsec:proof-cond-c2}  but its stochastic version. Applying Ito formula on the difference equation of $ {\tt B}(\overline{v}^\eps)-{\tt B}(\overline{v}^*)$ and then using integration by parts formula, the assumption \ref{A3} and \eqref{A:monotonocity} to have
 \begin{align}
&  \overline{\mathbb{E}} \Big[ \left \| {\tt B}(v^{\eps}) - \overline{v}  \right \|_{\mathcal{Z}} \Big] \notag \\
 & \leq - 2  \overline{\mathbb{E}} \left[\int_{D_T} \big(  {\tt A}(\grad \overline{v}^\eps) - {\tt A}(\grad \overline{v}^*)\big)
 \cdot   \grad \overline{v}^* {\tt b}^\prime(\overline{v}^\eps) \,dx\,dt \right]
 + 2 \overline{\mathbb{E}} \left[\ \int_{D_T} \big(  {\tt A}(\grad \overline{v}^\eps) - {\tt A}(\grad \overline{v}^*)\big)\cdot \nabla \overline{v}^*{\tt b}^\prime(\overline{v}^*)\,dx\,dt \right]\notag \\
  & \quad +  \overline{\mathbb{E}}\left[ \int_0^T \big\langle \sigma(\overline{v}^\eps) \overline{h}^\eps(t)-
 \sigma(\overline{v}^*) h(t), {\tt B}(v^{\eps}) -{\tt B}( \overline{v}^*) \big\rangle\,dt\right] + C\eps  \overline{\mathbb{E}}\left[\int_0^T \|\sigma(\overline{v}^\eps(t))\|_{L^2}^2\,dt\right] \notag \\
& \qquad + C\sqrt{ \eps} \,\overline{\mathbb{E}}\left[ \sup_{s\in [0,T]}\Big| \int_0^s \big\langle {\tt B}(v^{\eps}(r)) - {\tt B}(\overline{v}^*(r)), \sigma(\overline{v}^\eps(r))\,dW(r)\big\rangle\Big|\right]
\equiv \sum_{i=1}^5 \mathcal{B}_i\,. \label{inq:final-cond-c1}
 \end{align}
  We first consider the  term $\mathcal{B}_1$. Observe that  ${\tt b}^\prime(\overline{v}^\eps) \nabla \overline{v}^* \goto  {\tt b}^{\prime}(\overline{v}^*) \nabla \overline{v}^*$ a.e. in $\overline{\Omega} \times D_T$, and 
$|{\tt b}^\prime(\overline{v}^\eps) \nabla \overline{v}^*|^{p} \in L^1( \overline{\Omega} \times D_T)$.
 Hence an application of Lebesgue convergence theorem, along with the fact that ${\tt A}(\nabla \overline{v}^\eps) \rightharpoonup {\tt A}(\nabla \overline{v}^*) $ in $ L^{p^\prime}( \overline{\Omega} \times D_T)^d$ yields that $\mathcal{B}_1\goto 0$ as $\eps\goto 0$. 
 \vspace{0.15cm}
 
 Since ${\tt A}(\nabla \overline{v}^\eps) \rightharpoonup {\tt A}(\nabla \overline{v}^*) $ in $ L^{p^\prime}( \overline{\Omega} \times D_T)^d$, one can easily deduce that
 $\mathcal{B}_2\goto 0$.  In view of the assumption \ref{A3} and Lemma \ref{lem:est2.1}, it is easy to see that $\mathcal{B}_4 \goto 0$ as $\eps \goto 0$. 
  \vspace{0.15cm}
 
 We now focus on the term $\mathcal{B}_3$. A similar estimation  as of 
$\mathcal{I}_6$ and $\mathcal{I}_7$ gives 
\begin{align*}
\mathcal{B}_3 \le C_M \left(  \overline{\mathbb{E}}\left[ \int_0^T \|\overline{v}^\eps(t)-\overline{v}^*(t)\|_{L^2}^2\,dt\right]\right)^\frac{1}{2} \goto 0 \quad \text{as $\eps \goto 0$~~~(by ${\rm i)}$, Lemma \ref{lem:conv-limit-2-cond-c1-new})}\,.
\end{align*}
We apply Burkh$\ddot{o}$lder-Davis-Gundy inequality, the assumption \ref{A3}, Lemma \ref{lem:est2.1}, \eqref{B:reverse-L2} and 
\eqref{esti:B-Lipschitz} to obtain
\begin{align*}
\mathcal{B}_5  & \le C \sqrt{\eps}\,  \overline{\mathbb{E}}\left[ \left( \int_0^T 
 \|\overline{v}^\eps(t)-\overline{v}^*(t)\|_{L^2}^2  \|\overline{v}^\eps(t)\|_{L^2}^2\,dt \right)^\frac{1}{2}\right] \notag \\
 & \le C \sqrt{\eps}\,  \overline{\mathbb{E}}\left[  \sup_{0\le t\le T} \|\overline{v}^\eps(t)\|_{L^2} \left( \int_0^T
  \|\overline{v}^\eps(t)-\overline{v}^*(t)\|_{L^2}^2 \,dt \right)^\frac{1}{2}\right] \notag \\
  & \le  C \sqrt{\eps}\, \left\{  \overline{\mathbb{E}}\left[  \sup_{0\le t\le T} \| {\tt B}(\overline{v}^\eps(t))\|_{L^2}^2\right]\right\}^\frac{1}{2}  \left\{  \overline{\mathbb{E}}\left[ \int_0^T
  \|\overline{v}^\eps(t)-\overline{v}^*(t)\|_{L^2}^2 \,dt \right]\right\}^\frac{1}{2}\goto 0 \quad \text{as $\eps \goto 0$}\,.
\end{align*}
In view of the above  convergence results, one can pass to the limit in \eqref{inq:final-cond-c1} to arrive at  \eqref{conv:final-c1}
\end{proof}
Hence the condition \ref{C1} is satisfied by $u^\eps$. This finishes the proof of Theorem \ref{thm:main-ldp}.

\section{Existence of Invariant measure}\label{sec:invariant-measure}
This section is devoted to show existence of invariant measure for the associated semigroup $(P_t)_{t\ge 0}$ for strong solution of the doubly nonlinear SPDE \eqref{eq:doubly-nonlinear}. 
\vspace{0.2cm}

We start with the observation that instead of initial condition, one may consider the doubly-nonlinear SPDE \eqref{eq:doubly-nonlinear} 
with initial distribution $\rho_0$, a Borel probability measure on $L^2$, and prove existence of martingale solution. Infact, one can establish the following lemma.

\begin{lem}\label{lem:existence-ini-distribution}
Let the assumptions \ref{A2}-\ref{A4} hold true and the initial distribution satisfies the condition
\begin{align*}
 \int_{L^2} |x|^2 \rho_0(dx) < R \quad \text{for some $R>0$}\,. 
\end{align*}
Then there exists a martingale solution $\big( \hat{\Omega}, \hat{\mathcal{F}}, \hat{\mathbb{P}}, \{\hat{\mathcal{F}}_t\},  
 \hat{W}, \hat{u}\big)$ of \eqref{eq:doubly-nonlinear}, with $\mathcal{L}(\hat{u}(0))=\rho_0$ in $L^2$, satisfying the following estimates
 \begin{equation}\label{esti:weak-ini-distribution}
\begin{aligned}
\hat{ \mathbb{E}}\Big[\sup_{0\le t\le T} \|{\tt B}(\hat{u}(t))\|_{L^2}^2 + \int_0^T \| {\tt B}(\hat{u}(t))\|_{W_0^{1,p}}^p\,{\rm d}t \Big] \le C(R,p)\,,  \\
 \hat{ \mathbb{E}}\Big[ \int_{D_T} \Big( |{\tt A}(\nabla \hat{u}(t))|^{p^\prime} + |\nabla \hat{u}(t)|^p\Big)\,{\rm d}x\,{\rm d}t\Big] \le C(R,p)\,.
\end{aligned}
\end{equation}
\end{lem}
We first recall the definition of uniqueness in law of solutions of \eqref{eq:doubly-nonlinear}
\begin{defi}
We say that the problem \eqref{eq:doubly-nonlinear} has uniqueness in law property if and only if for any two martingale solutions 
 $\big(\Omega^i, \mathcal{F}^i, \mathbb{P}^i, \mathbb{F}^i, W^i, u^i\big)~(i=1,2)$  of \eqref{eq:doubly-nonlinear} with
 ${\rm Law}_{\mathbb{P}^1}(u^1(0))= {\rm Law}_{\mathbb{P}^2}(u^2(0))$ on $L^2$, one has 
\begin{align*}
{\rm Law}_{\mathbb{P}^1}(u^1)= {\rm Law}_{\mathbb{P}^2}(u^2)\quad \text{on}~~~C([0, T];L_w^2])\cap L^p([0, T];W_0^{1,p})\,,
\end{align*}
where ${\rm Law}_{\mathbb{P}^i}(u^i)$ for $i=1,2$ are probability measures on $C([0,T];L_w^2])\cap L^p([0,T];W_0^{1,p}).$
\end{defi}
Since the problem \eqref{eq:doubly-nonlinear} has a path-wise unique martingale solution, a direct application of \cite[Theorems $2~\& ~11$]{Ondrejat2004} yields that the martingale solution is unique in law. Hence we have the following lemma.
\begin{lem}\label{lem:uniqueness-law}
Let the assumptions of \ref{A2}-\ref{A4} hold. Then the martingale solution of  \eqref{eq:doubly-nonlinear}  is unique in law. 
\end{lem}

\begin{rem}[Continuous dependency of martingale solution on initial data]\label{rem:cont-dependence-initial-cond} 
We make a remark regarding the continuous dependency of martingale solution of  \eqref{eq:doubly-nonlinear}  on initial data. 
 Martingale solutions of  \eqref{eq:doubly-nonlinear} continuously depends on initial data. Indeed,  let $\{u_0^n\}$ be a $L^2$-valued sequence that is convergent weakly to  $u_0$ in $L^2$. Let $(\Omega_n,
 \mathcal{F}_n, \mathbb{P}_n, \mathbb{F}_n, W_n, u_n)$ be a martingale solution of \eqref{eq:doubly-nonlinear} with initial datum $u_0^n$.  First observe that, a weak martingale solution to the problem \eqref{eq:doubly-nonlinear} with deterministic initial data $u_0\in L^2$, is also a weak martingale solution to the SPDE \eqref{eq:doubly-nonlinear} with initial law $\rho_0=\delta_{u_0}$. Since $u_0^n $ converges weakly to $u_0\in L^2$, there exists a constant $R_1>0$ such that $\sup_{n\ge 1}\|u_0^n\|_{L^2}^2\le R_1$. Hence by
  Lemma \ref{lem:existence-ini-distribution}, $u_n$ satisfies the uniform estimate \eqref{esti:weak-ini-distribution}.  Thus using the arguments of Subsection \ref{subsec:cond-c1} and \cite[Section $4$]{Wittbold2019}, one arrive at the following:  there exist 
 \begin{itemize}
 \item[i)] a subsequence $\{u_{n_j}\}_{j=1}^\infty$,
 \item[ii)] a stochastic babis $\big(\bar{\Omega}, \bar{\mathcal{F}},
\bar{\mathbb{P}}, \bar{\mathbb{F}}\big)$,
\item[iii)] real-valued Wiener processes $\bar{W}_j$ and $\bar{W}$ with $\bar{W}_j(\bar{\omega})=\bar{W}(\bar{\omega});~j\ge 1$ for $\bar{\omega}\in \bar{\Omega}$,
\item[iv)] $\mathcal{C}_T:= C([0,T]; W^{-1,p^\prime})\cap L^2_w(0,T; W^{1,2}) \cap L^2(0,T; L^2)\cap C([0,T]; L^2_w)$-valued Borel measurable random variables $\bar{u}$ and $\{\bar{u}_j\}_{j=1}^\infty$ such that 
 \end{itemize}
 \begin{align*}
 \mathcal{L}({\tt B}(u_{n_j}))=\mathcal{L}({\tt B}(\bar{u}_j))~~\text{on}~~\mathcal{Z}_T; \quad {\tt B}(\bar{u}_j)\goto {\tt B}(\bar{u})~~\text{in}~~\mathcal{C}_T,~~\bar{\mathbb{P}}\text{-a.s.},
 \end{align*}
and the tuple $\big(\bar{\Omega},
\bar{ \mathcal{F}}, \bar{\mathbb{P}}, \bar{\mathbb{F}}, \bar{W}, \bar{u}\big)$ is a martingale solution of  \eqref{eq:doubly-nonlinear} with initial distribution $\delta_{u_0}$, on the interval $[0,T]$. 
\end{rem}

Recall the family  $\{P_t\}$ as defined in \eqref{defi:semi-group} i.e., 
\begin{align*}
(P_t \phi)(v):=\mathbb{E}\big[\phi({\tt B}(u(t,v)))\big], \quad \phi \in\mathcal{B}(L^2)\,, 
\end{align*}
where $u(t,v)$ is the path-wise unique solution of   \eqref{eq:doubly-nonlinear} with fixed initial data $u_0=v\in L^2$.  Note that $P_t\phi$ is bounded.  Moreover, thanks to Lemma \ref{lem:uniqueness-law}, one can use \cite[Corollary $23$]{Ondrejat2005} to get that 
$P_t \phi \in \mathcal{B}(L^2)$ and $\{P_t\}_{t\ge 0}$ is a semigroup on $\mathcal{B}(L^2)$. Furthermore, since the unique strong solution of  \eqref{eq:doubly-nonlinear} has path a.e. in $C([0,T;L^2])$, it is also a Markov semigroup; see \cite[Theorem 27]{Ondrejat2005}.

\begin{lem}
Markov semigroup $\{P_t\}$ is a Feller semigroup i.e., $P_t$ maps $C_b(L^2)$ into itself. 
\end{lem}
\begin{proof}
Let $\{v_n\}\subset L^2$ such that $v_n\goto v$ in $L^2$ and $\phi \in C_b(L^2)$. Let $u(t,v_n)$ denotes the pathwise unique solution of \eqref{eq:doubly-nonlinear} with initial condition $v_n$. Then proof of path-wise uniqueness and continuity of $\phi$ yields that for any $t\in [0,T)$
$\phi({\tt B}(u(t, v_n)))\goto \phi({\tt B}(u(t,v)))$ a.s. in $\Omega$. Since $\phi$ is bounded, by dominated convergence theorem, we have
$$\mathbb{E}\big[\phi({\tt B}(u(t, v_n)))\big]\goto \mathbb{E}\big[\phi({\tt B}(u(t, v)))\big]~~\text{ i.e.,}~~  (P_t \phi)(v_n) \goto (P_t \phi)(v).$$
 Since $\phi$ is bounded, we see that 
$|(P_t \phi)(z)|\le \|\phi\|_{\infty} < + \infty$ for all $z\in L^2$. Thus, $P_t \phi \in C_b(L^2)$. This completes the proof. 
\end{proof}
\begin{defi} \label{defi:se-weak-bounded}
A function $\phi \in SC(L_w^2)$ if and only if $\phi$ is sequentially continuous with respect to weak topology on $L^2$. By  $SC_b(L_w^2)$, we denote set of all bounded functions in $ SC(L_w^2)$. 
\end{defi}
The following inclusion holds; see \cite{Ferrario2019,Seidler1999}:
\begin{align*}
C_b(L_w^2)\subset SC_b(L_w^2)\subset C_b(L^2)\subset \mathcal{B}_b(L^2)\,.
\end{align*}
\begin{defi}
The family $\{P_t\}$ is said to be sequentially weakly Feller if and only if  $P_t$ maps  $SC_b(L_w^2)$ into itself. 
\end{defi}
To prove existence of an invariant measure, we will use the following result of Maslowski and Seidler \cite{Seidler1999}, a modification of the Krylov-Bogoliubov technique \cite{Krylov1937}. 

\begin{thm}\label{thm:M-S}
Assume that 
\begin{itemize}
\item[${\rm i^\prime)}$] $\{P_t\}$ is sequentially weakly Feller in $L^2$,
\item[${\rm ii^\prime)}$] for any $\eps>0$, there exists $R\equiv R(\eps)>0$ such that 
\begin{align*}
\sup_{T\ge 1}\frac{1}{T} \int_0^T \mathbb{P}\Big\{ \|{\tt B}(u(t,u_0))\|_{L^2}>R\Big\}\,dt < \eps.
\end{align*}
\end{itemize}
Then there exists at least one invariant measure for the strong solution of  \eqref{eq:doubly-nonlinear}.
\end{thm}

\subsection{Proof of Theorem \ref{thm:existence-invariant-measure}}
To prove Theorem \ref{thm:existence-invariant-measure}, we show two conditions of Theorem \ref{thm:M-S} under the assumptions 
\ref{A2}-\ref{A4} and \eqref{cond:extra-sigma}. 
\vspace{0.2cm}

\noindent{ \bf Proof of ${\rm i^\prime)}$}. Let $0<t\le T$, $v\in L^2$ and $\phi\in SC_b(L_w^2)$ be fixed. Let $\{v_n\}$ be a $L^2$-valued sequence such that $v_n$ converges $v$ weakly in $L^2$. Let $u_n$ be a strong solution of  \eqref{eq:doubly-nonlinear}, defined on the stochastic basis $(\Omega, \mathcal{F}, \mathbb{F}, \mathbb{P})$, on $[0,T]$ with initial datum $v_n$.  Then by Remark \ref{rem:cont-dependence-initial-cond}, there exist  a subsequence $\{u_{n_j}\}_{j=1}^\infty$,  a stochastic babis $\big(\bar{\Omega}, \bar{\mathcal{F}},
\bar{\mathbb{P}}, \bar{\mathbb{F}}\big)$, real-valued Wiener processes $\bar{W}_j$ and $\bar{W}$ with $\bar{W}_j(\bar{\omega})=\bar{W}(\bar{\omega});~j\ge 1$ for  all $\bar{\omega}\in \bar{\Omega}$, and $\mathcal{C}_T$-valued Borel measurable random variables $\bar{u}$ and $\{\bar{u}_j\}_{j=1}^\infty$ such that 
 \begin{align}
 \mathcal{L}({\tt B}(u_{n_j}))=\mathcal{L}({\tt B}(\bar{u}_j))~~\text{on}~~\mathcal{C}_T; \quad {\tt B}(\bar{u}_j)\goto {\tt B}(\bar{u})~~\text{in}~~\mathcal{C}_T,~~\bar{\mathbb{P}}\text{-a.s.}. \label{con-con-con}
 \end{align}
Moreover, the tuple $\big(\bar{\Omega},
\bar{ \mathcal{F}}, \bar{\mathbb{P}}, \bar{\mathbb{F}}, \bar{W}, \bar{u}\big)$ is a martingale solution of   \eqref{eq:doubly-nonlinear} with initial data $u_0$, on the interval $[0,T]$. Since $\phi\in SC_b(L_w^2)$, by \eqref{con-con-con}, $\bar{\mathbb{P}}$-a.s., $\phi({\tt B}(\bar{u}_j(t))) \goto \phi({\tt B}(\bar{u}(t)))$ for any fixed $t$. Moreover, by dominated convergence theorem
\begin{align}
\lim_{j\goto \infty} \bar{\mathbb{E}}\big[ \phi({\tt B}(\bar{u}_j(t))) \big]=  \bar{\mathbb{E}}\big[ \phi({\tt B}(\bar{u}(t))) \big]\,. \label{con-con-con-1}
\end{align}
Since $\mathcal{L}({\tt B}(\bar{u}_j))= \mathcal{L}({\tt B}(u_{n_j}))$ on $\mathcal{C}_T$, one has $\mathcal{L}({\tt B}(\bar{u}_j))= \mathcal{L}({\tt B}(u_{n_j}))$ on $L_w^2$ and therefore
\begin{align}
 \bar{\mathbb{E}}\big[ \phi({\tt B}(\bar{u}_j(t))) \big]=  \mathbb{E}\big[ \phi({\tt B}(\bar{u}_{n_j}(t))) \big]:= (P_t\phi)(v_{n_j})\,. \label{con-con-con-2}
\end{align}
Let $u$ be a strong solution of   \eqref{eq:doubly-nonlinear} with initial data $u_0$. Since $\bar{u}$ is also a martingale solution of   \eqref{eq:doubly-nonlinear} with initial data $u_0$, by  Lemma \ref{lem:uniqueness-law}, we see that $\mathcal{L}(u)=\mathcal{L}(\bar{u})$ on $\mathcal{C}_T$. Since ${\tt B}: L^2 \goto L^2$ is Lipschitz continuous, one has 
\begin{align}
(P_t \phi)(v):= \mathbb{E}\big[ \phi({\tt B}(u(t)))\big]= \bar{\mathbb{E}}\big[ \phi({\tt B}(\bar{u}(t))) \big]\,. \label{con-con-con-3}
\end{align}
In view of \eqref{con-con-con-1}, \eqref{con-con-con-2} and \eqref{con-con-con-3}, we get
\begin{align*}
\lim_{j\goto \infty} (P_t\phi)(v_{n_j})=\lim_{j\goto \infty} \bar{\mathbb{E}}\big[ \phi({\tt B}(\bar{u}_j(t))) \big]=  \bar{\mathbb{E}}\big[ \phi({\tt B}(\bar{u}(t))) \big]=\mathbb{E}\big[ \phi({\tt B}(u(t)))\big]=(P_t \phi)(v)\,.
\end{align*}
Moreover, by uniqueness, the whole sequence converges. This completes the proof of the assertion ${\rm i^\prime)}$. 
\vspace{.2cm}

\noindent{ \bf Proof of ${\rm ii^\prime)}$}. An application of It\^{o}-formula to $x\mapsto \frac{1}{2}\|x\|_{L^2}^2$ on ${\tt B}(u)$ and  integration by parts formula yields, after taking expectation
\begin{align*}
\frac{1}{2} \mathbb{E}\big[\|{\tt B}(u(t))\|_{L^2}^2\big] - \frac{1}{2}\|{\tt B}(u_0)\|_{L^2}^2 + \mathbb{E}\Big[ \int_0^t \big\langle {\tt A}(\nabla u(s)), \nabla {\tt B}(u(s))\big\rangle\,ds\Big] = \frac{1}{2} \mathbb{E}\Big[ \int_0^t \|\sigma(u(s))\|_{L^2}^2\,ds\Big]\,.
\end{align*}
In view of the assumption ${\rm ii)}$, \ref{A2} and the boundedness property of ${\tt b}^\prime$, we have 
\begin{align*}
C_1C_3 \mathbb{E}\Big[\int_0^t \|\nabla u(s)\|_{L^p}^p\,ds \Big] 
\le  \mathbb{E}\Big[ \int_0^t \big\langle {\tt A}( \nabla u(s)), \nabla {\tt B}(u(s)) \big\rangle\,ds\Big] + C_4 t \|K_1\|_{L^1}\,.
\end{align*}
Thus, we obtain
\begin{align*}
\mathbb{E}\big[\|{\tt B}(u(t))\|_{L^2}^2\big] + \mathbb{E}\Big[\int_0^t \big( 2C_1C_3 \|\nabla u(s)\|_{L^p}^p- \|\sigma(u(s))\|_{L^2}^2\big)\,ds \Big]
\le  2t C_4 \|K_1\|_{L^1} + \|{\tt B}(u_0)\|_{L^2}^2 \,.
\end{align*}
Thanks to  the condition \eqref{cond:extra-sigma}, we have,  for any $T>0$
\begin{align*}
\frac{1}{T} \int_0^T \mathbb{E}\big[\|u(t)\|_{L^2}^p\big]\,ds \le \frac{1}{\delta} \Big( 2\|K_1\|_{L^1} + \frac{1}{T}\|u_0\|_{L^2}^2\Big)\,.
\end{align*}
By Markov inequality, we get, for any $T>0$ and $R>0$
\begin{align*}
& \frac{1}{T}\int_0^T \mathbb{P}\big\{ \|u(t)\|_{L^2} >R\big\}\,dt \le \frac{1}{R^p}\, \frac{1}{T}\int_0^T \mathbb{E}\big[ \| u(t)\|_{L^2}^p\big]\,dt \notag \\
& \le  \frac{1}{ R^p \delta} \Big( 2\|K_1\|_{L^1} + \frac{1}{T}\|u_0\|_{L^2}^2\Big)\,.
\end{align*}
To complete the proof of Theorem \ref{thm:existence-invariant-measure}, one needs to take  sufficiently large $R\equiv R(\eps, \|u_0\|_{L^2}, \|K_1\|_{L^1}, \delta) >0$. 


\subsection{Declarations} The author would like to make the following declaration statements.

\begin{itemize}
\item {\bf Funding:} The author would like to acknowledge the financial support by Department of Science and Technology, Govt. of India-the INSPIRE fellowship~(IFA18-MA119).
\item {\bf Data availability:}  Data sharing is not applicable to this article as no datasets were generated or analyzed during the current study.
\item{\bf Conflict of interest:} The author has not disclosed any competing interests.

\end{itemize}

\end{document}